\documentclass[a4j,reqno]{amsart}
\usepackage{amsmath,amsthm}
\usepackage{amssymb}
\usepackage{amsthm}
\usepackage{amscd}
\usepackage{stmaryrd}
\usepackage[dvipdfmx]{graphicx}
\usepackage[all]{xy}
\usepackage{wrapfig}
\usepackage{color}
\usepackage{latexsym}
\usepackage{enumerate}
\usepackage{cases}
\usepackage[top=30mm,bottom=30mm,left=30mm,right=30mm]{geometry}
\usepackage{graphics}
\usepackage[dvipdfmx]{graphicx}
\usepackage{tikz}
\usetikzlibrary{patterns}
\usepgflibrary{arrows.meta}
\usepackage{pxpgfmark}
\colorlet{NewColor}{red!30!gray}
\usetikzlibrary{intersections, calc,decorations.markings}
\numberwithin{equation}{section}
\theoremstyle{definition}
\newtheorem{example}{Example}[section]
\newtheorem{definition}[example]{Definition}

\theoremstyle{plain}
\newtheorem{lemma}[example]{Lemma}
\newtheorem{theorem}[example]{Theorem}
\newtheorem{proposition}[example]{Proposition}
\newtheorem{corollary}[example]{Corollary}
\newtheorem{assumption}{Assumption}

\DeclareMathOperator{\res}{res}
\DeclareMathOperator{\refin}{\reflectbox{$\in$}}
\DeclareMathOperator{\cross}{cross}

\DeclareMathOperator{\add}{\mathsf{add}}
\DeclareMathOperator{\proj}{\mathsf{proj}}
\DeclareMathOperator{\inj}{\mathsf{inj}}
\DeclareMathOperator{\Fac}{\mathsf{Fac}}
\DeclareMathOperator{\Sub}{\mathsf{Sub}}
\DeclareMathOperator{\Mod}{\mathsf{mod}}
\DeclareMathOperator{\thick}{\mathsf{thick}}
\DeclareMathOperator{\Top}{top}

\DeclareMathOperator{\Trop}{Trop}

\DeclareMathOperator{\Int}{\mathsf{Int}}

\DeclareMathOperator{\Asym}{\mathbb{A}_{\rm sym}}
\DeclareMathOperator{\Acom}{\mathbb{A}_{\rm com}}

\DeclareMathOperator{\oQ}{\overline{\mbox{$Q$}}}
\DeclareMathOperator{\oW}{\overline{\mbox{$W$}}}
\DeclareMathOperator{\ox}{\overline{\mbox{$x$}}}
\DeclareMathOperator{\oy}{\overline{\mbox{$y$}}}
\DeclareMathOperator{\oox}{\overline{\overline{\mbox{$x$}}}}
\DeclareMathOperator{\ooy}{\overline{\overline{\mbox{$y$}}}}

\DeclareMathOperator{\g}{\gamma}
\DeclareMathOperator{\de}{\delta}
\DeclareMathOperator{\w}{\omega}
\DeclareMathOperator{\z}{\zeta}

\DeclareMathOperator{\bA}{\mathbb{A}}
\DeclareMathOperator{\bB}{\mathbb{B}}
\DeclareMathOperator{\bI}{\mathbb{I}}
\DeclareMathOperator{\bP}{\mathbb{P}}
\DeclareMathOperator{\bS}{\mathbb{S}}
\DeclareMathOperator{\bZ}{\mathbb{Z}}

\DeclareMathOperator{\cA}{\mathcal{A}}

\DeclareMathOperator{\si}{\mathsf{i}}
\DeclareMathOperator{\sr}{\mathsf{r}}

\DeclareMathOperator{\oB}{\overline{B}}
\DeclareMathOperator{\oP}{\overline{P}}

\DeclareMathOperator{\wB}{\widetilde{B}}
\DeclareMathOperator{\wG}{\widetilde{G}}
\DeclareMathOperator{\wP}{\widetilde{P}}
\DeclareMathOperator{\wz}{\widetilde{\z}}

\DeclareMathOperator{\redb}{{\color{blue}\bullet}}
\DeclareMathOperator{\bluec}{{\color{blue}\circ}}

\tikzset{->-/.style={decoration={markings,mark=at position .5 with {\arrow{>}}},postaction={decorate}}}

\title{Combinatorial cluster expansion formulas from triangulated surfaces}
\author{Toshiya Yurikusa}
\address{T. Yurikusa: Graduate School of Mathematics, Nagoya University, Chikusa-ku, Nagoya, 464-8602 Japan}
\email{m15049q@math.nagoya-u.ac.jp}

\begin{document}

\maketitle

\begin{abstract}
 We give a cluster expansion formula for cluster algebras with principal coefficients defined from triangulated surfaces in terms of perfect matchings of angles. Our formula simplifies the cluster expansion formula given by Musiker, Schiffler and Williams in terms of perfect matchings of snake graphs. A key point of our proof is to give a bijection between perfect matchings of angles in some triangulated polygon and perfect matchings of the corresponding snake graph. Moreover, they also correspond bijectively with perfect matchings of the corresponding bipartite graph and minimal cuts of the corresponding quiver with potential.
\end{abstract}


\section{Introduction}\label{intro}

 Cluster algebras, introduced by Fomin and Zelevinsky in $2002$ \cite{FZ1}, are commutative algebras with a distinguished set of generators, which are called cluster variables. Their original motivation was coming from an algebraic framework for total positivity and canonical bases in Lie Theory. In recent years, it has interacted with various subjects in mathematics, for example, quiver representations, Calabi-Yau categories, Poisson geometry, Teichm\"{u}ller spaces, exact WKB analysis, etc.

 In a cluster algebra with principal coefficients, by Laurent phenomenon, any cluster variable is expressed by a Laurent polynomial of the initial cluster variables $(x_1,\ldots,x_N)$ and coefficients $(y_1,\ldots,y_N)$
 \begin{equation*}
  x=\frac{f(x_1, \ldots,x_N,y_1,\ldots,y_N)}{x_1^{d_1} \cdots x_N^{d_N}},
 \end{equation*}
where $f(x_1, \ldots, x_N,y_1,\ldots,y_N) \in \bZ[x_1,\ldots, x_N,y_1,\ldots,y_N]$ and $d_i \in \bZ_{\ge 0}$ \cite{FZ1,FZ2}. An explicit formula for the Laurent polynomials of cluster variables is called a {\it cluster expansion formula}.

 We study cluster algebras defined from triangulated surfaces that are developed in \cite{FoG1,FoG2,FST,FT,GSV}. In this case, Musiker, Schiffler and Williams gave a cluster expansion formula in terms of perfect matchings of snake graphs. Using it, they proved the positivity conjecture \cite{MSW1} and constructed two bases \cite{MSW2} for these cluster algebras. The first aim of this paper is to give a cluster expansion formula for these cluster algebras in terms of perfect matchings of angles (Theorem \ref{main}). This simplifies their formula as we will discuss later. The second aim of this paper is to give bijections between several different combinatorial objects containing perfect matchings of snake graphs (Theorem \ref{bijthm}).

 This paper is organized as follows. In the rest of this section, we give our results and some examples. For simplicity, we first specialize Theorem \ref{main} to the coefficient-free case, that is $y_i = 1$ for all $i$ (Theorem \ref{main'}). Using Theorem \ref{main}, we also study {\it $f$-vectors} of cluster variables. In Section \ref{Preliminary}, we recall basic definitions and facts on cluster algebras, triangulated surfaces and the cluster expansion formula of Musiker-Schiffler-Williams. We prove Theorem \ref{main} and a part of Theorem \ref{bijthm} simultaneously in Section \ref{proofofmain}. We prove our results for the corresponding bipartite graphs in Section \ref{bipartite} and study minimal cuts of the corresponding quivers with potential in Sections \ref{pmmc}. Finally, some elements in $\cA(T)$ correspond to {\it essential loops} in $T$ (see Section \ref{essentialloops} for details). In the case of a marked surface without punctures, it is known that these elements and cluster variables form a base of $\cA(T)$ (see Theorem \ref{base}). We give the formula for these elements in terms of good perfect matchings of angles in Theorem \ref{loopformula}.

\subsection{Our results in the coefficient-free case}\label{cfcase}

 Let $(S,M)$ be a marked surface and $T$ a tagged triangulation of $(S,M)$. Let $\cA(T)$ be the cluster algebra with principal coefficients defined from $T$ (see Subsection \ref{clalgfromtri}). Then there is a bijection between cluster variables in $\cA(T)$ and tagged arcs of $(S,M)$, which are obtained from ordinary arcs by tagging their endpoints {\it plain} or {\it notched} (see Theorem \ref{bij}). We represent tagged arcs as follows:
\[
\begin{tikzpicture}
 \coordinate (0) at (0,0) node[left]{plain};
 \coordinate (1) at (1,0);   \fill (1) circle (0.7mm);
 \draw (0) to (1);
\end{tikzpicture}
\hspace{7mm}
\begin{tikzpicture}
 \coordinate (0) at (0,0) node[left]{notched};
 \coordinate (1) at (1,0);   \fill (1) circle (0.7mm);
 \draw (0) to node[pos=0.8]{\rotatebox{90}{\footnotesize $\bowtie$}} (1);
\end{tikzpicture}
\]

\noindent For simplicity, in this paper, we assume that if $(S,M)$ is a closed surface with exactly one puncture, all tagged arcs are plain arcs. For a tagged arc $\de$, we denote by $x_{\de}$ the corresponding cluster variable in $\cA(T)$. We index the tagged arcs of $T$ by $[1,N] := \{1,\ldots,N\}$. In particular, $x_i$ (resp., $y_i$) is the corresponding initial cluster variable (resp., coefficient) in $\cA(T)$ for $i \in [1,N]$.

\begin{definition}
 We call a tagged arc $\de$
\begin{itemize}
 \item a {\it plain arc} if its both ends are tagged plain,
 \item a {\it $1$-notched arc} if an end of $\de$ is tagged plain and the other end is tagged notched,
 \item a {\it $2$-notched arc} if its both ends are tagged notched.
\end{itemize}
\end{definition}

 To give cluster expansion formulas, by changing tags, we can make the following assumption (see Proposition \ref{MSW1315}).
\begin{assumption}\label{ass1}
 The initial tagged triangulation $T$ consists of plain arcs and $1$-notched arcs, with at most one $1$-notched arc incident to each puncture.
\end{assumption}
\noindent In this case, for each $1$-notched arc of $T$, the corresponding plain arc is also in $T$. Then there is a unique ideal triangulation $T^0$ obtained from $T$ by replacing every $1$-notched arc with the corresponding loop cutting out a once-punctured monogon and by forgetting tags.

 For a tagged arc $\de$ of $(S,M)$, we denote by $\overline{\de}$ the plain arc corresponding to $\de$. Now, we only consider the case of $\g := \overline{\de} \notin T$. Let $p$ and $q$ be the endpoints of $\g$. Let $\g^{(p)}$ be the $1$-notched arc obtained from $\g$ by tagging its end $p$ notched. Similarly, we define the $2$-notched arc $\g^{(pq)}$ with both ends tagged notched:
\[
\begin{tikzpicture}
 \coordinate (0) at (0,0);   \fill (0) circle (0.7mm);
 \coordinate (1) at (1.5,0);   \fill (1) circle (0.7mm);
 \coordinate (2) at (0.75,-0.27) node[below] at (2) {$\g$};
 \node at (0.8,0.05){};
 \node at (0.8,-0.08){};
 \draw (0) to (1);
\end{tikzpicture}
\hspace{7mm}
\begin{tikzpicture}
 \coordinate (0) at (0,0);   \fill (0) circle (0.7mm);
 \coordinate (1) at (1.5,0) node[right] at (1) {$p$};   \fill (1) circle (0.7mm);
 \coordinate (2) at (0.75,-0.1) node[below] at (2) {$\g^{(p)}$};
 \draw (0) to node[pos=0.9]{\rotatebox{90}{\footnotesize $\bowtie$}} (1);
\end{tikzpicture}
\hspace{7mm}
\begin{tikzpicture}
 \coordinate (0) at (0,0) node[left] at (0) {$q$};   \fill (0) circle (0.7mm);
 \coordinate (1) at (1.5,0) node[right] at (1) {$p$};   \fill (1) circle (0.7mm);
 \coordinate (2) at (0.75,-0.1) node[below] at (2) {$\g^{(pq)}$};
 \draw (0) to node[pos=0.1]{\rotatebox{90}{\footnotesize $\bowtie$}} node[pos=0.9]{\rotatebox{90}{\footnotesize $\bowtie$}} (1);
\end{tikzpicture}
\]
In particular, $\de=\g$, $\g^{(p)}$ or $\g^{(pq)}$. By changing tags, we can make the following assumption (see Proposition \ref{MSW1315}).
\begin{assumption}\label{ass2}
 If $\de = \g^{(p)}$ (resp., $\de = \g^{(pq)}$), there is no $1$-notched arc incident to $p$ (resp., $p$ or $q$) in $T$.
\end{assumption}
 Our cluster expansion formula for $x_{\g}$ (resp., $x_{\g^{(p)}}$, $x_{\g^{(pq)}}$) comes down to type $A$ (resp., $D$, $\widetilde{D}$) corresponding to polygons with no punctures (resp., one puncture, two punctures). We construct a triangulated polygon $T_{\de}$ associated with $\de$ as follows.

 Let $\tau_1,\ldots,\tau_n$ be the arcs of $T^0$ crossing $\g$ in order of occurrence along $\g$ (we can have $\tau_{i}=\tau_j$ even if $i \neq j$). Hence $\g$ crosses $n+1$ triangles $\triangle_0,\ldots,\triangle_n$, in this order. Suppose first that none of these triangles is self-folded. Then for $i \in [0,n]$, let $\triangle_{\g,i}$ be a copy of the oriented triangle $\triangle_i$, hence $\triangle_{\g,i}$ contains the sides $\tau_i$ and $\tau_{i+1}$ ($\tau_1$ only if $i=0$, and $\tau_n$ only if $i=n$). Then $T_{\g}$ is the triangulation of an $(n+3)$-gon obtained by gluing these triangles along the edges $\tau_i$. Similarly, we construct $T_{\g^{(p)}}$ (resp., $T_{\g^{(pq)}}$) by adjoining to $T_{\g}$ copies of all triangles incident to $p$ (resp., $p$ and $q$) if none of them is self-folded. See Figure \ref{Tdelta}. If $\g$ crosses self-folded triangles or there are self-folded triangles incident to $p$ or $q$, we adapt the construction using the local transformations of Figure \ref{replace1notched}. Note that, by Assumption \ref{ass2}, it is not necessary to consider the case, where the end of $\de$ in the middle of Figure \ref{replace1notched} is tagged notched.

\begin{figure}[h]
   \caption{Triangulated polygon $T_{\de}$ for each tagged arc $\de$}
   \label{Tdelta}
\[
\begin{tikzpicture}
 \coordinate (l) at (0,0);
 \coordinate (lu) at (0.8,1);
 \coordinate (ru) at (3.1,1);
 \coordinate (ld) at (0.8,-1);
 \coordinate (rd) at (3.1,-1);
 \coordinate (r) at (3.9,0);
 \coordinate (u) at (2.1,1);
 \coordinate (d) at (1.8,-1);
 \node at (1.75,0.4) {$\cdots$};
 \node at (1.95,-1.4) {$T_{\g}$ (type $A$)};
 \draw (l) to (lu);
 \draw (l) to (ld);
 \draw (lu) to node[fill=white,inner sep=1,pos=0.4]{\scriptsize $\tau_1$} (ld);
 \draw (lu) to (ru);
 \draw (ld) to (rd);
 \draw (ru) to (r);
 \draw (rd) to (r);
 \draw (ru) to node[fill=white,inner sep=1,pos=0.4]{\scriptsize $\tau_n$} (rd);
 \draw (u) to node[fill=white,inner sep=1,pos=0.4]{\scriptsize $\tau_{n-1}$} (rd);
 \draw (lu) to node[fill=white,inner sep=1,pos=0.4]{\scriptsize $\tau_2$} (d);
 \draw[blue] (l) to node[fill=white,inner sep=1]{\scriptsize $\g$} (r);
\end{tikzpicture}
     \hspace{4mm}
\begin{tikzpicture}
 \coordinate (l) at (0,0);
 \coordinate (lu) at (0.8,1);
 \coordinate (ld) at (0.8,-1);
 \coordinate (ru) at (3.1,1);
 \coordinate (rd) at (3.1,-1);
 \coordinate (rru) at (4.7,1);
 \coordinate (rrd) at (4.7,-1);
 \coordinate (r) at (3.9,0);
 \coordinate (u) at (2.1,1);
 \coordinate (d) at (1.8,-1);
 \node at (1.75,0.4) {$\cdots$};
 \node at (2.35,-1.4) {$T_{\g^{(p)}}$ (type $D$)};
 \node at (4.05,0.15) {$p$};
 \draw (l) to (lu);
 \draw (l) to (ld);
 \draw (lu) to node[fill=white,inner sep=1,pos=0.4]{\scriptsize $\tau_1$} (ld);
 \draw (lu) to (ru);
 \draw (ld) to (rd);
 \draw (ru) to (r);
 \draw (rd) to (r);
 \draw (ru) to (rru);
 \draw (rru) to (rrd);
 \draw (rrd) to (rd);
 \draw (rrd) to (r);
 \draw (ru) to node[fill=white,inner sep=1,pos=0.4]{\scriptsize $\tau_n$} (rd);
 \draw (u) to node[fill=white,inner sep=1,pos=0.4]{\scriptsize $\tau_{n-1}$} (rd);
 \draw (lu) to node[fill=white,inner sep=1,pos=0.4]{\scriptsize $\tau_2$} (d);
 \draw[blue] (l) to node[fill=white,inner sep=1]{\scriptsize $\g^{(p)}$} node[pos=0.95]{\rotatebox{90}{\footnotesize $\bowtie$}} (r);
 \draw[loosely dotted] (4.4,-0.2) arc (-30:110:0.5);
\end{tikzpicture}
     \hspace{4mm}
\begin{tikzpicture}
 \coordinate (l) at (0,0);
 \coordinate (lu) at (0.8,1);
 \coordinate (ld) at (0.8,-1);
 \coordinate (ru) at (3.1,1);
 \coordinate (rd) at (3.1,-1);
 \coordinate (rru) at (4.7,1);
 \coordinate (rrd) at (4.7,-1);
 \coordinate (llu) at (-0.8,1);
 \coordinate (lld) at (-0.8,-1);
 \coordinate (r) at (3.9,0);
 \coordinate (u) at (2.1,1);
 \coordinate (d) at (1.8,-1);
 \node at (1.75,0.4) {$\cdots$};
 \node at (1.95,-1.4) {$T_{\g^{(pq)}}$ (type $\widetilde{D}$)};
 \node at (4.05,0.15) {$p$};
 \node at (-0.15,-0.15) {$q$};
 \draw (l) to (lu);
 \draw (l) to (ld);
 \draw (lu) to node[fill=white,inner sep=1,pos=0.4]{\scriptsize $\tau_1$} (ld);
 \draw (lu) to (ru);
 \draw (ld) to (rd);
 \draw (ru) to (r);
 \draw (rd) to (r);
 \draw (ru) to (rru);
 \draw (rru) to (rrd);
 \draw (rrd) to (rd);
 \draw (rrd) to (r);
 \draw (lu) to (llu);
 \draw (llu) to (lld);
 \draw (lld) to (ld);
 \draw (llu) to (l);
 \draw (ru) to node[fill=white,inner sep=1,pos=0.4]{\scriptsize $\tau_n$} (rd);
 \draw (u) to node[fill=white,inner sep=1,pos=0.4]{\scriptsize $\tau_{n-1}$} (rd);
 \draw (lu) to node[fill=white,inner sep=1,pos=0.4]{\scriptsize $\tau_2$} (d);
 \draw[blue] (l) to node[fill=white,inner sep=1]{\scriptsize $\g^{(pq)}$} node[pos=0.05]{\rotatebox{90}{\footnotesize $\bowtie$}} node[pos=0.95]{\rotatebox{90}{\footnotesize $\bowtie$}} (r);
 \draw[loosely dotted] (4.4,-0.2) arc (-30:110:0.5);
 \draw[loosely dotted] (-0.5,0.2) arc (150:290:0.5);
\end{tikzpicture}
\]
\end{figure}
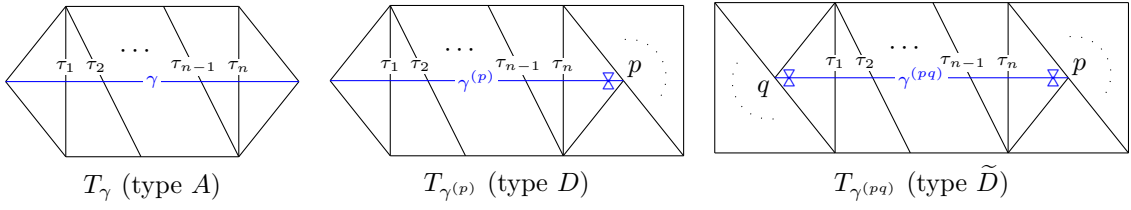
\begin{figure}[h]
   \caption{Replacing self-folded triangles}
   \label{replace1notched}
\begin{tikzpicture}[baseline=7mm]
 \coordinate (0) at (0,0);   \fill (0) circle (0.7mm);
 \coordinate (p) at (0,1);   \fill (p) circle (0.7mm);
 \draw (0) to node[fill=white,inner sep=1,pos=0.7]{\scriptsize $\tau_i$} (p);
 \draw (0) .. controls (-1.5,2) and (1.5,2) .. node[fill=white,inner sep=1]{\scriptsize $\tau_{i-1}=\tau_{i+1}$} (0);
 \draw [blue] (-0.7,0.4) node [below] {\scriptsize $\de$} to (0.7,0.4);
\end{tikzpicture}
 \hspace{-8mm}$\rightarrow$
\begin{tikzpicture}[baseline=7mm]
 \coordinate (0) at (0,0);
 \coordinate (l) at (-1,1.5);
 \coordinate (p) at (0,1.5);
 \coordinate (r) at (1,1.5);
 \draw (0) to node[fill=white,inner sep=1,pos=0.7]{\scriptsize $\tau_{i-1}$} (l);
 \draw (p) to node[fill=white,inner sep=1]{\scriptsize $\tau_{i}$} (l);
 \draw (p) to node[fill=white,inner sep=1,pos=0.3]{\scriptsize $\tau_{i}$} (0);
 \draw (p) to node[fill=white,inner sep=1]{\scriptsize $\tau_{i}$} (r);
 \draw (0) to node[fill=white,inner sep=1,pos=0.7]{\scriptsize $\tau_{i+1}$} (r);
 \draw [blue] (-0.7,0.5) node [below] {\scriptsize $\de$} to (0.7,0.5);
\end{tikzpicture}
     \hspace{3mm}
\begin{tikzpicture}[baseline=7mm]
 \coordinate (0) at (0,0);
 \coordinate (p) at (0,1);
 \draw (0) to node[fill=white,inner sep=1]{\scriptsize $r$} (p);
 \draw (0) .. controls (-1.5,2) and (1.5,2) .. node[fill=white,inner sep=1]{\scriptsize $\tau_{1}\ \mbox{or}\ \tau_{n}$} (0);
 \draw [blue] (-0.7,1) node [left] {\scriptsize $\de$} to (p);
 \fill (0) circle (0.7mm); \fill (p) circle (0.7mm);
\end{tikzpicture}
 \hspace{-9mm}$\rightarrow$
\begin{tikzpicture}[baseline=7mm]
 \coordinate (d) at (0,0);
 \coordinate (u) at (0,1.5);
 \coordinate (r) at (1,1);
 \draw (d) to node[fill=white,inner sep=1,pos=0.45]{\scriptsize $\tau_{1}$} node[fill=white,inner sep=1,pos=0.3]{\scriptsize $\mbox{or}\ \tau_{n}$} (u);
 \draw (r) to node[fill=white,inner sep=1]{\scriptsize $r$} (u);
 \draw (r) to node[fill=white,inner sep=1]{\scriptsize $r$} (d);
 \draw [blue] (-0.4,1) node [left] {\scriptsize $\de$} to (r);
\end{tikzpicture}
     \hspace{3mm}
\begin{tikzpicture}[baseline=7mm]
 \coordinate (0) at (0,0);
 \coordinate (p) at (0,1);
 \draw (0) to node[fill=white,inner sep=1,pos=0.7]{\scriptsize $\z_i$} (p);
 \draw (0) .. controls (-1.5,2) and (1.5,2) .. node[fill=white,inner sep=1]{\scriptsize $\z_{i-1}=\z_{i+1}$} (0);
 \draw [blue] (-0.7,0) node [above] {\scriptsize $\de$} to node[pos=0.7]{\rotatebox{90}{\footnotesize $\bowtie$}} (0);
 \fill (0) circle (0.7mm); \fill (p) circle (0.7mm);
\end{tikzpicture}
 \hspace{-8mm}$\rightarrow$
\begin{tikzpicture}[baseline=7mm]
 \coordinate (0) at (0,0);
 \coordinate (l) at (-1,1.5);
 \coordinate (p) at (0,1.5);
 \coordinate (r) at (1,1.5);
 \draw (0) to node[fill=white,inner sep=1,pos=0.7]{\scriptsize $\z_{i-1}$} (l);
 \draw (p) to node[fill=white,inner sep=1]{\scriptsize $\z_{i}$} (l);
 \draw (p) to node[fill=white,inner sep=1,pos=0.3]{\scriptsize $\z_{i}$} (0);
 \draw (p) to node[fill=white,inner sep=1]{\scriptsize $\z_{i}$} (r);
 \draw (0) to node[fill=white,inner sep=1,pos=0.7]{\scriptsize $\z_{i+1}$} (r);
 \draw [blue] (-0.7,0) node [above] {\scriptsize $\de$} to node[pos=0.7]{\rotatebox{90}{\footnotesize $\bowtie$}} (0);
\end{tikzpicture}
\end{figure}
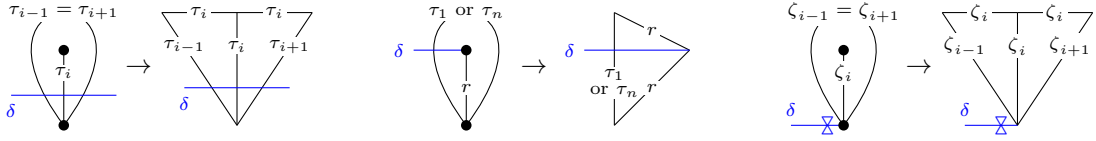

 In this paper, we call interior arcs of each polygon $T_{\de}$ {\it diagonals} and non-interior arcs of $T_{\de}$ {\it boundary segments}. We recall a definition we introduced in \cite{Y}.

\begin{definition}\label{defpm}
 A {\it perfect matching of angles} in $T_{\de}$ is a selection of marked angles such that:\par
 (1) each vertex $v$ incident to at least one diagonal is matched to one marked angle incident to $v$,\par
 (2) each triangle of $T_{\de}$ has exactly one marked angle.\\
 We denote by $\bA(T_{\de})$ the set of perfect matchings of angles in $T_{\de}$.
\end{definition}

 It is easy to see that $\bA(T_{\de})\neq\emptyset$ (e.g. see Figure \ref{minmatching}).

 For a diagonal or boundary segment $\tau$ of $T_{\de}$, we denote $x_{\tau} = x_{\tau'}$ if $\tau$ corresponds to a non-boundary segment $\tau'$ of $T$ and we denote $x_{\tau} = 1$ otherwise. Then, for an angle $a$ of $T_{\de}$, $x_a:=x_{\tau}$, where $\tau$ is the side opposite to $a$ in the triangle containing $a$. Using Assumption \ref{ass1}, we define a ring homomorphism
\begin{equation}\label{Phi1}
\Phi : \bZ[x_1^{\pm 1},\ldots, x_N^{\pm 1}] \rightarrow \bZ[x_1^{\pm 1},\ldots,x_N^{\pm 1}]
\end{equation}
 by
{\setlength\arraycolsep{0.5mm}
\begin{eqnarray*}
   \Phi(x_{j})&:=&\left\{
     \begin{array}{ll}
 x_{j}x_{k} & \ \ \mbox{if $j$ is a $1$-notched arc, where $k$ is the plain arc of $T$ corresponding to $j$,} \\
 x_{j} & \ \ \mbox{otherwise,}
     \end{array} \right.
\end{eqnarray*}}for any $j \in [1,N]$. Our main result Theorem \ref{main} gives a cluster expansion formula for cluster algebras with principal coefficients defined from triangulated surfaces. In this subsection, we specialize it to the coefficient-free case.

\begin{theorem}\label{main'}
 Let $\de$ be a tagged arc of $(S,M)$.\par
 (1) If $\overline{\de} \notin T$, we have
\[
 x_{\de}=\Phi\Biggl(\frac{1}{\cross(T,\de)}\sum_{A\in\bA(T_{\de})}x(A)\Biggr),\ \ \text{where}\ \ \cross(T,\de):=\prod_{\tau\in T_{\de}} x_{\tau}\ \ \text{and}\ \ x(A):=\prod_{a\in A} x_a.
\]

 (2) Suppose that $\overline{\de} \in T$ and $\de \notin T$. Let $p$ and $q$ be the endpoints of $\overline{\de}$. If $p$ (resp., $q$) is a puncture, we denote by $\ell_p$ (resp., $\ell_q$) the loop with endpoint $q$ (resp., $p$) cutting out a monogon containing only $p$ (resp., $q$). We can define triangulated polygons $T_{\ell_p}$ and $T_{\ell_q}$ in the same way as for plain arcs. Then, for $s = p$ or $q$, we have
{\setlength\arraycolsep{0.5mm}
\begin{eqnarray*}
   x_{\de}&=&\left\{
     \begin{array}{ll}
 \frac{x_{\ell_s}}{x_{\overline{\de}}} &\ \ \mbox{if $\de=\overline{\de}^{(s)}$}\\
 \frac{x_{\ell_p}x_{\ell_q}+1}{x_{\overline{\de}}} &\ \ \mbox{if $\de=\overline{\de}^{(pq)}$}\\
     \end{array} \right\},
\ \ \text{where}\ \ x_{\ell_s}=\Phi\Biggl(\frac{1}{\cross(T,\ell_s)}\sum_{A\in\bA(T_{\ell_s})}x(A)\Biggr).
\end{eqnarray*}}
\end{theorem}

 There are two key steps to prove Theorem \ref{main}.

 The first step is the cluster expansion formula given by Musiker-Schiffler-Williams \cite{MSW1}. A {\it perfect matching} in a graph $G$ is a set $P$ of edges of $G$ such that each vertex of $G$ is contained in exactly one edge in $P$. One can construct a snake graph $G_{\de}$ associated with $T_{\de}$. Musiker-Schiffler-Williams gave a cluster expansion formula in terms of perfect matchings of $G_{\de}$ (see Subsection \ref{MSWformula}). Note that perfect matchings of $G_{\g^{(p)}}$ and $G_{\g^{(pq)}}$ are different from general perfect matchings of graphs, that are also called {\it symmetric perfect matchings} and {\it compatible perfect matchings}, respectively (see Definitions \ref{defsym} and \ref{defcom}).

 The second step is Theorem \ref{bijthm} below. It gives bijections between several different combinatorial objects, that we introduce now.
 The bipartite graph $B_{\de}$ associated with $T_{\de}$ is defined as follows: The set of black vertices consists of vertices incident to at least one diagonal of $T_{\de}$ and the set of white vertices consists of triangles of $T_{\de}$. Edges are drawn between the white vertex corresponding to a triangle $ABC$ and the three black vertices corresponding to $A$, $B$ and $C$ if they exist.
 On the other hand, we associate to $\de$ a quiver with potential $(\oQ_{\de},\overline{W}_{\de})$ in Subsection \ref{QPandcuts}, and we define {\it minimal cuts} of $(\oQ_{\de},\overline{W}_{\de})$ in Definition \ref{defmin}.

\begin{theorem}\label{bijthm}
 There are bijections between the following objects:\par
 (1) Perfect matchings of angles in $T_{\de}$, \hspace{3mm} (2) Perfect matchings of $G_{\de}$,\par
 (3) Perfect matchings of $B_{\de}$, \hspace{17mm} (4) Minimal cuts of $(\oQ_{\de},\overline{W}_{\de})$,\\
 for any tagged arc $\de$ of $(S,M)$ such that $\overline{\de} \notin T$.
\end{theorem}

 By Theorem \ref{bijthm}, we also obtain cluster expansion formulas in terms of perfect matchings of bipartite graphs and minimal cuts of quivers with potential. More precisely, the bijection between (1) and (3) in Theorem \ref{bijthm} is induced by a natural bijection $\varpi$ between the set of angles incident to at least one diagonal of $T_{\de}$ and the set of edges of $B_{\de}$ given by the following picture:
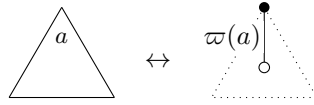
\[
\begin{tikzpicture}[baseline=-8mm]
 \coordinate (u) at (0,0);
 \coordinate (l) at (-0.7,-1.2);
 \coordinate (r) at (0.7,-1.2);
 \node at (0,-0.4) {\small $a$};
 \draw (u) -- (l);
 \draw (u) -- (r);
 \draw (r) -- (l);
\end{tikzpicture}
 \hspace{3mm} \leftrightarrow \hspace{3mm}
\begin{tikzpicture}[baseline=-8mm]
 \coordinate (u) at (0,0);
 \coordinate (l) at (-0.7,-1.2);
 \coordinate (r) at (0.7,-1.2);
 \coordinate (c) at (0,-0.8);
 \draw[dotted] (u) -- (l);
 \draw[dotted] (u) -- (r);
 \draw[dotted] (r) -- (l);
 \draw[left] (u) -- node[fill=white,inner sep=1]{$\varpi(a)$} (c);
 \fill (u) circle (0.7mm); \filldraw[fill=white] (c) circle (0.7mm);
\end{tikzpicture}
\]
 For a side $e$ of $B_{\de}$, we denote $x_{e} = x_{\varpi^{-1}(e)}$. 
 For a tagged arc $\de$ of $(S,M)$ with $\overline{\de} \notin T$, we have
\[
 x_{\de}=\Phi\Biggl(\frac{1}{\cross(T,\de)}\sum_{E}x(E)\Biggr),\ \ \text{where}\ \ x(E):=\prod_{e\in A} x_e,
\]
and $E$ runs over all perfect matchings of $B_{\de}$. Similarly, we obtain a cluster expansion formula in terms of minimal cuts of quivers with potential (see Corollary \ref{cutformula}).

 Our main result Theorem \ref{main} is obtained from the bijection between (1) and (2) in Theorem \ref{bijthm} and the cluster expansion formula of Musiker-Schiffler-Williams by showing that the bijection preserves the corresponding initial cluster variables. Notice that the construction of $T_{\de}$ is simpler than the one of $G_{\de}$. Moreover, the definition of a perfect matching of angles is more uniform than the definition of a perfect matching of $G_{\de}$, where three cases need to be distinguished depending of the tags attached to $\de$. Therefore, our new formula simplifies the formula of Musiker-Schiffler-Williams.

\subsection{Example in the coefficient-free case}\label{excoefffree}
 Let $(S,M)$ be a square with three punctures. We consider the following tagged triangulation $T$ and the corresponding ideal triangulation $T^0$:
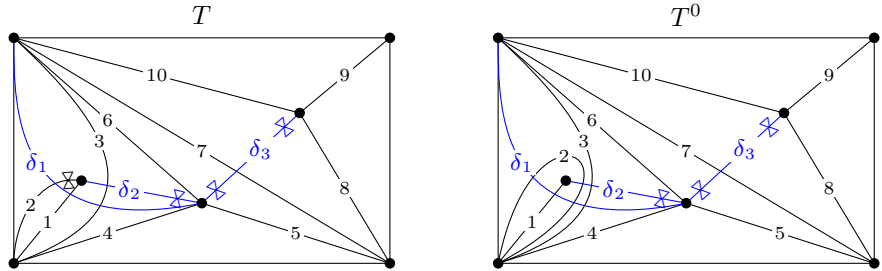
\[

\]
 The cluster algebra $\cA(T)$ has initial cluster variables $x_1,\ldots,x_{10}$. The ring homomorphism $\Phi : \bZ[x_1^{\pm 1},\ldots, x_{10}^{\pm 1}] \rightarrow \bZ[x_1^{\pm 1},\ldots,x_{10}^{\pm 1}]$ is given by
{\setlength\arraycolsep{0.5mm}
\begin{eqnarray*}
   \Phi(x_{i})=\left\{
     \begin{array}{ll}
 x_{1}x_{2} & \ \ \mbox{if $i=2$,} \\
 x_{i} & \ \ \mbox{otherwise.}
     \end{array} \right.
\end{eqnarray*}}The combinatorial data corresponding to the above three tagged arcs $\de_1$, $\de_2$ and $\de_3$ are given as follows:

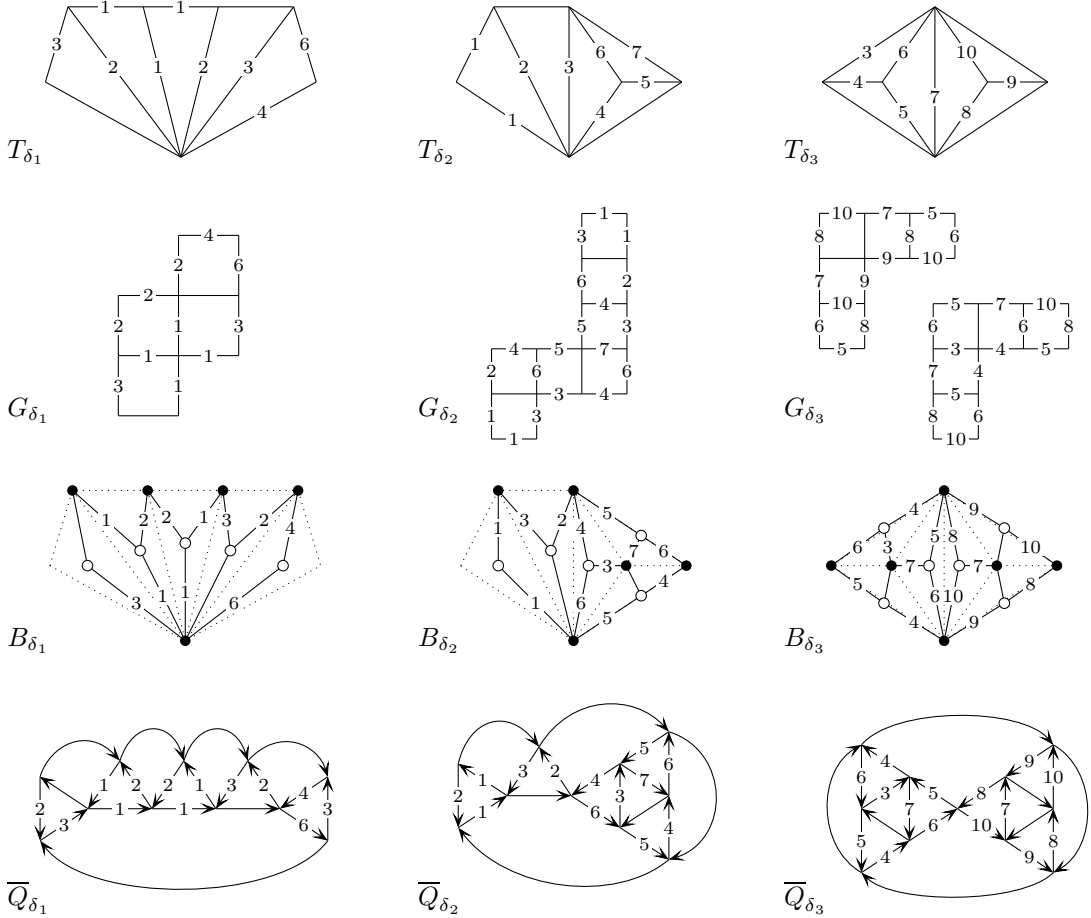
\[
     \begin{array}{lll}
 T_{\de_1}
\begin{tikzpicture}
 \coordinate (l) at (0,0);
 \coordinate (lu) at (0.3,1);
 \coordinate (lc) at (1.3,1);
 \coordinate (ru) at (3.3,1);
 \coordinate (rc) at (2.3,1);
 \coordinate (r) at (3.6,0);
 \coordinate (d) at (1.8,-1);
 \draw (l) to node[fill=white,inner sep=1]{\scriptsize $3$} (lu);
 \draw (lu) to node[fill=white,inner sep=1]{\scriptsize $1$} (lc);
 \draw (lc) to node[fill=white,inner sep=1]{\scriptsize $1$} (rc);
 \draw (rc) to (ru);
 \draw (ru) to node[fill=white,inner sep=1]{\scriptsize $6$} (r);
 \draw (r) to node[fill=white,inner sep=1,pos=0.4]{\scriptsize $4$} (d);
 \draw (l) to (d);
 \draw (lu) to node[fill=white,inner sep=1,pos=0.4]{\scriptsize $2$} (d);
 \draw (lc) to node[fill=white,inner sep=1,pos=0.4]{\scriptsize $1$} (d);
 \draw (rc) to node[fill=white,inner sep=1,pos=0.4]{\scriptsize $2$} (d);
 \draw (ru) to node[fill=white,inner sep=1,pos=0.4]{\scriptsize $3$} (d);
\end{tikzpicture}
 \hspace{10mm}
 & T_{\de_2}
\begin{tikzpicture}
 \coordinate (l) at (0,0);
 \coordinate (lu) at (0.5,1);
 \coordinate (ru) at (1.5,1);
 \coordinate (r) at (3,0);
 \coordinate (d) at (1.5,-1);
 \coordinate (c) at (2.2,0);
 \draw (l) to node[fill=white,inner sep=1]{\scriptsize $1$} (lu);
 \draw (lu) to (ru);
 \draw (ru) to node[fill=white,inner sep=1,pos=0.6]{\scriptsize $7$} (r);
 \draw (r) to (d);
 \draw (l) to node[fill=white,inner sep=1]{\scriptsize $1$} (d);
 \draw (lu) to node[fill=white,inner sep=1,pos=0.4]{\scriptsize $2$} (d);
 \draw (ru) to node[fill=white,inner sep=1,pos=0.4]{\scriptsize $3$} (d);
 \draw (ru) to node[fill=white,inner sep=1,pos=0.6]{\scriptsize $6$} (c);
 \draw (r) to node[fill=white,inner sep=1,pos=0.6]{\scriptsize $5$} (c);
 \draw (d) to node[fill=white,inner sep=1,pos=0.6]{\scriptsize $4$} (c);
\end{tikzpicture}
 \hspace{10mm}
 & T_{\de_3}
\begin{tikzpicture}
 \coordinate (l) at (0,0);
 \coordinate (u) at (1.5,1);
 \coordinate (r) at (3,0);
 \coordinate (d) at (1.5,-1);
 \coordinate (lc) at (0.8,0);
 \coordinate (rc) at (2.2,0);
 \draw (l) to node[fill=white,inner sep=1,pos=0.4]{\scriptsize $3$} (u);
 \draw (u) to (r);
 \draw (r) to (d);
 \draw (d) to (l);
 \draw (l) to node[fill=white,inner sep=1,pos=0.6]{\scriptsize $4$} (lc);
 \draw (d) to node[fill=white,inner sep=1,pos=0.6]{\scriptsize $5$} (lc);
 \draw (u) to node[fill=white,inner sep=1,pos=0.6]{\scriptsize $6$} (lc);
 \draw (d) to node[fill=white,inner sep=1,pos=0.6]{\scriptsize $8$} (rc);
 \draw (r) to node[fill=white,inner sep=1,pos=0.6]{\scriptsize $9$} (rc);
 \draw (u) to node[fill=white,inner sep=1,pos=0.6]{\scriptsize $10$} (rc);
 \draw (u) to node[fill=white,inner sep=1,pos=0.6]{\scriptsize $7$} (d);
\end{tikzpicture}\vspace{5mm}\\
\end{array}
\]
 We use Theorem \ref{main'} to obtain the cluster expansions of $x_{\de_1}$, $x_{\de_2}$ and $x_{\de_3}$ with respect to the initial cluster variables $x_1,\ldots,x_{10}$ in $\cA(T)$.

\noindent (1) $\de_1$: There are five perfect matchings of angles in $T_{\de_1}$, corresponding to five monomials as follows:
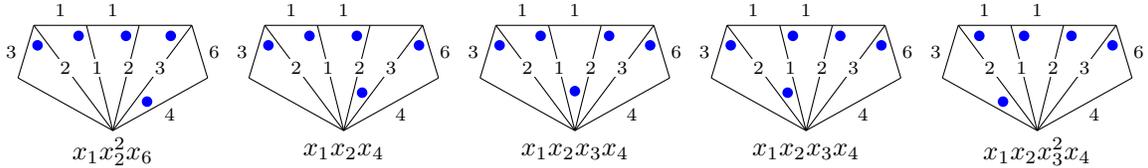
\[
%
\]
 Since $\cross(T,\de_1)=x_1x_2^2x_3$, the corresponding cluster variable is
{\setlength\arraycolsep{0.5mm}
\[
 x_{\de_1} = \Phi\biggl(\frac{1}{x_2x_3}(x_2x_6+x_4+x_3x_4+x_3x_4+x_3^2x_4)\biggr) = \frac{1}{x_1x_2x_3}(x_1x_2x_6+x_4+2x_3x_4+x_3^2x_4).
\]
}

\noindent (2) $\de_2$: There are nine perfect matchings of angles in $T_{\de_2}$, corresponding to nine monomials as follows:
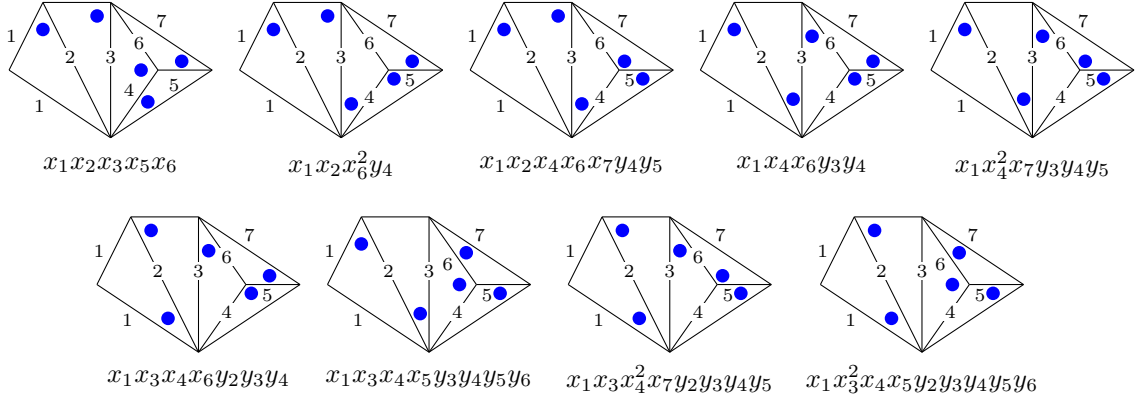
\[
%
\]
 Since $\cross(T,\de_2)=x_2x_3x_4x_5x_6$, the corresponding cluster variable is
{\setlength\arraycolsep{0.5mm}
\begin{eqnarray*}
 x_{\de_2} &=& \Phi\biggl(\scalebox{0.85}{$\displaystyle\frac{1}{x_2x_3x_4x_5x_6}
\left(
     \begin{array}{ll}
   x_1x_2x_3x_5x_6+x_1x_2x_6^2+x_1x_2x_4x_6x_7+x_1x_4x_6+x_1x_4^2x_7\\
   +x_1x_3x_4x_6+x_1x_3x_4x_5+x_1x_3x_4^2x_7+x_1x_3^2x_4x_5
     \end{array} \right)
$}\biggr)\\
   &=& \scalebox{0.85}{$\displaystyle\frac{1}{x_2x_3x_4x_5x_6}
\left(
     \begin{array}{ll}
   x_1x_2x_3x_5x_6+x_1x_2x_6^2+x_1x_2x_4x_6x_7+x_4x_6+x_4^2x_7\\
   +x_3x_4x_6+x_3x_4x_5+x_3x_4^2x_7+x_3^2x_4x_5
     \end{array} \right).
$}
\end{eqnarray*}
}

\noindent (3) $\de_3$: There are $12$ perfect matchings of angles in $T_{\de_3}$ as follows:
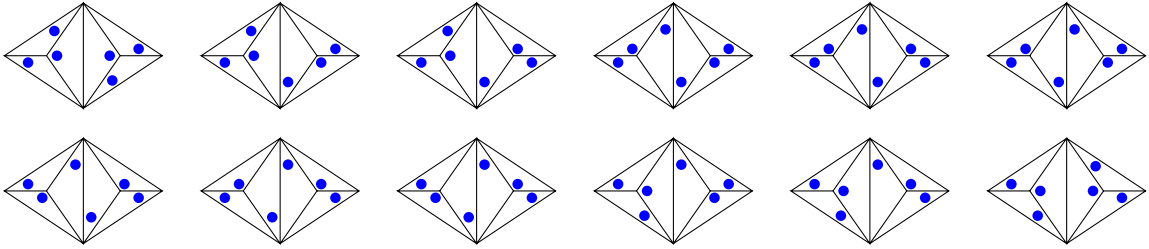
\[
%
\]
and $6$ others obtained by rotation of angle $\pi$ from the bottom row. Since $\cross(T,\de_3)=x_4x_5x_6x_7x_8x_9x_{10}$, the corresponding cluster variable is
\[
 x_{\de_3}=\scalebox{0.9}{$\displaystyle\frac{1}{x_4x_5x_6x_7x_8x_9x_{10}}
\left(
     \begin{array}{ll}
   x_4x_5x_7^2x_9x_{10}+x_4x_5x_7x_{10}^2+x_3x_5^2x_7x_9x_{10}\\
   +x_4x_5x_7x_8x_{10}+x_5x_6x_9x_{10}+x_3x_5^2x_{10}^2\\
   +x_3x_5^2x_8x_{10}+x_5x_6x_{10}^2+x_3x_5x_6x_8x_{10}\\
   +x_5x_6x_8x_{10}+x_3x_5x_6x_8^2+x_6^2x_8x_{10}\\
   +x_6^2x_8^2+x_4x_6x_7x_8x_{10}+x_3x_5x_6x_7x_8x_9\\
   +x_4x_6x_7x_8^2+x_6^2x_7x_8x_9+x_4x_6x_7^2x_8x_9
     \end{array} \right),
$}
\]
which is not affected by $\Phi$ since $x_2$ don't appear.

 For the case (2), we illustrate Theorem \ref{bijthm} in Examples \ref{ex2s}, \ref{ex2b} and \ref{ex2q}.

\subsection{Our results in the principal coefficients case}\label{principalcoeff}

 We keep the notations of Subsection \ref{cfcase}. Let $\z_1,\ldots,\z_m$ (resp., $\xi_1,\ldots,\xi_{\ell}$) be the diagonals of $T_{\de}$ incident to $p$ (resp., $q$) winding counter-clockwisely around $p$ (resp., $q$) such that $\tau_n$, $\z_1$, and $\z_m$ (resp., $\tau_1$, $\xi_1$, and $\xi_{\ell}$) are contained in the same triangle (see Figure \ref{minmatching}).
 We define an element $A_{-}(T_{\de}) \in \bA(T_{\de})$, which we call the {\it minimal matching} of $T_{\de}$, satisfying the following {\it min-condition}: For each boundary vertex $v$ of $T_{\de}$ that is incident to at least one diagonal of $T_{\de}$, the angle $a \in A_{-}(T_{\de})$ at $v$ comes first in the counterclockwise order around $v$. Clearly, the minimal matching is uniquely determined (see Figure \ref{minmatching}).

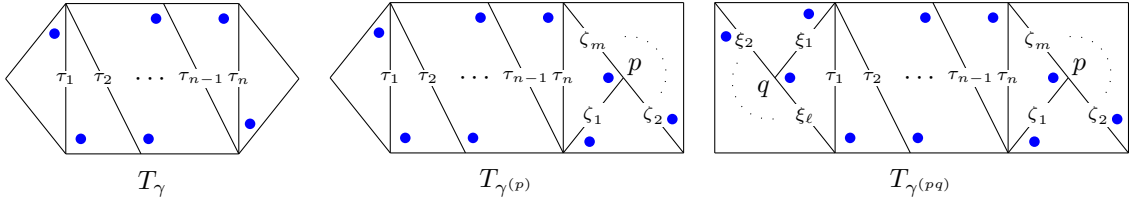
\begin{figure}[h]
   \caption{Minimal matchings}
   \label{minmatching}
\[
\begin{tikzpicture}
 \coordinate (l) at (0,0);
 \coordinate (lu) at (0.8,1);
 \coordinate (ru) at (3.1,1);
 \coordinate (ld) at (0.8,-1);
 \coordinate (rd) at (3.1,-1);
 \coordinate (r) at (3.9,0);
 \coordinate (u) at (2.1,1);
 \coordinate (d) at (1.8,-1);
 \fill[blue] (0.65,0.6) circle (0.7mm); \fill[blue] (2,0.8) circle (0.7mm); \fill[blue] (2.9,0.8) circle (0.7mm); \fill[blue] (1,-0.8) circle (0.7mm); \fill[blue] (1.9,-0.8) circle (0.7mm); \fill[blue] (3.25,-0.6) circle (0.7mm);
 \node at (1.95,0) {$\cdots$};
 \node at (1.95,-1.4) {$T_{\g}$};
 \draw (l) to (lu);
 \draw (l) to (ld);
 \draw (lu) to node[fill=white,inner sep=1]{\scriptsize $\tau_1$} (ld);
 \draw (lu) to (ru);
 \draw (ld) to (rd);
 \draw (ru) to (r);
 \draw (rd) to (r);
 \draw (ru) to node[fill=white,inner sep=1]{\scriptsize $\tau_n$} (rd);
 \draw (u) to node[fill=white,inner sep=1]{\scriptsize $\tau_{n-1}$} (rd);
 \draw (lu) to node[fill=white,inner sep=1]{\scriptsize $\tau_2$} (d);
\end{tikzpicture}
     \hspace{4mm}
\begin{tikzpicture}
 \coordinate (l) at (0,0);
 \coordinate (lu) at (0.8,1);
 \coordinate (ld) at (0.8,-1);
 \coordinate (ru) at (3.1,1);
 \coordinate (rd) at (3.1,-1);
 \coordinate (rru) at (4.7,1);
 \coordinate (rrd) at (4.7,-1);
 \coordinate (r) at (3.9,0);
 \coordinate (u) at (2.1,1);
 \coordinate (d) at (1.8,-1);
 \fill[blue] (0.65,0.6) circle (0.7mm); \fill[blue] (2,0.8) circle (0.7mm); \fill[blue] (2.9,0.8) circle (0.7mm); \fill[blue] (1,-0.8) circle (0.7mm); \fill[blue] (1.9,-0.8) circle (0.7mm); \fill[blue] (3.45,-0.85) circle (0.7mm); \fill[blue] (4.55,-0.55) circle (0.7mm); \fill[blue] (3.7,0) circle (0.7mm);
 \node at (1.95,0) {$\cdots$};
 \node at (2.35,-1.4) {$T_{\g^{(p)}}$};
 \node at (4.05,0.15) {$p$};
 \draw (l) to (lu);
 \draw (l) to (ld);
 \draw (lu) to node[fill=white,inner sep=1]{\scriptsize $\tau_1$} (ld);
 \draw (lu) to (ru);
 \draw (ld) to (rd);
 \draw (ru) to node[fill=white,inner sep=1]{\scriptsize $\z_m$} (r);
 \draw (rd) to node[fill=white,inner sep=1]{\scriptsize $\z_1$} (r);
 \draw (ru) to (rru);
 \draw (rru) to (rrd);
 \draw (rrd) to (rd);
 \draw (rrd) to node[fill=white,inner sep=1]{\scriptsize $\z_2$} (r);
 \draw (ru) to node[fill=white,inner sep=1]{\scriptsize $\tau_n$} (rd);
 \draw (u) to node[fill=white,inner sep=1]{\scriptsize $\tau_{n-1}$} (rd);
 \draw (lu) to node[fill=white,inner sep=1]{\scriptsize $\tau_2$} (d);
 \draw[loosely dotted] (4.4,-0.2) arc (-30:110:0.5);
\end{tikzpicture}
     \hspace{4mm}
\begin{tikzpicture}
 \coordinate (l) at (0,0);
 \coordinate (lu) at (0.8,1);
 \coordinate (ld) at (0.8,-1);
 \coordinate (ru) at (3.1,1);
 \coordinate (rd) at (3.1,-1);
 \coordinate (rru) at (4.7,1);
 \coordinate (rrd) at (4.7,-1);
 \coordinate (llu) at (-0.8,1);
 \coordinate (lld) at (-0.8,-1);
 \coordinate (r) at (3.9,0);
 \coordinate (u) at (2.1,1);
 \coordinate (d) at (1.8,-1);
 \fill[blue] (2,0.8) circle (0.7mm); \fill[blue] (2.9,0.8) circle (0.7mm); \fill[blue] (1,-0.8) circle (0.7mm); \fill[blue] (1.9,-0.8) circle (0.7mm); \fill[blue] (3.45,-0.85) circle (0.7mm); \fill[blue] (4.55,-0.55) circle (0.7mm); \fill[blue] (3.7,0) circle (0.7mm); \fill[blue] (0.45,0.85) circle (0.7mm); \fill[blue] (-0.65,0.55) circle (0.7mm); \fill[blue] (0.2,0) circle (0.7mm);
 \node at (1.95,0) {$\cdots$};
 \node at (1.95,-1.4) {$T_{\g^{(pq)}}$};
 \node at (4.05,0.15) {$p$};
 \node at (-0.15,-0.15) {$q$};
 \draw (l) to node[fill=white,inner sep=1]{\scriptsize $\xi_1$} (lu);
 \draw (l) to node[fill=white,inner sep=1]{\scriptsize $\xi_{\ell}$} (ld);
 \draw (lu) to node[fill=white,inner sep=1]{\scriptsize $\tau_1$} (ld);
 \draw (lu) to (ru);
 \draw (ld) to (rd);
 \draw (ru) to node[fill=white,inner sep=1]{\scriptsize $\z_m$} (r);
 \draw (rd) to node[fill=white,inner sep=1]{\scriptsize $\z_1$} (r);
 \draw (ru) to (rru);
 \draw (rru) to (rrd);
 \draw (rrd) to (rd);
 \draw (rrd) to node[fill=white,inner sep=1]{\scriptsize $\z_2$} (r);
 \draw (lu) to (llu);
 \draw (llu) to (lld);
 \draw (lld) to (ld);
 \draw (llu) to node[fill=white,inner sep=1]{\scriptsize $\xi_2$} (l);
 \draw (ru) to node[fill=white,inner sep=1]{\scriptsize $\tau_n$} (rd);
 \draw (u) to node[fill=white,inner sep=1]{\scriptsize $\tau_{n-1}$} (rd);
 \draw (lu) to node[fill=white,inner sep=1]{\scriptsize $\tau_2$} (d);
 \draw[loosely dotted] (4.4,-0.2) arc (-30:110:0.5);
 \draw[loosely dotted] (-0.5,0.2) arc (150:290:0.5);
\end{tikzpicture}
\]
\end{figure}

 We expand the ring homomorphism (\ref{Phi1}) into
\[
\Phi : \bZ[x_1^{\pm 1},\ldots, x_N^{\pm 1},y_1^{\pm 1},\ldots,y_N^{\pm 1}] \rightarrow \bZ[x_1^{\pm 1},\ldots,x_N^{\pm 1},y_1^{\pm 1},\ldots,y_N^{\pm 1}]
\]
 by
{\setlength\arraycolsep{0.5mm}
\begin{eqnarray*}
   \Phi(y_j)&:=&\left\{
     \begin{array}{ll}
 \frac{y_j}{y_k} & \ \ \mbox{if $j$ is plain and corresponds to the $1$-notched arc $k$ of $T$,} \\
 y_j & \ \ \mbox{otherwise,}
     \end{array} \right.
\end{eqnarray*}}for any $j \in [1,N]$.
 For two sets $A$ and $B$, we denote by $A \triangle B$ the symmetric difference $(A \cup B)\setminus(A \cap B)$. An {\it exterior angle} of $T_{\de}$ is an angle between a boundary segment and a diagonal of $T_{\de}$. Let $A \in \bA(T_{\de})$. We denote by $Y'(A)$ the set of diagonals of $T_{\de}$ that are sides of at least one exterior angle in $A_{-}(T_{\de}) \triangle A$. We define the set
{\setlength\arraycolsep{0.5mm}
\begin{eqnarray*}
   Y(A)&:=&\left\{
     \begin{array}{ll}
 Y'(A)\sqcup\{\tau_1\} & \ \ \mbox{if $\de=\g^{(pq)}$, $n=1$, and $A$ contains at least one of the four angles between} \\
 & \ \ \mbox{$\z_m$ or $\xi_{\ell}$ and $\tau_1$ or a boundary segment of $T_{\g^{(pq)}}$,} \\
 Y'(A) & \ \ \mbox{otherwise.}
     \end{array} \right.
\end{eqnarray*}}We are ready to state the main theorem of this paper.


\begin{theorem}\label{main}
 Let $\de$ be a tagged arc of $(S,M)$.\par
 (1) If $\overline{\de} \notin T$, we have
\[
 x_{\de}=\Phi\Biggl(\frac{1}{\cross(T,\de)}\sum_{A\in\bA(T_{\de})}x(A)y(A)\Biggr),\ \ \text{where}\ \ y(A):=\prod_{\tau \in Y(A)} y_{\tau}.
\]

 (2) Suppose that $\overline{\de} \in T$ and $\de \notin T$. Let $r$ and $s$ be the endpoints of $\overline{\de}$. Then, for $s = p$ or $q$, we have
{\setlength\arraycolsep{0.5mm}
\begin{eqnarray*}
   x_{\de}&=&\left\{
     \begin{array}{ll}
 \frac{x_{\ell_s}}{x_{\overline{\de}}} &\ \ \mbox{if $\de=\overline{\de}^{(s)}$,}\\
 \frac{x_{\ell_p}x_{\ell_q}y_{\overline{\de}}+(1-\prod_{\tau \in T}y_{\tau}^{e_p(\tau)})(1-\prod_{\tau \in T}y_{\tau}^{e_q(\tau)})}{x_{\overline{\de}}} &\ \ \mbox{if $\de=\overline{\de}^{(pq)}$,}\\
     \end{array} \right.
\end{eqnarray*}}where $e_s(\tau)$ is the number of ends of $\tau$ incident to $s$, and
\[
 x_{\ell_s}=\Phi\Biggl(\frac{1}{\cross(T,\ell_s)}\sum_{A\in\bA(T_{\ell_s})}x(A)y(A)\Biggr).
\]
\end{theorem}

 Since Theorem \ref{main}(2) follows from \cite[Lemma 8.2, Theorem 8.6]{FT} and \cite[Proposition 4.21]{MSW1}, we only prove Theorem \ref{main}(1) in Section \ref{proofofmain}. Theorem \ref{main} is a generalization of \cite[Theorem 1.3]{Y}.

 In the rest of this section, we consider the bipartite graph $B_{\de}$.
 We define the {\it minimal matching} of $B_{\de}$ by $E_-(B_{\de}) := \varpi^{-1}(A_-(T_{\de})) \in \bP(B_{\de})$, where $\bP(B_{\de})$ the set of perfect matchings of $B_{\de}$. For a diagonal $\tau$ of $T_{\de}$, there are exactly two triangles $\triangle$, $\triangle'$ of $T_{\de}$ with edge $\tau$. We label by $\tau$ the square of $B_{\de}$ whose vertices are two white vertices corresponding to $\triangle$, $\triangle'$ and two black vertices corresponding to endpoints of $\tau$.

\begin{proposition}\label{biparenclosed}
 For $E \in \bP(B_{\de})$, the set $E_-(B_{\de}) \triangle E$ consists of all boundary edges of some (possibly empty or disconnected) subgraph $B_E$ of $B_{\de}$ that is a union of squares.
\end{proposition}

 We denote by $I(E)$ the set of the squares of $B_{\de}$ contained in $B_E$.

\begin{proposition}\label{biparformula}
 For $E \in \bP(B_{\de})$, $I(E)=Y(\varpi^{-1}(E))$ holds.
\end{proposition}

 By Theorem \ref{main} and Proposition \ref{biparformula}, for a tagged arc $\de$ of $(S,M)$ such that $\overline{\de} \notin T$, we have
\[
 x_{\de}=\Phi\Biggl(\frac{1}{\cross(T,\de)}\sum_{E \in \mathbb{P}(B_{\de})}x(E)y(E)\Biggr),\ \ \text{where}\ \ y(E):=\prod_{i \in I(E)} y_{i}.
\]
This formula is a generalization of the cluster expansion formula in type $A$ given by Carroll and Price \cite{CPr} (see also \cite{CPi,P}). We prove Propositions \ref{biparenclosed} and \ref{biparformula} in Section \ref{bipartite}.

\subsection{Example in the principal coefficients case}\label{exprincipal}

 We consider the square $(S,M)$ with three punctures and the tagged triangulation $T$ of $(S,M)$ given in Subsection \ref{excoefffree}. The cluster algebra $\cA(T)$ has initial cluster variables $x_1,\ldots,x_{10}$ and initial principal coefficients $y_1,\ldots,y_{10}$. The ring homomorphism $\Phi : \bZ[x_1^{\pm 1},\ldots, x_{10}^{\pm 1},y_1^{\pm 1},\ldots,y_{10}^{\pm 1}] \rightarrow \bZ[x_1^{\pm 1},\ldots,x_{10}^{\pm 1},y_1^{\pm 1},\ldots,y_{10}^{\pm 1}]$ is given by
{\setlength\arraycolsep{0.5mm}
\begin{eqnarray*}
   \Phi(x_{i})=\left\{
\begin{array}{ll}
 x_{1}x_{2} & \ \ \mbox{if $i=2$,} \\
 x_{i} & \ \ \mbox{otherwise,}
     \end{array} \right.\hspace{3mm}
   \Phi(y_i)=\left\{
     \begin{array}{ll}
 \frac{y_1}{y_2} & \ \ \mbox{if $i=1$,} \\
 y_i & \ \ \mbox{otherwise.}
     \end{array} \right.
\end{eqnarray*}}We use Theorem \ref{main} to obtain the cluster expansions of $\de_1$, $\de_2$ and $\de_3$ given in Subsection \ref{excoefffree} with respect to the initial cluster variables $x_1,\ldots,x_{10}$ and coefficients $y_1,\ldots,y_{10}$ in $\cA(T)$.

\noindent (1) $\de_1$: The minimal matching is
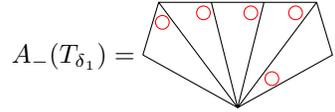
\[
A_{-}(T_{\de_1})=
%
.
\]
 Then there are five perfect matchings of angles in $T_{\de_1}$, corresponding to five monomials as follows:
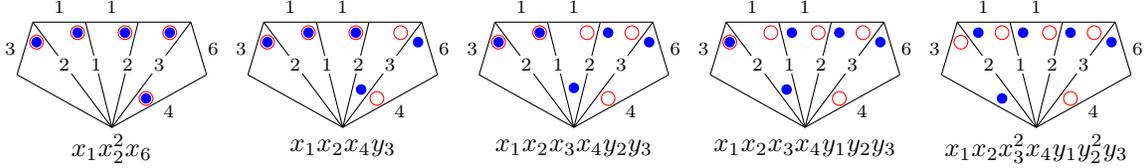
\[
%
\]
 Since $\cross(T,\de_1)=x_1x_2^2x_3$, the corresponding cluster variable is
{\setlength\arraycolsep{0.5mm}
\begin{eqnarray*}
 x_{\de_1} &=& \Phi\biggl(\frac{1}{x_2x_3}(x_2x_6+x_4y_3+x_3x_4y_2y_3+x_3x_4y_1y_2y_3+x_3^2x_4y_1y_2^2y_3)\biggr)\\
              &=& \frac{1}{x_1x_2x_3}(x_1x_2x_6+x_4y_3+x_3x_4y_2y_3+x_3x_4y_1y_3+x_3^2x_4y_1y_2y_3).
\end{eqnarray*}
}

\noindent (2) $\de_2$: The minimal matching is
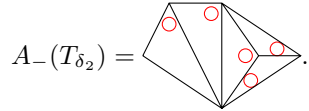
\[
A_{-}(T_{\de_2})=
%
.
\]
 Then there are nine perfect matchings of angles in $T_{\de_2}$, corresponding to nine monomials as follows:
\[
%
\]
 Remark that the three angles incident to the puncture of $T_{\de_2}$ are not exterior angles and thus don't contribute to the coefficients. Since $\cross(T,\de_2)=x_2x_3x_4x_5x_6$, the corresponding cluster variable is
{\setlength\arraycolsep{0.5mm}
\begin{eqnarray*}
 x_{\de_2} &=& \Phi\biggl(\scalebox{0.85}{$\displaystyle\frac{1}{x_2x_3x_4x_5x_6}
\left(
     \begin{array}{ll}
   x_1x_2x_3x_5x_6+x_1x_2x_6^2y_4+x_1x_2x_4x_6x_7y_4y_5+x_1x_4x_6y_3y_4+x_1x_4^2x_7y_3y_4y_5\\
   +x_1x_3x_4x_6y_2y_3y_4+x_1x_3x_4x_5y_3y_4y_5y_6+x_1x_3x_4^2x_7y_2y_3y_4y_5+x_1x_3^2x_4x_5y_2y_3y_4y_5y_6
     \end{array} \right)
$}\biggr)\\
   &=& \scalebox{0.85}{$\displaystyle\frac{1}{x_2x_3x_4x_5x_6}
\left(
     \begin{array}{ll}
   x_1x_2x_3x_5x_6+x_1x_2x_6^2y_4+x_1x_2x_4x_6x_7y_4y_5+x_4x_6y_3y_4+x_4^2x_7y_3y_4y_5\\
   +x_3x_4x_6y_2y_3y_4+x_3x_4x_5y_3y_4y_5y_6+x_3x_4^2x_7y_2y_3y_4y_5+x_3^2x_4x_5y_2y_3y_4y_5y_6
     \end{array} \right).
$}
\end{eqnarray*}
}

\noindent (3) $\de_3$: The minimal matching is
\[
A_{-}(T_{\de_3})=
%
.
\]
 Then there are $12$ perfect matchings of angles in $T_{\de_3}$ as follows:
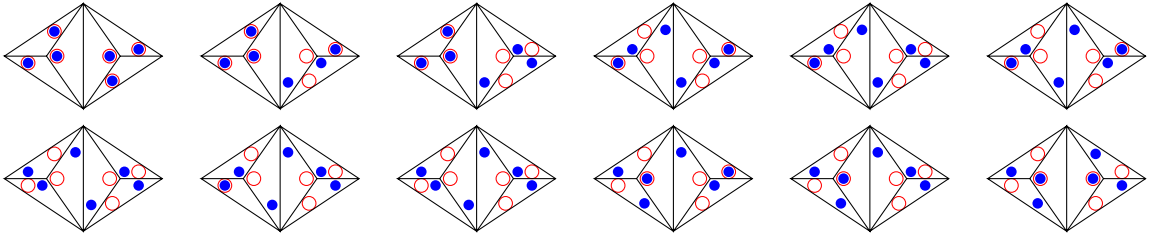
\[
%
\]
and $6$ others obtained by rotation of angle $\pi$ from the bottom row. Since $\cross(T,\de_3)=x_4x_5x_6x_7x_8x_9x_{10}$, the corresponding cluster variable is
\[
 x_{\de_3}=\scalebox{0.9}{$\displaystyle\frac{1}{x_4x_5x_6x_7x_8x_9x_{10}}
\left(
     \begin{array}{ll}
   x_4x_5x_7^2x_9x_{10}+x_4x_5x_7x_{10}^2y_8+x_3x_5^2x_7x_9x_{10}y_6\\
   +x_4x_5x_7x_8x_{10}y_8y_9+x_5x_6x_9x_{10}y_4y_6+x_3x_5^2x_{10}^2y_6y_8\\
   +x_3x_5^2x_8x_{10}y_6y_8y_9+x_5x_6x_{10}^2y_4y_6y_8+x_3x_5x_6x_8x_{10}y_6y_7y_8\\
   +x_5x_6x_8x_{10}y_4y_6y_8y_9+x_3x_5x_6x_8^2y_6y_7y_8y_9+x_6^2x_8x_{10}y_4y_6y_7y_8\\
   +x_6^2x_8^2y_4y_6y_7y_8y_9+x_4x_6x_7x_8x_{10}y_4y_5y_6y_7y_8+x_3x_5x_6x_7x_8x_9y_6y_7y_8y_9y_{10}\\
   +x_4x_6x_7x_8^2y_4y_5y_6y_7y_8y_9+x_6^2x_7x_8x_9y_4y_6y_7y_8y_9y_{10}+x_4x_6x_7^2x_8x_9y_4y_5y_6y_7y_8y_9y_{10}
     \end{array} \right),
$}
\]
which is not affected by $\Phi$ since $x_2$ and $y_1$ don't appear.

\subsection{$f$-vectors and intersection numbers}

 We keep the notations of Subsection \ref{principalcoeff}. We recall $f$-vectors of cluster variables \cite[Definition 2.6]{FuG}: For a cluster variable $x$ of $\cA(T)$, let $f_{x,1},\ldots,f_{x,N}$ be the maximal degrees of $y_1,\ldots,y_N$ in the polynomial obtained from the Laurent expression of $x$ by substituting $1$ for each of $x_1,\ldots,x_N$. The integer vector $f_x:=(f_{x,1},\ldots,f_{x,N}) \in \bZ_{\ge 0}^N$ is called the {\it $f$-vector of $x$}. For a tagged arc $\de$ of $(S,M)$ such that $\overline{\de} \notin T$, by Theorem \ref{main}(1), the $f$-vector $(f_{x_{\de},1},\ldots,f_{x_{\de},N})$ of $x_{\de}$ is given by
\begin{equation}\label{fTde}
 \prod_{i=1}^N y_i^{f_{x_{\de},i}} = \Phi\Biggl( \prod_{\tau \in T_{\de}} y_{\tau}\Biggr).
\end{equation}

 On the other hand, for tagged arcs $\de$ and $\epsilon$ of $(S,M)$, Qiu and Zhou \cite{QZ} defined the intersection number between $\de$ and $\epsilon$ as follows: Assume that $\de$ and $\epsilon$ intersect transversally in a minimum number of points in $S \setminus M$. Then we define the intersection number $\Int(\de,\epsilon)=A+B+C$, where
\begin{itemize}
 \item $A$ is the number of intersection points of $\de$ and $\epsilon$ in $S \setminus M$;
 \item $B$ is the number of pairs of an end of $\de$ and an end of $\epsilon$ that are incident to a common puncture such that their tags are different;
 \item $C=0$ unless the ideal arcs corresponding to $\de$ and $\epsilon$ form a self-folded triangle, in which case $C=-1$.
\end{itemize}
 Note that this definition is different from the ``intersection number'' $(\de | \epsilon)$ defined in \cite[Definition 8.4]{FST}. We give the main result of this subsection.

\begin{theorem}
 For a tagged arc $\de$ of $(S,M)$, we have $f_{x_{\de},i}=\Int(\de,i)$ for $i \in [1,N]$.
\end{theorem}

\begin{proof}
 Considering in each case, it is easy to show that both $f_{x_{\de},i}$ and $\Int(\de,i)$ are equal to $f \in \bZ_{\ge 0}$ given as follows: If $\de \in T$, then $f=0$; Suppose that $\overline{\de} \notin T$. If $i$ is a plain arc of $T$, then $f$ is the number of diagonals of $T_{\de}$ corresponding to $i$. If $i$ is a $1$-notched arc of $T$, then $f$ is the number of diagonals of $T_{\de}$ corresponding to $i$ minus the number of diagonals of $T_{\de}$ corresponding to $\overline{i}$; Suppose that $\overline{\de} \in T$ and $T \notin T$. We use the notations of Theorem \ref{main}(2). If $\de=\overline{\de}^{(s)}$, then $f=e_s(i)-\de_{i\overline{\de}}$, where $\de_{i\overline{\de}}$ is the Kronecker delta. If $\de=\overline{\de}^{(pq)}$, then $f=e_p(i)+e_q(i)$.
\end{proof}

\section{Preliminary}\label{Preliminary}

 For the convenience of the reader, we recall basic definitions and facts about cluster algebras, triangulated surfaces and the cluster expansion formulas of Musiker-Schiffler-Williams (e.g. \cite{FST,FZ1,FZ2,MSW1}).


\subsection{Cluster algebras with principal coefficients}\label{CApc}

 To define cluster algebras with principal coefficients, we need to prepare some notations. Let ${\mathcal F}:=\mathbb{Q}(t_1,\ldots,t_{2N})$ be the field of rational functions in $2N$ variables over $\mathbb{Q}$.

\begin{definition}\cite[Definition 2.3]{FZ2}
 A {\it labeled seed} (or simply, {\it seed}) is a pair $(x,\oB)$ consisting of the following data:\par
 (i) $x=(x_1,\ldots,x_N,y_1,\ldots,y_N)$ is a free generating set of ${\mathcal F}$ over $\mathbb{Q}$.\par
 (ii) $\oB=(b_{ij})_{1 \le i \le 2N,1 \le j \le N}$ is a $2N \times N$ integer matrix whose upper part $B=(b_{ij})_{1 \le i, j \le N}$ is {\it skew-symmetric}, that is, $b_{ij}=-b_{ji}$ for any $i, j \in [1,N]$.\\
 Then we refer to $x$ as the {\it cluster}, to each $x_i$ as a {\it cluster variable}, to each $y_i$ as a {\it coefficient} and to $\oB$ as the {\it exchange matrix} of $(x,\oB)$.
\end{definition}

 In general, one may consider {\it skew-symmetrizable} or {\it sign-skew-symmetric} matrices as exchange matrices \cite{FZ1}. In this paper, we only study the skew-symmetric case as we focus on cluster algebras defined from triangulated surfaces.

\begin{definition}\cite[Definition 2.4, (2.15)]{FZ2}
 For a seed $(x,\oB)$, the $\it{mutation}$ $\mu_k(x,\oB)=(x',\oB')$ in direction $k$ $(1 \le k \le N)$ is defined as follows.\par
 (i) $x'=(x'_1,\ldots,x'_N,y_1,\ldots,y_N)$ is defined by
\begin{equation}\label{clustermutation}
  x_k x'_k = \prod_{i=1}^{N}x_i^{[b_{ik}]_+}y_i^{[b_{N+i,k}]_+}+\prod_{i=1}^{N}x_i^{[-b_{ik}]_+}y_i^{[-b_{N+i,k}]_+}, \ \ \mbox{and} \ \ x'_i = x_i \ \ \mbox{if} \ \ i \neq k,
\end{equation}
 where $[x]_{+}:=\max(x,0)$.\par
 (ii) $\oB'=(b'_{ij})_{1 \le i \le 2N,1 \le j \le N}$ is defined by
\begin{eqnarray}\label{matrixmutation}
   b'_{ij} =  \left\{
     \begin{array}{ll}
   -b_{ij} & \mbox{if} \ \ i=k \ \ \mbox{or} \ \ j=k,\\
   b_{ij}+\frac{b_{ik}}{|b_{ik}|}[b_{ik}b_{kj}]_+ & \mbox{otherwise}.
     \end{array} \right.
\end{eqnarray}
\end{definition}

 Then it is elementary that $\mu_k(x,\oB)$ is also a seed. Moreover, $\mu_k$ is an involution, that is, we have $\mu_k\mu_k(x,\oB)=(x,\oB)$.\par

 Now we define cluster algebras with principal coefficients. For a skew-symmetric $N \times N$ integer matrix $B$, we define $\wB=(b_{ij})$ as the $2N \times N$ integer matrix whose upper part $(b_{ij})_{1 \le i,j \le N}$ is $B$ and lower part $(b_{ij})_{N+1 \le i \le 2N,1 \le j \le N}$ is the $N \times N$ identity matrix. We fix a seed $(x=(x_1,\ldots,x_N,y_1,\ldots,y_N),\wB)$ that we call an {\it initial seed}. We also call each $x_i$ an {\it initial cluster variable}.

\begin{definition}\cite[Definition 3.1]{FZ2}\label{clusteralg}
 The {\it cluster algebra} $\cA(B)=\cA(x,\wB)$ {\it with principal coefficients} for the initial seed $(x,\wB)$ is the $\bZ$-subalgebra of ${\mathcal F}$ generated by the cluster variables and coefficients obtained by all sequences of mutations from $(x,\wB)$.
\end{definition}

 One of the remarkable properties of cluster algebras is the {\it Laurent phenomenon}.

\begin{theorem}\cite[Theorem 3.1]{FZ1}
 Every element of the cluster algebra $\cA(B)$ is a Laurent polynomial over $\bZ[y_1,\ldots,y_N]$ in the initial cluster variables, that is, $\cA(B)$ is contained in $\bZ[x_1^{\pm 1},\ldots,x_N^{\pm 1},y_1,\ldots,y_N]$.
\end{theorem}

\begin{example}\label{ex1}
The matrix $B=
\begin{bmatrix}
0&1\\
-1&0
\end{bmatrix}$
is skew-symmetric. Let $((x_1,x_2,y_1,y_2),\wB)$ be a seed. Then we get the cluster algebra with principal coefficients
 \begin{equation*}
  \cA(B)=\bZ\bigl[x_1,x_2,\mbox{$\frac{x_2+y_1}{x_1},\frac{1+x_1y_2}{x_2},\frac{x_2+y_1+x_1y_1y_2}{x_1x_2}$},y_1,y_2\bigr].
 \end{equation*}
\end{example}


\subsection{Ideal and tagged triangulations}\label{tag tri}

 Let $S$ be a connected compact oriented Riemann surface with (possibly empty) boundary and $M$ a non-empty finite set of marked points on $S$ with at least one marked point on each boundary component if $S$ has boundaries. We call the pair $(S,M)$ a {\it marked surface}. Any marked point in the interior of $S$ is called a {\it puncture}. For technical reasons, $(S,M)$ is not a monogon with at most one puncture, a digon without punctures, a triangle without punctures nor a sphere with at most three punctures.

 An {\it ordinary arc} $\de$ in $(S,M)$ is a curve in $S$ with endpoints in $M$, considered up to isotopy, such that: $\de$ does not intersect itself except at its endpoints; $\de$ is disjoint from $M$ and from the boundary of $S$ except at its endpoints; $\de$ does not cut out an unpunctured monogon or an unpunctured digon. An ordinary arc with two identical endpoints is called a {\it loop}. A curve homotopic to a boundary component between two marked points is called a {\it boundary segment}.

 Two ordinary arcs are called {\it compatible} if they do not intersect in the interior of $S$. An {\it ideal triangulation} is a maximal collection of pairwise compatible ordinary arcs.
 A triangle with only two distinct sides is called {\it self-folded} (see Figure \ref{self-folded}).
\begin{figure}[htp]
   \caption{A self-folded triangle and the corresponding tagged arc}
   \label{self-folded}
\[
 \begin{xy}
  (0,0)="0"*{\bullet}, **\crv{(-10,10)&(0,20)&(10,10)}, "0"+(0,10)="1"*{\bullet}, "1"+(0,6)*{{\scriptsize\mbox{$\g$}}}
  \ar@{-}"0";"1"^(.0){o}^(.9){p}
 \end{xy}\hspace{30mm}
 \begin{xy}
  (0,0)="0"*{\bullet}, "0"+(0,10)="1"*{\bullet}, **\crv{(5,7)}, "1"+(2,-2)*{{\rotatebox{40}{\footnotesize $\bowtie$}}}, "1"+(6,-5)*{{\scriptsize\mbox{$\iota(\g)$}}}
  \ar@{-}"0";"1"^(.0){o}^(.9){p}
 \end{xy}
\]
\end{figure}
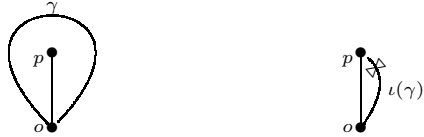

 For an ideal triangulation $T$, a {\it flip} at an ordinary arc $\de \in T$ replaces $\de$ by another arc $\de' \notin T$ such that $T\setminus\{\de\}\cup\{\de'\}$ is an ideal triangulation. Notice that an ordinary arc inside a self-folded triangle can not be flipped. This problem was solved by the notion of {\it tagged arcs} introduced in \cite{FST}.

\begin{definition}\cite[Definition 7.1]{FST}
 A {\it tagged arc} is an ordinary arc with each end tagged in one of two ways, {\it plain} or {\it notched}, such that the following conditions are satisfied: the tagged arc does not cut out a once-punctured monogon; an endpoint lying on the boundary of $S$ is tagged plain; both ends of a loop are tagged in the same way.
\end{definition}

 In this paper, we assume that if $(S,M)$ is a closed surface with exactly one puncture, all tagged arcs are plain arcs. For an ordinary arc $\g$ of $(S,M)$, we define a tagged arc $\iota(\g)$ as follows: If $\g$ does not cut out a once-punctured monogon, $\iota(\g)$ is the tagged arc obtained from $\g$ by tagging both ends plain: If $\g$ cuts out a once-punctured monogon with endpoint $o$ and puncture $p$, $\iota(\g)$ is the tagged arc obtained by tagging the unique arc that connects $o$ and $p$ and does not intersect $\g$, plain at $o$ and notched at $p$ (see Figure \ref{self-folded}). For a tagged arc $\de$, we denote by $\de^{\circ}$ the ordinary arc obtained from $\de$ by forgetting its tags.

\begin{definition}\cite[Definition 7.4]{FST}
 Two tagged arcs $\de$ and $\epsilon$ are called {\it compatible} if the following conditions are satisfied:
\begin{itemize}
 \item the two ordinary arcs $\de^{\circ}$ and $\epsilon^{\circ}$ are compatible,
 \item if $\de^{\circ}=\epsilon^{\circ}$, at least one end of $\epsilon$ is tagged in the same way as the corresponding end of $\de$,
 \item if $\de^{\circ}\neq\epsilon^{\circ}$ and they have a common endpoint $o$, the ends of $\de$ and $\epsilon$ at $o$ are tagged in the same way.
\end{itemize}
 A {\it tagged triangulation} is a maximal collection of pairwise compatible tagged arcs.
\end{definition}

 Note that it is possible to flip at any tagged arc of a tagged triangulation \cite[Theorem 7.9]{FST}. Moreover, any two tagged triangulations of $(S,M)$ are connected by a sequence of flips by \cite[Proposition 7.10]{FST}.


\subsection{Cluster algebras defined from triangulated surfaces}\label{clalgfromtri}

 Let $(S,M)$ be a marked surface. First, we consider an ideal triangulation $T$ of $(S,M)$. For an ordinary arc $\g$, $\pi(\g)$ is defined as follows: if there is a self-folded triangle in $T$ with non-loop side $\g$, $\pi(\g)$ is its loop side; otherwise $\pi(\g)=\g$.

\begin{definition}\cite[Definition 4.1]{FST}
 Let $T$ be an ideal triangulation of $(S,M)$ and $t_1,\ldots,t_N$ be all ordinary arcs of $T$. For any non-self-folded triangle $\triangle$ in $T$, an $N \times N$ matrix $B^{\triangle}=(b_{ij}^{\triangle})$ is defined by
\[
b_{ij}^{\triangle}=\left\{
\begin{array}{l}
1,\ \ \ \mbox{if $\pi(t_i)$ and $\pi(t_j)$ are sides of $\triangle$ with $\pi(t_j)$ following $\pi(t_i)$ in the clockwise order,}\\
-1,\ \mbox{if $\pi(t_i)$ and $\pi(t_j)$ are sides of $\triangle$ with $\pi(t_j)$ following $\pi(t_i)$ in the counterclockwise order,}\\
0,\ \ \ \mbox{otherwise.}
\end{array}
\right.
\]
 We define the $N \times N$ matrix $B_T=\sum_{\triangle}B^{\triangle}$, where $\triangle$ runs over all non-self-folded triangles in $T$.
\end{definition}

 We consider a tagged triangulation $T$ of $(S,M)$. We obtain a tagged triangulation $\hat{T}$ from $T$ by simultaneous changing all tags at some punctures, in such a way that there is an ideal triangulation $T^0$ satisfying $\hat{T}=\iota(T^0)$ (see \cite[Remark 3.11]{MSW1}). Notice that $\hat{T}$ satisfies Assumption \ref{ass1}.

\begin{definition}\cite[Definition 9.6]{FST}
 For a tagged triangulation $T$, we define the $N \times N$ matrix $B_T:=B_{T^0}$.
\end{definition}

 Since $B_T$ is skew-symmetric, we get a cluster algebra $\cA(T):=\cA(B_T)$ with principal coefficients for any tagged triangulation $T$.

\begin{theorem}\label{bij}\cite[Theorem 7.11]{FST}\cite[Theorem 6.1]{FT}
 Let $T$ be a tagged triangulation of $(S,M)$. Then the tagged arcs $\de$ of $(S,M)$ correspond bijectively with the cluster variables $x_{\de}$ in $\cA(T)$. This induces that the tagged triangulations $T'$ of $(S,M)$ correspond bijectively with the clusters $x_{T'}$ in $\cA(T)$. Moreover, the tagged triangulation obtained from $T'$ by flipping at $\de \in T'$ corresponds the cluster obtained from $x_{T'}$ by mutating at $x_{\de}$.
\end{theorem}

 For a tagged arc $t$ and a puncture $p$ of $(S,M)$, we denote by $t^{(p)}$ the tagged arc obtained from $t$ by changing tags at $p$, where $t^{(p)}=t$ if $p$ is not an endpoint of $t$.

\begin{proposition}\cite[Proposition 3.15]{MSW1}\label{MSW1315}
 Let $T$ be a tagged triangulation of $(S,M)$ consisting of tagged arcs $t_1,\ldots,t_N$. We denote by $T^{(p)}$ the tagged triangulation consisting of $t_1^{(p)},\ldots,t_N^{(p)}$. Let $\Sigma_{T}=(x, B_T)$ and $\Sigma_{T^{(p)}}=(x^{(p)},$ $B_{T^{(p)}})$ be the corresponding initial seeds of $\cA(T)$ and $\cA(T^{(p)})$, respectively. Then for a tagged arc $\de$, we have
\[
 [x_{\de^{(p)}}]_{\Sigma_{T^{(p)}}}^{\cA(T^{(p)})}=[x_{\de}]_{\Sigma_{T}}^{\cA(T)}|_{x \leftarrow x^{(p)}, y \leftarrow y^{(p)}},
\]
 where $[x_{\de}]_{\Sigma_{T}}^{\cA(T)}$ is the cluster expansion of $x_{\de}$ with respect to $\Sigma_{T}$ in $\cA(T)$.
\end{proposition}

 In view of Proposition \ref{MSW1315}, since we have $\hat{T}=T^{(p_1\cdots p_r)}$ for some punctures $p_1,\ldots,p_r$, it is enough to consider a tagged triangulation $T$ satisfying $T=\hat{T}$, that is satisfying Assumption \ref{ass1}. In the rest of this paper, we assume that any tagged triangulation satisfy Assumption \ref{ass1}. Moreover, suppose that there is a $1$-notched arc $t \in T$ with endpoint $p$ tagged notched. Let $s \in T$ the corresponding plain arc. Then $t^{(p)}=s$ and $s^{(p)}=t$ hold. Therefore, for a tagged arc $\de$, we have
\[
 [x_{\de^{(p)}}]_{\Sigma_{T}}^{\cA(T)} = [x_{\de}]_{\Sigma_{T^{(p)}}}^{\cA(T^{(p)})}|_{x \leftarrow x^{(p)}, y \leftarrow y^{(p)}} = [x_{\de}]_{\Sigma_{T}}^{\cA(T)}|_{x_t \leftrightarrow x_s}
\]
by Proposition \ref{MSW1315}. Thus we can make Assumption \ref{ass2}.


\subsection{Musiker-Schiffler-Williams cluster expansion formulas}\label{MSWformula}

 In this subsection, we recall the cluster expansion formula given by Musiker-Schiffler-Williams \cite{MS,MSW1}. We call it the {\it MSW formula}. Fix a marked surface $(S,M)$ and a tagged triangulation $T$ of $(S,M)$ satisfying Assumptions \ref{ass1} and \ref{ass2}. Let $\g$ be a plain arc of $(S,M)$ such that $\g \notin T$. We use the notations of the introduction.

\subsubsection{Formula for plain arcs}

 Recall the MSW formula for $x_{\g}$. In the triangulation $T_{\g}$ constructed in the introduction, triangles have at most two sides that are non-boundary segments and at least one side that is a boundary segment. We construct the {\it snake graph} $\overline{G}_{\g}:=\overline{G}_{T_{\g}}$ from $T_{\g}$ by {\it unfolding} each triangle of $T_{\g}$, two sides of which are non-boundary segments, along its third side (see Figure \ref{unfolding}). We label all edges of $\overline{G}_{\g}$ by the corresponding tagged arcs of $T$.

\begin{figure}[h]
   \caption{Unfolding $\triangle$, where $a$ is boundary segment, while $b$ and $c$ are not}
   \label{unfolding}
\[
  \begin{xy}
   (0,0)="0", +(15,0)="1", +(15,0)="2", "0"+(0,-10)="3", +(10,0)="4", +(10,0)="5", +(10,0)="6", "2"+(5,-5)="A", +(15,0)="B", +(5,10)="00", +(15,0)="01", "00"+(0,-10)="02", +(10,0)="03", +(10,0)="04", +(10,0)="05", "03"+(5,-10)="06", +(15,0)="07", "1"+(0,-7)*{\triangle}, "01"+(0,-7)*{\triangle}, "06"+(0,7)*{\rotatebox{180}{$\triangle$}}
   \ar@{-}"0";"1" \ar@{-}"1";"2" \ar@{-}"3";"4" {\color{blue}\ar@{-}"4";"5"_{a}}\ar@{-}"5";"6" \ar@{-}"1";"4"_{b} \ar@{-}"1";"5"^{c} \ar@{=>}"A";"B"^{{\rm unfolding}}_{{\rm along}\ \ a} \ar@{-}"00";"01" \ar@{-}"02";"03" {\color{blue}\ar@{-}"03";"04"|{\ a\ }}\ar@{-}"04";"05" \ar@{-}"06";"07" \ar@{-}"01";"03"_{b} \ar@{-}"01";"04"^{c} \ar@{-}"03";"06"_{b} \ar@{-}"04";"06"^{c}
  \end{xy}
\]
\end{figure}
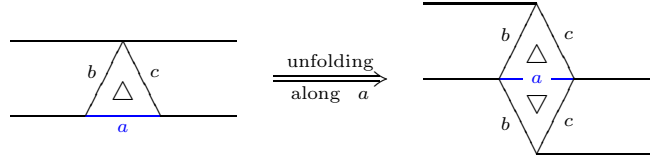

\begin{example}\label{exsnake}
 We construct the snake graph $\overline{G}_{\de_1}$ for the tagged arc $\de_1$ given in Subsection \ref{excoefffree} as follows:
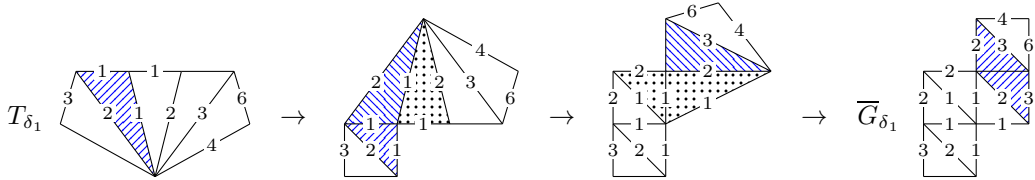
\[
 T_{\de_1}\hspace{2mm}
\begin{tikzpicture}[scale=0.7,baseline=0mm]
 \coordinate (l) at (0,0);
 \coordinate (lu) at (0.3,1);
 \coordinate (lc) at (1.3,1);
 \coordinate (ru) at (3.3,1);
 \coordinate (rc) at (2.3,1);
 \coordinate (r) at (3.6,0);
 \coordinate (d) at (1.8,-1);
 \fill [pattern=north east lines, pattern color=blue] (d)--(lu)--(lc)--(d);
 \draw (l) to node[fill=white,inner sep=1]{\scriptsize $3$} (lu);
 \draw (lu) to node[fill=white,inner sep=1]{\scriptsize $1$} (lc);
 \draw (lc) to node[fill=white,inner sep=1]{\scriptsize $1$} (rc);
 \draw (rc) to (ru);
 \draw (ru) to node[fill=white,inner sep=1]{\scriptsize $6$} (r);
 \draw (r) to node[fill=white,inner sep=1,pos=0.4]{\scriptsize $4$} (d);
 \draw (l) to (d);
 \draw (lu) to node[fill=white,inner sep=1,pos=0.4]{\scriptsize $2$} (d);
 \draw (lc) to node[fill=white,inner sep=1,pos=0.4]{\scriptsize $1$} (d);
 \draw (rc) to node[fill=white,inner sep=1,pos=0.4]{\scriptsize $2$} (d);
 \draw (ru) to node[fill=white,inner sep=1,pos=0.4]{\scriptsize $3$} (d);
\end{tikzpicture}
     \hspace{3mm}\rightarrow\hspace{3mm}
\begin{tikzpicture}[scale=0.7,baseline=7mm]
 \coordinate (l) at (0.3,0);
 \coordinate (lu) at (0.3,1);
 \coordinate (lc) at (1.3,1);
 \coordinate (ru) at (3.3,1);
 \coordinate (rc) at (2.3,1);
 \coordinate (r) at (3.6,2);
 \coordinate (u) at (1.8,3);
 \coordinate (d) at (1.3,0);
 \fill [pattern=north west lines, pattern color=blue] (u)--(lu)--(lc)--(u);
 \fill [pattern=north east lines, pattern color=blue] (d)--(lu)--(lc)--(d);
 \fill [blue,pattern=dots] (u)--(rc)--(lc)--(u);
 \draw (l) to node[fill=white,inner sep=1]{\scriptsize $3$} (lu);
 \draw (lu) to node[fill=white,inner sep=1]{\scriptsize $1$} (lc);
 \draw (lc) to node[fill=white,inner sep=1]{\scriptsize $1$} (rc);
 \draw (rc) to (ru);
 \draw (ru) to node[fill=white,inner sep=1]{\scriptsize $6$} (r);
 \draw (r) to node[fill=white,inner sep=1,pos=0.4]{\scriptsize $4$} (u);
 \draw (l) to (d);
 \draw (lu) to node[fill=white,inner sep=1,pos=0.4]{\scriptsize $2$} (u);
 \draw (lc) to node[fill=white,inner sep=1,pos=0.4]{\scriptsize $1$} (u);
 \draw (rc) to node[fill=white,inner sep=1,pos=0.4]{\scriptsize $2$} (u);
 \draw (ru) to node[fill=white,inner sep=1,pos=0.4]{\scriptsize $3$} (u);
 \draw (lu) to node[fill=white,inner sep=1]{\scriptsize $2$} (d);
 \draw (lc) to node[fill=white,inner sep=1]{\scriptsize $1$} (d);
\end{tikzpicture}
     \hspace{3mm}\rightarrow\hspace{3mm}
\begin{tikzpicture}[scale=0.7,baseline=7mm]
 \coordinate (ld) at (0,0);
 \coordinate (d) at (1,0);
 \coordinate (l) at (0,1);
 \coordinate (c) at (1,1);
 \coordinate (lu) at (0,2);
 \coordinate (cu) at (1,2);
 \coordinate (u) at (1,3);
 \coordinate (r) at (3,2);
 \coordinate (ru) at (2,3.3);
 \fill [blue,pattern=dots] (lu)--(c)--(cu)--(lu);
 \fill [blue,pattern=dots] (r)--(c)--(cu)--(r);
 \fill [pattern=north west lines, pattern color=blue] (r)--(cu)--(u)--(r);
 \draw (l) to node[fill=white,inner sep=1]{\scriptsize $2$} (d);
 \draw (ld) to node[fill=white,inner sep=1]{\scriptsize $3$} (l);
 \draw (l) to node[fill=white,inner sep=1]{\scriptsize $1$} (c);
 \draw (c) to node[fill=white,inner sep=1]{\scriptsize $1$} (d);
 \draw (ld) to (d);
 \draw (lu) to node[fill=white,inner sep=1]{\scriptsize $2$} (l);
 \draw (lu) to node[fill=white,inner sep=1]{\scriptsize $2$} (cu);
 \draw (c) to node[fill=white,inner sep=1]{\scriptsize $1$} (cu);
 \draw (lu) to node[fill=white,inner sep=1]{\scriptsize $1$} (c);
 \draw (r) to node[fill=white,inner sep=1,pos=0.6]{\scriptsize $1$} (c);
 \draw (r) to node[fill=white,inner sep=1,pos=0.6]{\scriptsize $2$} (cu);
 \draw (r) to node[fill=white,inner sep=1,pos=0.6]{\scriptsize $3$} (u);
 \draw (r) to node[fill=white,inner sep=1,pos=0.6]{\scriptsize $4$} (ru);
 \draw (ru) to node[fill=white,inner sep=1]{\scriptsize $6$} (u);
 \draw (u) to (cu);
\end{tikzpicture}
     \hspace{3mm}\rightarrow\hspace{3mm} \overline{G}_{\de_1} \hspace{2mm}
\begin{tikzpicture}[scale=0.7,baseline=7mm]
 \coordinate (ld) at (0,0);
 \coordinate (d) at (1,0);
 \coordinate (l) at (0,1);
 \coordinate (c) at (1,1);
 \coordinate (lu) at (0,2);
 \coordinate (cu) at (1,2);
 \coordinate (u) at (1,3);
 \coordinate (rd) at (2,1);
 \coordinate (r) at (2,2);
 \coordinate (ru) at (2,3);
 \fill [pattern=north east lines, pattern color=blue] (r)--(cu)--(rd)--(r);
 \fill [pattern=north east lines, pattern color=blue] (r)--(cu)--(u)--(r);
 \draw (l) to node[fill=white,inner sep=1]{\scriptsize $2$} (d);
 \draw (ld) to node[fill=white,inner sep=1]{\scriptsize $3$} (l);
 \draw (l) to node[fill=white,inner sep=1]{\scriptsize $1$} (c);
 \draw (c) to node[fill=white,inner sep=1]{\scriptsize $1$} (d);
 \draw (ld) to (d);
 \draw (lu) to node[fill=white,inner sep=1]{\scriptsize $2$} (l);
 \draw (lu) to node[fill=white,inner sep=1]{\scriptsize $2$} (cu);
 \draw (c) to node[fill=white,inner sep=1]{\scriptsize $1$} (cu);
 \draw (lu) to node[fill=white,inner sep=1]{\scriptsize $1$} (c);
 \draw (rd) to node[fill=white,inner sep=1]{\scriptsize $1$} (c);
 \draw (rd) to node[fill=white,inner sep=1]{\scriptsize $2$} (cu);
 \draw (r) to node[fill=white,inner sep=1]{\scriptsize $3$} (rd);
 \draw (cu) to node[fill=white,inner sep=1]{\scriptsize $2$} (u);
 \draw (r) to node[fill=white,inner sep=1]{\scriptsize $3$} (u);
 \draw (u) to node[fill=white,inner sep=1]{\scriptsize $4$} (ru);
 \draw (ru) to node[fill=white,inner sep=1]{\scriptsize $6$} (r);
 \draw (r) to (cu);
\end{tikzpicture}
\]
\end{example}

 Note that $\overline{G}_{\g}$ consists of $n$ squares with diagonals $\tau_i$ for $1 \le i \le n$. We call these squares {\it tiles} of $\overline{G}_{\g}$. Let $G_{\g} := G_{T_{\g}}$ be the graph obtained from $\overline{G}_{\g}$ by removing the diagonal of each tile. It is easy to see that the following special perfect matching is uniquely determined.

\begin{definition}\cite[Definition 4.7]{MSW1}\label{e0min}
 Let $e_0$ be the edge of $G_{\g}$ corresponding to the boundary segment of $T_{\g}$ that follows $\tau_1$ in the clockwise direction in the triangle $T_0$. The {\it minimal matching} $P_{-}(G_{\g})$ is the perfect matching of $G_{\g}$ containing $e_0$ and consisting only of boundary edges.
\end{definition}

 In Example \ref{exsnake}, $e_0$ is the bottom edge of $\overline{G}_{\de_1}$.

\begin{theorem}\cite[Theorem 5.1]{MS}\label{JPthm}
 For $P \in \bP(G_{\g})$, the set $P_{-}(G_{\g}) \triangle P$ consists of all boundary edges of some (possibly empty or disconnected) subgraph $G_P$ of $G_{\g}$ that is a union of tiles.
\end{theorem}

 We denote by $J(P)$ the set of the diagonals of all tiles of $\overline{G}_{\g}$ that are contained in $G_P$. The following cluster expansion formula is obtained by using perfect matchings of $G_{\g}$.

\begin{theorem}\label{MSW1cef}\cite[Theorem 4.10]{MSW1}
 We have
 \begin{equation*}
  x_{\g}=\Phi\Biggl(\frac{1}{\cross(T,\g)}\sum_{P \in \bP(G_{\g})}x(P)y(P)\Biggr),\ \ x(P):=\prod_{e \in P}x_{e},\ \ y(P):=\prod_{j \in J(P)}y_j.
 \end{equation*}
\end{theorem}

\subsubsection{Formula for $1$-notched arcs}

 Recall the MSW formula for $x_{\g^{(p)}}$. Let $q \neq p$ be the other endpoint of $\g^{(p)}$. In the same way as above, for the ordinary loop $\ell_p$ defined in Theorem \ref{main'}, we get the snake graph $G_{\ell_p}$ which is denoted by $G_{\g^{(p)}}$ in the introduction. By construction, $G_{\ell_p}$ contains two disjoint subgraphs $G_{\ell_p}^1$ and $G_{\ell_p}^2$ with same form as $G_{\g}$. Moreover, we consider the subgraph $H_{\ell_p}^i$ of $G_{\ell_p}^i$ obtained by removing the vertex $p$ and the two edges $\z_1$, $\z_m$.

\begin{example}\label{exloops}
 Let $\ell_p$ be the ordinary loop such that $\iota(\ell_p)=\de_2$ in given Subsection \ref{excoefffree}. We have the triangulated polygon $T_{\ell_p}$, the snake graph $G_{\ell_p}$ and the subgraphs $G_{\ell_p}^i$ and $H_{\ell_p}^i$ of $G_{\ell_p}$ as follows:
\[
\begin{tikzpicture}[scale=0.8,baseline=1mm]
 \coordinate (l) at (-2.5,0);
 \coordinate (llu) at (-2,1);
 \coordinate (lu) at (-1,1);
 \coordinate (u) at (0,1);
 \coordinate (ru) at (1,1);
 \coordinate (r) at (2.5,0);
 \coordinate (ld) at (-1,-1);
 \coordinate (d) at (0,-1);
 \coordinate (rd) at (1,-1);
 \coordinate (rrd) at (2,-1);
 \node at (-1,1.7) {$T_{\ell_p}$};
 \draw (llu)--(lu) (rrd)--(rd) (d)--(ld);
 \draw (l) to node[fill=white,inner sep=1]{\scriptsize $1$} (llu);
 \draw (lu) to node[fill=white,inner sep=1]{\scriptsize $6$} (u) node[above]{\scriptsize $p$};
 \draw (u) to node[fill=white,inner sep=1]{\scriptsize $4$} (ru);
 \draw (ru) to node[fill=white,inner sep=1]{\scriptsize $1$} (r);
 \draw (r) to node[fill=white,inner sep=1]{\scriptsize $1$} (rrd);
 \draw (rd) to node[fill=white,inner sep=1]{\scriptsize $7$} (d);
 \draw (ld) to node[fill=white,inner sep=1]{\scriptsize $1$} (l);
 \draw (llu) to node[fill=white,inner sep=1]{\scriptsize $2$} (ld);
 \draw (lu) to node[fill=white,inner sep=1]{\scriptsize $3$} (ld);
 \draw (u) to node[fill=white,inner sep=1]{\scriptsize $4$} (ld);
 \draw (u) to node[fill=white,inner sep=1]{\scriptsize $5$} (d);
 \draw (u) to node[fill=white,inner sep=1]{\scriptsize $6$} (rd);
 \draw (ru) to node[fill=white,inner sep=1]{\scriptsize $3$} (rd);
 \draw (ru) to node[fill=white,inner sep=1]{\scriptsize $2$} (rrd);
\end{tikzpicture}
     \hspace{10mm}
\begin{tikzpicture}[scale=0.7,baseline=15mm]
 \coordinate (00) at (0,0);
 \coordinate (01) at (0,1);
 \coordinate (02) at (0,2);
 \coordinate (10) at (1,0);
 \coordinate (11) at (1,1);
 \coordinate (12) at (1,2);
 \coordinate (21) at (2,1);
 \coordinate (22) at (2,2);
 \coordinate (23) at (2,3);
 \coordinate (24) at (2,4);
 \coordinate (25) at (2,5);
 \coordinate (31) at (3,1);
 \coordinate (32) at (3,2);
 \coordinate (33) at (3,3);
 \coordinate (34) at (3,4);
 \coordinate (35) at (3,5);
 \node at (0.5,4) {$G_{\ell_p}$};
 \draw (00) to node[fill=white,inner sep=1]{\scriptsize $1$} (01);
 \draw (01) to node[fill=white,inner sep=1]{\scriptsize $2$} (02);
 \draw (00) to node[fill=white,inner sep=1]{\scriptsize $1$} (10);
 \draw (01) to (11);
 \draw (02) to node[fill=white,inner sep=1]{\scriptsize $4$} (12) node[above]{\scriptsize $p$};
 \draw (10) to node[fill=white,inner sep=1]{\scriptsize $3$} (11);
 \draw (11) to node[fill=white,inner sep=1]{\scriptsize $6$} (12);
 \draw (11) to node[fill=white,inner sep=1]{\scriptsize $3$} (21);
 \draw (12) to node[fill=white,inner sep=1]{\scriptsize $5$} (22);
 \draw (21) to (22);
 \draw (22) to node[fill=white,inner sep=1]{\scriptsize $5$} (23) node[left]{\scriptsize $p$};
 \draw (23) to node[fill=white,inner sep=1]{\scriptsize $6$} (24);
 \draw (24) to node[fill=white,inner sep=1]{\scriptsize $3$} (25);
 \draw (21) to node[fill=white,inner sep=1]{\scriptsize $4$} (31) node[right,below]{\scriptsize $p$};
 \draw (22) to node[fill=white,inner sep=1]{\scriptsize $7$} (32);
 \draw (23) to node[fill=white,inner sep=1]{\scriptsize $4$} (33);
 \draw (24) to (34);
 \draw (25) to node[fill=white,inner sep=1]{\scriptsize $1$} (35);
 \draw (31) to node[fill=white,inner sep=1]{\scriptsize $6$} (32);
 \draw (32) to node[fill=white,inner sep=1]{\scriptsize $3$} (33);
 \draw (33) to node[fill=white,inner sep=1]{\scriptsize $2$} (34);
 \draw (34) to node[fill=white,inner sep=1]{\scriptsize $1$} (35);
 \draw[white,ultra thin] (01) to node[black]{\scriptsize $2$} (10);
 \draw[white,ultra thin] (02) to node[black]{\scriptsize $3$} (11);
 \draw[white,ultra thin] (12) to node[black]{\scriptsize $4$} (21);
 \draw[white,ultra thin] (22) to node[black]{\scriptsize $5$} (31);
 \draw[white,ultra thin] (23) to node[black]{\scriptsize $6$} (32);
 \draw[white,ultra thin] (24) to node[black]{\scriptsize $3$} (33);
 \draw[white,ultra thin] (25) to node[black]{\scriptsize $2$} (34);
\end{tikzpicture}
     \hspace{10mm}
\begin{tikzpicture}[scale=0.7,baseline=15mm]
 \coordinate (00) at (0,0);
 \coordinate (01) at (0,1);
 \coordinate (02) at (0,2);
 \coordinate (10) at (1,0);
 \coordinate (11) at (1,1);
 \coordinate (12) at (1,2);
 \coordinate (23) at (2,3);
 \coordinate (24) at (2,4);
 \coordinate (25) at (2,5);
 \coordinate (33) at (3,3);
 \coordinate (34) at (3,4);
 \coordinate (35) at (3,5);
 \node at (1.3,4) {$G_{\ell_p}^2$};
 \node at (1.7,1) {$G_{\ell_p}^1$};
 \draw (00) to node[fill=white,inner sep=1]{\scriptsize $1$} (01);
 \draw (01) to node[fill=white,inner sep=1]{\scriptsize $2$} (02);
 \draw (00) to node[fill=white,inner sep=1]{\scriptsize $1$} (10);
 \draw (01) to (11);
 \draw (02) to node[fill=white,inner sep=1]{\scriptsize $4$} (12) node[above]{\scriptsize $p$};
 \draw (10) to node[fill=white,inner sep=1]{\scriptsize $3$} (11);
 \draw (11) to node[fill=white,inner sep=1]{\scriptsize $6$} (12);
 \draw (23) node[left]{\scriptsize $p$} to node[fill=white,inner sep=1]{\scriptsize $6$} (24);
 \draw (24) to node[fill=white,inner sep=1]{\scriptsize $3$} (25);
 \draw (23) to node[fill=white,inner sep=1]{\scriptsize $4$} (33);
 \draw (24) to (34);
 \draw (25) to node[fill=white,inner sep=1]{\scriptsize $1$} (35);
 \draw (33) to node[fill=white,inner sep=1]{\scriptsize $2$} (34);
 \draw (34) to node[fill=white,inner sep=1]{\scriptsize $1$} (35);
 \draw[white,ultra thin] (01) to node[black]{\scriptsize $2$} (10);
 \draw[white,ultra thin] (02) to node[black]{\scriptsize $3$} (11);
 \draw[white,ultra thin] (24) to node[black]{\scriptsize $3$} (33);
 \draw[white,ultra thin] (25) to node[black]{\scriptsize $2$} (34);
\end{tikzpicture}
     \hspace{10mm}
\begin{tikzpicture}[scale=0.7,baseline=15mm]
 \coordinate (00) at (0,0);
 \coordinate (01) at (0,1);
 \coordinate (02) at (0,2);
 \coordinate (10) at (1,0);
 \coordinate (11) at (1,1);
 \coordinate (24) at (2,4);
 \coordinate (25) at (2,5);
 \coordinate (33) at (3,3);
 \coordinate (34) at (3,4);
 \coordinate (35) at (3,5);
 \node at (1.3,4) {$H_{\ell_p}^2$};
 \node at (1.7,1) {$H_{\ell_p}^1$};
 \draw (00) to node[fill=white,inner sep=1]{\scriptsize $1$} (01);
 \draw (01) to node[fill=white,inner sep=1]{\scriptsize $2$} (02);
 \draw (00) to node[fill=white,inner sep=1]{\scriptsize $1$} (10);
 \draw (01) to (11);
 \draw (10) to node[fill=white,inner sep=1]{\scriptsize $3$} (11);
 \draw (24) to node[fill=white,inner sep=1]{\scriptsize $3$} (25);
 \draw (24) to (34);
 \draw (25) to node[fill=white,inner sep=1]{\scriptsize $1$} (35);
 \draw (33) to node[fill=white,inner sep=1]{\scriptsize $2$} (34);
 \draw (34) to node[fill=white,inner sep=1]{\scriptsize $1$} (35);
 \draw[white,ultra thin] (01) to node[black]{\scriptsize $2$} (10);
 \draw[white,ultra thin] (25) to node[black]{\scriptsize $2$} (34);
\end{tikzpicture}
\]
\end{example}

\begin{definition}\cite[Definition 4.15]{MSW1}\label{defsym}
 A perfect matching $P$ of $G_{\ell_p}$ is $\g${\it -symmetric} if $P|_{H_{\ell_p}^1} \simeq P|_{H_{\ell_p}^2}$. We denote by $\bP(G_{\g^{(p)}})$ the set of $\g$-symmetric perfect matchings of $G_{\ell_p}$. We also refer to elements of $\bP(G_{\g^{(p)}})$ as perfect matchings of $G_{\g^{(p)}}$.
\end{definition}

\begin{theorem}\cite[Theorem 4.17, Lemma 12.4]{MSW1}\label{MSW1sym}
 For $P \in \bP(G_{\ell_p})$, let $\res (P)$ be a unique perfect matching of $G_{\g}$ such that $\res (P) \setminus (\res (P) \cap \{\z_1,\z_m\}) = P|_{H_{\ell_p}^1}$. Then $P|_{G_{\ell_p}^i} \simeq \res(P)$ for some $i \in \{1,2\}$. Moreover, we have
\[
  x_{\g^{(p)}}=\Phi\Biggl(\frac{1}{\cross(T,\g^{(p)})}\sum_{P \in \bP(G_{\g^{(p)}})}\ox(P)\oy(P)\Biggr),\ \ \ox(P):=\frac{x(P)}{x(\res (P))},\ \ \oy(P):=\frac{y(P)}{y(\res (P))}.
\]
\end{theorem}

\begin{example}\label{ex2s}
 For $G_{\ell_p}$ and $G_{\ell_p}^i$ in Example \ref{exloops}, their minimal matchings are
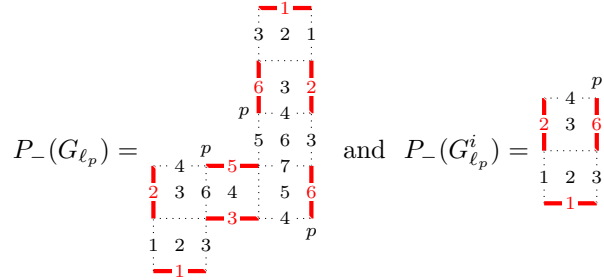
\[
P_{-}(G_{\ell_p})=

\]
 Then there are nine $\g$-symmetric perfect matchings $P$ of $G_{\ell_p}$, corresponding to nine monomials $\ox(P)\oy(P)$ as follows:
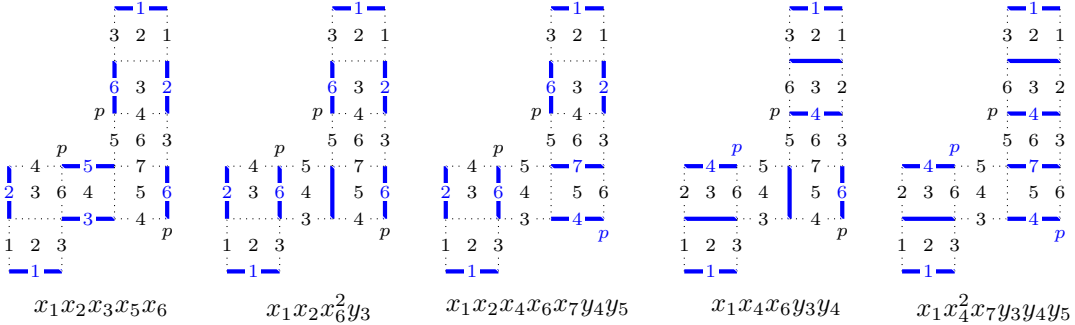
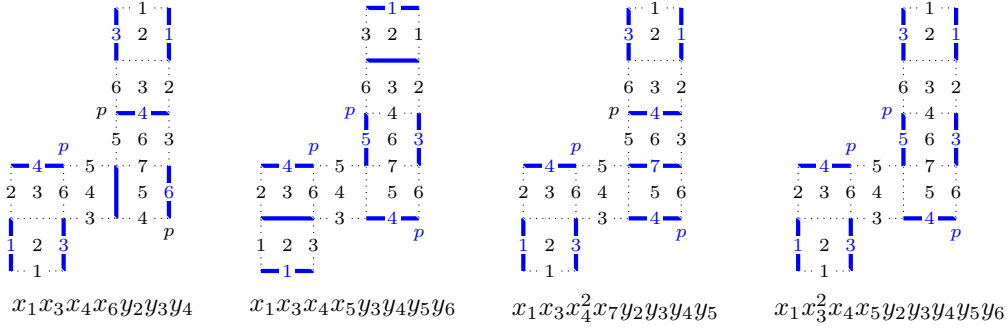
\[
%
\]
 Since these are perfect matchings of $G_{\de_2}$ for $\de_2$ given in Subsection \ref{excoefffree}, the corresponding cluster variable is
{\setlength\arraycolsep{0.5mm}
\begin{eqnarray*}
 x_{\de_2} &=& \Phi\biggl(\scalebox{0.85}{$\displaystyle\frac{1}{x_2x_3x_4x_5x_6}
\left(
     \begin{array}{ll}
   x_1x_2x_3x_5x_6+x_1x_2x_6^2y_4+x_1x_2x_4x_6x_7y_4y_5+x_1x_4x_6y_3y_4+x_1x_4^2x_7y_3y_4y_5\\
   +x_1x_3x_4x_6y_2y_3y_4+x_1x_3x_4x_5y_3y_4y_5y_6+x_1x_3x_4^2x_7y_2y_3y_4y_5+x_1x_3^2x_4x_5y_2y_3y_4y_5y_6
     \end{array} \right)
$}\biggr)\\
   &=& \scalebox{0.85}{$\displaystyle\frac{1}{x_2x_3x_4x_5x_6}
\left(
     \begin{array}{ll}
   x_1x_2x_3x_5x_6+x_1x_2x_6^2y_4+x_1x_2x_4x_6x_7y_4y_5+x_4x_6y_3y_4+x_4^2x_7y_3y_4y_5\\
   +x_3x_4x_6y_2y_3y_4+x_3x_4x_5y_3y_4y_5y_6+x_3x_4^2x_7y_2y_3y_4y_5+x_3^2x_4x_5y_2y_3y_4y_5y_6
     \end{array} \right).
$}
\end{eqnarray*}
}
\end{example}

\subsubsection{Formula for $2$-notched arcs}

 Recall the MSW formula for $x_{\g^{(pq)}}$. As above, we get ordinary loops $\ell_p$ and $\ell_q$ and the snake graphs $G_{\ell_p}$ and $G_{\ell_q}$. Note that the pair $(G_{\ell_p},G_{\ell_q})$ is denoted by $G_{\g^{(pq)}}$ in the introduction. Remark that $\g$ may be a loop. Then we denote by $\ell_p$ and $\ell_q$ the loops as in Figure \ref{looploop} although they are not ordinary loops.

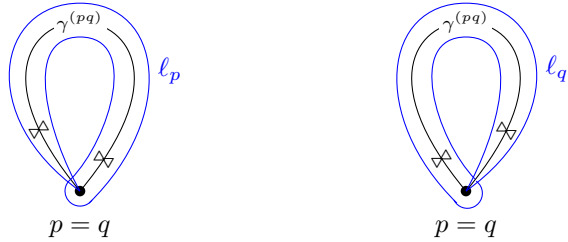
\begin{figure}[h]
   \caption{Analogues of $\ell_p$ and $\ell_q$ for a $2$-notched loop}
   \label{looploop}
\begin{tikzpicture}
 \coordinate (0) at (0,0);   \fill (0) circle (0.7mm); \node at (0,-0.5) {$p=q$};
 \draw (0) .. controls (-2.5,3) and (2.5,3) .. node[fill=white,inner sep=1]{\scriptsize $\g^{(pq)}$} node[pos=0.1]{\rotatebox{30}{\footnotesize $\bowtie$}} node[pos=0.95]{\rotatebox{150}{\footnotesize $\bowtie$}} (0);
 \draw [blue] (0) .. controls (-1.4,1) and (-1.1,2.5) .. (0,2.5);
 \draw [blue] (0,2.5) .. controls (1.1,2.5) and (1.3,1) .. node[right]{$\ell_p$} (0.12,-0.16);
 \draw [blue] (0.16,-0.12) arc (-37:-234:0.2);
 \draw [blue] (-0.12,0.16) .. controls (0.4,0.7) and (0.8,2.05) .. (0,2.05);
 \draw [blue] (0,2.05) .. controls (-0.8,2.05) and (-0.4,0.7) .. (0);
\end{tikzpicture}
\begin{tikzpicture}
 \coordinate (0) at (0,0);   \fill (0) circle (0.7mm); \node at (0,-0.5) {$p=q$};
 \draw (0) .. controls (-2.5,3) and (2.5,3) .. node[fill=white,inner sep=1]{\scriptsize $\g^{(pq)}$} node[pos=0.05]{\rotatebox{30}{\footnotesize $\bowtie$}} node[pos=0.9]{\rotatebox{150}{\footnotesize $\bowtie$}} (0);
 \draw [blue] (0) .. controls (1.4,1) and (1.1,2.5) .. node[right]{$\ell_q$} (0,2.5);
 \draw [blue] (0,2.5) .. controls (-1.1,2.5) and (-1.3,1) .. (-0.12,-0.16);
 \draw [blue] (-0.16,-0.12) arc (-153:54:0.2);
 \draw [blue] (0.12,0.16) .. controls (-0.4,0.7) and (-0.8,2.05) .. (-0,2.05);
 \draw [blue] (-0,2.05) .. controls (0.8,2.05) and (0.4,0.7) .. (0);
\end{tikzpicture}
\end{figure}

\begin{definition}\cite[Definition 4.18]{MSW1}\label{defcom}
 Let $P_p$ and $P_q$ be $\g$-symmetric perfect matchings of $G_{\ell_p}$ and $G_{\ell_q}$, respectively. The pair $(P_p,P_q)$ is $\g${\it -compatible} if $\res(P_p) \simeq \res(P_q)$. We denote by $\bP(G_{\g^{(pq)}})$ the set of $\g$-compatible pairs of $\bP(G_{\g^{(p)}}) \times \bP(G_{\g^{(q)}})$. We also refer to elements of $\bP(G_{\g^{(pq)}})$ as perfect matchings of $G_{\g^{(pq)}}$.
\end{definition}

\begin{theorem}\label{MSW1com}\cite[Theorem 4.20]{MSW1}
 We have
 \begin{equation*}
  x_{\g^{(pq)}}=\Phi\Biggl(\frac{1}{\cross(T,\g^{(pq)})}\sum_{(P_p,P_q) \in \bP(G_{\g^{(pq)}})}\oox(P_p,P_q)\ooy(P_p,P_q)\Biggr),
 \end{equation*}
 where
\[
 \oox(P_p,P_q):=\frac{x(P_p)x(P_q)}{x(\res (P_p))^3}, \hspace{3mm} \ooy(P_p,P_q):=\frac{y(P_p)y(P_q)}{y(\res (P_p))^3}.
\]
\end{theorem}

\section{Proof of Theorem \ref{main}}\label{proofofmain}

 In this section, we keep the notations of the previous sections. We prove the bijection between (1) and (2) in Theorem \ref{bijthm} and Theorem \ref{main} in the three cases of $\de=\g$, $\g^{(p)}$ and $\g^{(pq)}$. Notice that the same notations $\Phi$ and $\cross(T,\de)$ appear in Theorems \ref{main}, \ref{MSW1cef}, \ref{MSW1sym} and \ref{MSW1com}. So we only need to consider $x(A)$ and $y(A)$ for $A \in \bA(T_{\de})$. Let $A(T_{\de})$ be the set of angles incident to at least one diagonal of $T_{\de}$, and let $A_{\rm ex}(T_{\de})$ be the set of exterior angles of $T_{\de}$ which are angles between boundary segments and diagonals of $T_{\de}$. In particular, $A_{\rm ex}(T_{\de})$ is contained in $A(T_{\de})$. For a set $S$, we denote by $\#S$ the cardinality of $S$.

\subsection{The case of plain arcs}

 Recall the result of our previous paper \cite{Y}. For a plain arc $\g$, we denote by $(G_{\g})_1$ (resp., $(G_{\g})_b$) the set of edges (resp., boundary edges) of $G_{\g}$. Let $A(\overline{G}_{\g})$ be the set of angles between a diagonal $\tau_i$ and a side of the square with diagonal $\tau_i$ in $\overline{G}_{\g}$, and $\overline{\varphi} : A(\overline{G}_{\g}) \rightarrow (G_{\g})_1$ the surjective map sending $a \in A(\overline{G}_{\g})$ to the side that is opposite to $a$. By the unfolding process (see Subsection \ref{MSWformula}), there is a canonical surjection $\pi : A(\overline{G}_{\g}) \rightarrow A(T_{\g})$ compatible with the construction of $\overline{G}_{\g}$.

\begin{theorem}\label{p.m.bij}\cite[Lemma 3.2, Proposition 3.4]{Y}
 There exists a bijection $\varphi : A(T_{\g}) \rightarrow (G_{\g})_1$ making the following diagram commutative:
\[\SelectTips{eu}{} \xymatrix@C=0.3mm@R=7mm{
 & A(\overline{G}_{\g}) \ar@{->>}[ld]_{\pi} \ar@{->}[rd]^{\overline{\varphi}} & \\
 A(T_{\g}) \ar@{->}[rr]^{\sim}_{\varphi} & & (G_{\g})_1
} \]
Moreover the map $\varphi$ induces a bijection $\varphi : \bA(T_{\g}) \rightarrow \bP(G_{\g})$ satisfying $x(A)=x(\varphi(A))$ for $A \in \bA(T_{\g})$.
\end{theorem}

 Theorem \ref{p.m.bij} clearly gives the bijection between (1) and (2) in Theorem \ref{bijthm} for plain arcs. We only need to show that $y(A)=y(\varphi(A))$ for $A \in \bA(T_{\g})$ to prove Theorem \ref{main} for plain arcs.

\begin{lemma}\label{exbound}
 The restriction $\varphi|_{A_{\rm ex}(T_{\g})}$ of $\varphi$ deduces a bijection between $A_{\rm ex}(T_{\g})$ and $(G_{\g})_b$.
\end{lemma}

\begin{proof}
 The complement $A(T_{\g}) \setminus A_{\rm ex}(T_{\g})$ consists of angles $a_i$ between $\tau_i$ and $\tau_{i+1}$ for $i \in [1,n-1]$, in particular, $\#(A(T_{\g})\setminus A_{\rm ex}(T_{\g}))=n-1$. It follows from the unfolding process that $\varphi(a_i) \in (G_{\g})_1\setminus(G_{\g})_b$. Since $\#((G_{\g})_1\setminus(G_{\g})_b)=n-1$, the restriction $\varphi|_{A(T_{\g}) \setminus A_{\rm ex}(T_{\g})}$ is bijective and so is $\varphi|_{A_{\rm ex}(T_{\g})}$.
\end{proof}

\begin{proposition}\label{ybijgamma}
 For $A \in \bA(T_{\g})$, we have $Y(A)=J({\varphi(A)})$, that is $y(A)=y(\varphi(A))$.
\end{proposition}

\begin{proof}
 By Theorem \ref{p.m.bij} and Lemma \ref{exbound}, $\varphi(A_{-}(T_{\g}))$ is a perfect matching of $G_{\g}$ consisting only of boundary edges. In particular, since $e_0 \in \varphi(A_{-}(T_{\g}))$, where $e_0$ was defined in Definition \ref{e0min}, $\varphi(A_{-}(T_{\g}))=P_{-}(G_{\g})$ holds. Thus we have $\varphi(A_{-}(T_{\g}) \triangle A)=P_{-}(G_{\g}) \triangle \varphi(A)$. On the other hand, $\varphi$ maps the four angles incident to $\tau_i$ in $T_{\g}$ to sides of the square with diagonal $\tau_i$ in $\overline{G}_{\g}$. Therefore, $(A_{-}(T_{\g}) \triangle A) \cap A_{\rm ex}(T_{\g})$ contains an angle incident to $\tau_i$, which is equivalent to $\tau_i \in Y(A)$, if and only if $(P_{-}(G_{\g}) \triangle \varphi(A)) \cap (G_{\g})_b$ contains an edge of the square with diagonal $\tau_i$ in $\overline{G}_{\g}$, which is equivalent to $\tau_i \in J({\varphi(A)})$ by the definition.
\end{proof}

\begin{proof}[Proof of Theorem \ref{main} for plain arcs]
 The assertion follows from Theorems \ref{MSW1cef} and \ref{p.m.bij} and Proposition \ref{ybijgamma}.
\end{proof}

 Finally, we prepare the following lemma to use later.

\begin{lemma}\label{boundseg}
 For $A \in \bA(T_{\g})$, if $A_{-}(T_{\g}) \triangle A$ contains an exterior angle incident to $\tau_i$ in $T_{\g}$, it contains all exterior angles incident to $\tau_i$ in $T_{\g}$.
\end{lemma}

\begin{proof}
 By Theorem \ref{JPthm}, for $P \in \bP(G_{\g})$, if $P_{-}(T_{\g}) \triangle P$ contains a boundary sides of the square with diagonal $\tau_i$ in $\overline{G}_{\g}$, it contains all boundary sides of the square with diagonal $\tau_i$ in $\overline{G}_{\g}$. Since $\varphi$ maps the four angles incident to $\tau_i$ in $T_{\g}$ to sides of the square with diagonal $\tau_i$ in $\overline{G}_{\g}$, the assertion follows from Lemma \ref{exbound}.
\end{proof}

\subsection{The case of $1$-notched arcs}

 In this subsection, we show the following theorem.

\begin{theorem}\label{1-notchedbij}
 There is a bijection $\varphi_p : \bA(T_{\g^{(p)}}) \rightarrow \bP(G_{\g^{(p)}})$ satisfying $x(A)=\ox(\varphi_p(A))$ and $y(A)=\oy(\varphi_p(A))$ for $A \in \bA(T_{\g^{(p)}})$.
\end{theorem}

 Theorem \ref{1-notchedbij} clearly gives the bijection between (1) and (2) in Theorem \ref{bijthm} for $1$-notched arcs. To prove Theorem \ref{1-notchedbij}, we prepare the following notations as in Figure \ref{Tellgamma(p)}. By construction of the triangulated polygon $T_{\ell_p}$, it contains two disjoint subgraphs $T_{\ell_p}^1$ and $T_{\ell_p}^2$ with same form as $T_{\g}$, where $T_{\ell_p}^1$ has the boundary segment $\z_m$ of $T_{\ell_p}$. The subgraph $U_{\ell_p}^i$ of $T_{\ell_p}^i$ is obtained by removing the vertex $p$ and the two sides $\z_1$, $\z_m$. For $i \in \{1,2\}$, let $v_i$ (resp., $v_i'$) be the common endpoint of $\tau_n$ and $\zeta_m$ (resp., $\zeta_1$) in $T_{\ell_p}^i$. Let $a_i$ (resp., $a_i'$) be the angle at $v_i$ (resp., $v_i'$) that comes first in the counterclockwise (resp., clockwise) order around $v_i$ (resp., $v_i'$). We denote by $a_i^{\circ}$ an angle between $\tau_n$ and the boundary segment of the triangle with sides $\tau_{n-1}$ and $\tau_n$ of $T_{\ell_p}^i$. If $n>1$, it is uniquely determined, that is $a_i^{\circ}=a_i$ or $a_i^{\circ}=a_i'$.

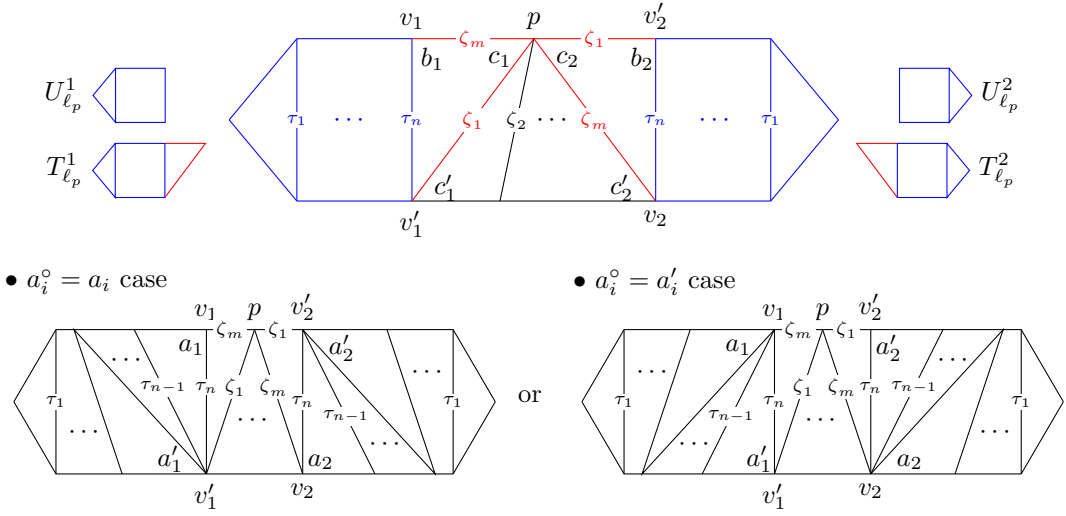
\begin{figure}[h]
   \caption{$T_{\ell_p}$ and subgraphs $T_{\ell_p}^i$ and $U_{\ell_p}^i$ of $T_{\ell_p}$}
   \label{Tellgamma(p)}
\[
\begin{tikzpicture}[scale=0.3,baseline=-20mm]
 \coordinate (l) at (-5,0); \node[left] at (l) {$U_{\ell_p}^1$};
 \coordinate (llu) at (-4,1.2);
 \coordinate (lld) at (-4,-1.2);
 \coordinate (lu) at (-1.8,1.2);
 \coordinate (ld) at (-1.8,-1.2);
 \draw[blue] (llu)--(l)--(lld)--(llu)--(lu)--(ld)--(lld);
\end{tikzpicture}
\hspace{-17mm}
\begin{tikzpicture}[scale=0.3,baseline=-10mm]
 \coordinate (l) at (-5,0); \node[left] at (l) {$T_{\ell_p}^1$};
 \coordinate (llu) at (-4,1.2);
 \coordinate (lld) at (-4,-1.2);
 \coordinate (lu) at (-1.8,1.2);
 \coordinate (ld) at (-1.8,-1.2);
 \coordinate (p) at (0,1.2);
 \draw[blue] (llu)--(l)--(lld)--(llu)--(lu)--(ld)--(lld);
 \draw[red] (lu)--(p)--(ld);
\end{tikzpicture}
\hspace{3mm}
\begin{tikzpicture}[scale=0.9]
 \coordinate (l) at (-4.5,0);
 \coordinate (llu) at (-3.5,1.2);
 \coordinate (lld) at (-3.5,-1.2);
 \coordinate (lu) at (-1.8,1.2); \node[above] at (lu) {$v_1$};
 \coordinate (ld) at (-1.8,-1.2); \node[below] at (ld) {$v_1'$};
 \coordinate (p) at (0,1.2); \node[above] at (p) {$p$};
 \coordinate (d) at (-0.5,-1.2);
 \coordinate (ru) at (1.8,1.2); \node[above] at (ru) {$v_2'$};
 \coordinate (rd) at (1.8,-1.2); \node[below] at (rd) {$v_2$};
 \coordinate (rru) at (3.5,1.2);
 \coordinate (rrd) at (3.5,-1.2);
 \coordinate (r) at (4.5,0);
 \node[blue] at (-2.7,0) {$\cdots$};
 \node[blue] at (2.7,0) {$\cdots$};
 \node at (0.3,0) {$\cdots$};
 \node at (-1.5,0.9) {$b_1$};
 \node at (1.6,0.9) {$b_2$};
 \node at (-0.5,0.9) {$c_1$};
 \node at (-1.3,-1) {$c_1'$};
 \node at (0.5,0.9) {$c_2$};
 \node at (1.3,-1) {$c_2'$};
 \draw[blue] (l) to (llu);
 \draw[blue] (l) to (lld);
 \draw[blue] (llu) to (lu);
 \draw[blue] (lld) to (ld);
 \draw (ld) to (rd);
 \draw[blue] (rd) to (rrd);
 \draw[blue] (ru) to (rru);
 \draw[blue] (r) to (rru);
 \draw[blue] (r) to (rrd);
 \draw[blue] (llu) to node[fill=white,inner sep=1]{\scriptsize $\tau_1$} (lld);
 \draw[blue] (lu) to node[fill=white,inner sep=1]{\scriptsize $\tau_n$} (ld);
 \draw[red] (lu) to node[fill=white,inner sep=1]{\scriptsize $\z_m$} (p);
 \draw[red] (p) to node[fill=white,inner sep=1]{\scriptsize $\z_1$} (ru);
 \draw[red] (p) to node[fill=white,inner sep=1]{\scriptsize $\z_1$} (ld);
 \draw (p) to node[fill=white,inner sep=1]{\scriptsize $\z_2$} (d);
 \draw[red] (p) to node[fill=white,inner sep=1]{\scriptsize $\z_m$} (rd);
 \draw[blue] (ru) to node[fill=white,inner sep=1]{\scriptsize $\tau_n$} (rd);
 \draw[blue] (rru) to node[fill=white,inner sep=1]{\scriptsize $\tau_1$} (rrd);
\end{tikzpicture}
\hspace{8mm}
\begin{tikzpicture}[scale=0.3,baseline=-20mm]
 \coordinate (ru) at (1.8,1.2);
 \coordinate (rd) at (1.8,-1.2);
 \coordinate (rru) at (4,1.2);
 \coordinate (rrd) at (4,-1.2);
 \coordinate (r) at (5,0); \node[right] at (r) {$U_{\ell_p}^2$};
 \draw[blue] (rru)--(r)--(rrd)--(rru)--(ru)--(rd)--(rrd);
\end{tikzpicture}
\hspace{-23mm}
\begin{tikzpicture}[scale=0.3,baseline=-10mm]
 \coordinate (p) at (0,1.2);
 \coordinate (ru) at (1.8,1.2);
 \coordinate (rd) at (1.8,-1.2);
 \coordinate (rru) at (4,1.2);
 \coordinate (rrd) at (4,-1.2);
 \coordinate (r) at (5,0); \node[right] at (r) {$T_{\ell_p}^2$};
 \draw[blue] (rru)--(r)--(rrd)--(rru)--(ru)--(rd)--(rrd);
 \draw[red] (ru)--(p)--(rd);
\end{tikzpicture}
\]
\[
\begin{tikzpicture}[scale=0.8,baseline=0]
 \coordinate (l) at (-4,0);
 \coordinate (llu) at (-3.3,1.2);
 \coordinate (lld) at (-3.3,-1.2);
 \coordinate (ln-1) at (-2,1.2);
 \coordinate (lu) at (-0.8,1.2); \node[above] at (lu) {$v_1$};
 \coordinate (ld) at (-0.8,-1.2); \node[below] at (ld) {$v_1'$};
 \coordinate (p) at (0,1.2); \node[above] at (p) {$p$};
 \coordinate (d) at (-0.2,-1.2);
 \coordinate (ru) at (0.8,1.2); \node[above] at (ru) {$v_2'$};
 \coordinate (rd) at (0.8,-1.2); \node[below] at (rd) {$v_2$};
 \coordinate (rn-1) at (2,-1.2);
 \coordinate (rru) at (3.3,1.2);
 \coordinate (rrd) at (3.3,-1.2);
 \coordinate (r) at (4,0);
 \coordinate (lr) at (-3,1.2);
 \coordinate (rl) at (3,-1.2);
 \coordinate (lrd) at (-2.2,-1.2);
 \coordinate (rld) at (2.2,1.2);
 \node at (-2.8,2) {\mbox{$\bullet\ a_i^{\circ}=a_i$ case}};
 \node at (-2.8,-0.5) {$\cdots$};
 \node at (2.9,0.5) {$\cdots$};
 \node at (-2.1,0.7) {$\cdots$};
 \node at (2.2,-0.7) {$\cdots$};
 \node at (0,-0.3) {$\cdots$};
 \node at (-1.05,0.9) {$a_1$};
 \node at (1.1,-1) {$a_2$};
 \node at (-1.4,-0.95) {$a_1'$};
 \node at (1.45,0.9) {$a_2'$};
 \draw (l)--(llu)--(lu);
 \draw (l)--(lld)--(ld)--(rd)--(rrd)--(r);
 \draw (r)--(rru)--(ru);
 \draw (ld)--(lr)--(lrd);
 \draw (ru)--(rl)--(rld);
 \draw (llu) to node[fill=white,inner sep=1]{\scriptsize $\tau_1$} (lld);
 \draw (ln-1) to node[fill=white,inner sep=1,pos=0.4]{\scriptsize $\tau_{n-1}$} (ld);
 \draw (lu) to node[fill=white,inner sep=1,pos=0.4]{\scriptsize $\tau_n$} (ld);
 \draw (lu) to node[fill=white,inner sep=1]{\scriptsize $\z_m$} (p);
 \draw (p) to node[fill=white,inner sep=1]{\scriptsize $\z_1$} (ru);
 \draw (p) to node[fill=white,inner sep=1,pos=0.4]{\scriptsize $\z_1$} (ld);
 \draw (p) to node[fill=white,inner sep=1,pos=0.4]{\scriptsize $\z_m$} (rd);
 \draw (ru) to node[fill=white,inner sep=1]{\scriptsize $\tau_n$} (rd);
 \draw (rn-1) to node[fill=white,inner sep=1,pos=0.4]{\scriptsize $\tau_{n-1}$} (ru);
 \draw (rru) to node[fill=white,inner sep=1]{\scriptsize $\tau_1$} (rrd);
\end{tikzpicture}
     \hspace{3mm} {\rm or} \hspace{3mm}
\begin{tikzpicture}[scale=0.8,baseline=0]
 \coordinate (l) at (-4,0);
 \coordinate (llu) at (-3.3,1.2);
 \coordinate (lld) at (-3.3,-1.2);
 \coordinate (ln-1) at (-2,-1.2);
 \coordinate (lu) at (-0.8,1.2); \node[above] at (lu) {$v_1$};
 \coordinate (ld) at (-0.8,-1.2); \node[below] at (ld) {$v_1'$};
 \coordinate (p) at (0,1.2); \node[above] at (p) {$p$};
 \coordinate (d) at (-0.2,-1.2);
 \coordinate (ru) at (0.8,1.2); \node[above] at (ru) {$v_2'$};
 \coordinate (rd) at (0.8,-1.2); \node[below] at (rd) {$v_2$};
 \coordinate (rn-1) at (2,1.2);
 \coordinate (rru) at (3.3,1.2);
 \coordinate (rrd) at (3.3,-1.2);
 \coordinate (r) at (4,0);
 \coordinate (lr) at (-3,-1.2);
 \coordinate (rl) at (3,1.2);
 \coordinate (lrd) at (-2.2,1.2);
 \coordinate (rld) at (2.2,-1.2);
 \node at (-2.8,2) {\mbox{$\bullet\ a_i^{\circ}=a_i'$ case}};
 \node at (-2.8,0.5) {$\cdots$};
 \node at (2.9,-0.5) {$\cdots$};
 \node at (-2.1,-0.7) {$\cdots$};
 \node at (2.2,0.7) {$\cdots$};
 \node at (0,-0.3) {$\cdots$};
 \node at (-1.4,0.9) {$a_1$};
 \node at (1.45,-1) {$a_2$};
 \node at (-1.05,-0.95) {$a_1'$};
 \node at (1.1,0.9) {$a_2'$};
 \draw (l)--(llu)--(lu);
 \draw (l)--(lld)--(ld)--(rd)--(rrd)--(r);
 \draw (r)--(rru)--(ru);
 \draw (lu)--(lr)--(lrd);
 \draw (rd)--(rl)--(rld);
 \draw (llu) to node[fill=white,inner sep=1]{\scriptsize $\tau_1$} (lld);
 \draw (ln-1) to node[fill=white,inner sep=1,pos=0.4]{\scriptsize $\tau_{n-1}$} (lu);
 \draw (lu) to node[fill=white,inner sep=1]{\scriptsize $\tau_n$} (ld);
 \draw (lu) to node[fill=white,inner sep=1]{\scriptsize $\z_m$} (p);
 \draw (p) to node[fill=white,inner sep=1]{\scriptsize $\z_1$} (ru);
 \draw (p) to node[fill=white,inner sep=1,pos=0.4]{\scriptsize $\z_1$} (ld);
 \draw (p) to node[fill=white,inner sep=1,pos=0.4]{\scriptsize $\z_m$} (rd);
 \draw (ru) to node[fill=white,inner sep=1,pos=0.4]{\scriptsize $\tau_n$} (rd);
 \draw (rn-1) to node[fill=white,inner sep=1,pos=0.4]{\scriptsize $\tau_{n-1}$} (rd);
 \draw (rru) to node[fill=white,inner sep=1]{\scriptsize $\tau_1$} (rrd);
\end{tikzpicture}
\]
\end{figure}

 By Theorem \ref{p.m.bij} and Proposition \ref{ybijgamma}, there exists a bijection $\varphi^p : A(T_{\ell_p}) \rightarrow (G_{\ell_p})_1$ which induces a bijection $\varphi^p : \bA(T_{\ell_p}) \rightarrow \bP(G_{\ell_p})$ satisfying $x(A)=x(\varphi^p(A))$ and $y(A)=y(\varphi^p(A))$ for $A \in \bA(T_{\ell_p})$.

\begin{lemma}\label{fourbij}
 The restrictions of $\varphi^p$ induce bijections
\[
   \varphi^p|_{A(U_{\ell_p}^i)\sqcup \{a_i^{\circ}\}} : A(U_{\ell_p}^i)\sqcup \{a_i^{\circ}\} \rightarrow (H_{\ell_p}^i)_1,\ \ \varphi^p|_{A(T_{\ell_p}^i)} : A(T_{\ell_p}^i) \rightarrow (G_{\ell_p}^i)_1
\]
 for $i \in \{1,2\}$. Moreover, the map $\varphi^p|_{A(T_{\ell_p}^i)}$ induces a bijection between $\bA(T_{\ell_p}^i)$ and $\bP(G_{\ell_p}^i)$.
\end{lemma}

\begin{proof}
 The first assertion follows immediately from the unfolding process. The second assertion follows from $T_{\ell_p}^i \simeq T_{\g}$, $G_{\ell_p}^i \simeq G_{\g}$, and Theorem \ref{p.m.bij}.
\end{proof}

\begin{definition}
 We say that $A \in \bA(T_{\ell_p})$ is $\g${\it -symmetric} if the restrictions of $A$ satisfies $A|_{A(U_{\ell_p}^1)\sqcup \{a_1^{\circ}\}} \simeq A|_{A(U_{\ell_p}^2)\sqcup \{a_2^{\circ}\}}$. We denote by $\Asym(T_{\ell_p})$ the set of $\g$-symmetric perfect matchings of angles in $T_{\ell_p}$.
\end{definition}

 Let $A \in \bA(T_{\ell_p})$. It follows from Theorem \ref{MSW1sym} and Lemma \ref{fourbij} that $A|_{A(T_{\ell_p}^i)} \in \bA(T_{\ell_p}^i)$ for some $i \in \{1,2\}$. Since it is uniquely determined up to isomorphism, we denote it by $\res(A)$.

\begin{proposition}\label{bijsym}
 The map $\varphi^p$ induces a bijection $\varphi^p : \Asym(T_{\ell_p}) \rightarrow \bP(G_{\g^{(p)}})$ satisfying $\ox(A)=\ox(\varphi^p(A))$ and $\oy(A)=\oy(\varphi^p(A))$ for $A \in \Asym(T_{\ell_p})$, where
\[
 \ox(A):=\frac{x(A)}{x(\res(A))},\ \ \oy(A):=\frac{y(A)}{y(\res(A))}.
\]
\end{proposition}

\begin{proof}
 It follows from Lemma \ref{fourbij} that $A \in \bA(T_{\ell_p})$ is $\g$-symmetric if and only if $\varphi^p(A) \in \bP(G_{\g^{(p)}})$. Since $\varphi^p$ is a bijection between $\bA(T_{\ell_p})$ and $\bP(G_{\ell_p})$, it induces a bijection between $\Asym(T_{\ell_p})$ and $\bP(G_{\g^{(p)}})$. On the other hand, Theorem \ref{p.m.bij} and Proposition \ref{ybijgamma} imply that $x(A)=x(\varphi^p(A))$ and $y(A)=y(\varphi^p(A))$ for $A \in \bA(T_{\ell_p})$, and also $x(\res(A))=x(\varphi^p(\res(A)))$ and $y(\res(A))=y(\varphi^p(\res(A)))$ for $A \in \Asym(T_{\ell_p})$ since $T_{\ell_p}^i \simeq T_{\g}$. Since $\varphi^p$ is compatible with $\res$, we have
\[
 \ox(A)=\frac{x(\varphi^p(A))}{x(\varphi^p(\res(A)))}=\frac{x(\varphi^p(A))}{x(\res(\varphi^p(A)))}=\ox(\varphi^p(A)),
\]
similarly, $\oy(A)=\oy(\varphi^p(A))$ for $A \in \Asym(T_{\ell_p})$.
\end{proof}

 All that is left is to give the following proposition for the proof of Theorem \ref{1-notchedbij}.

\begin{proposition}\label{bijsymangle}
 There is a bijection $\psi^p : \Asym(T_{\ell_p}) \rightarrow \bA(T_{\g^{(p)}})$ satisfying $\ox(A)=x(\psi^p(A))$ and $\oy(A)=y(\psi^p(A))$ for $A \in \Asym(T_{\ell_p})$.
\end{proposition}

 To prove Proposition \ref{bijsymangle}, we prepare some lemmas. We denote by $T_{\ell_p} \setminus T_{\ell_p}^2$ the subgraphs obtained from $T_{\ell_p}$ by removing $U_{\ell_p}^2$ and $\z_1$ of $T_{\ell_p}^2$. Similarly, we define the notation $T_{\ell_p} \setminus T_{\ell_p}^1$. For $i \in \{1,2\}$, let $c_i$ and $c_i'$ be the angles as in Figure \ref{Tellgamma(p)}.

\begin{lemma}\label{cc'}
 For $A \in \Asym(T_{\ell_p})$ and $i \in \{1,2\}$, $c_i \in A$ if and only if $c_i' \in A$.
\end{lemma}

\begin{proof}
 Suppose that $c_i \in A$. Since $T_{\ell_p}^i$ has $n+1$ triangles, it follows from $c_i \in A$ that $\#A|_{A(T_{\ell_p}^i)}=n$. Thus $c_i' \in A$ since $T_{\ell_p}^i$ has $n+1$ vertices incident to at least one diagonal in $T_{\ell_p}^i$. The proof of the converse assertion is similar.
\end{proof}

 For $A \in \Asym(T_{\ell_p})$, the $\g$-symmetry implies that $a_1^{\circ} \in A$ if and only if $a_2^{\circ} \in A$. It is consistent to use the notations $a_i^{\circ} \in A$ and $a_i^{\circ} \notin A$. Let $b_1$ (resp., $b_2$) be the angles as in Figure \ref{Tellgamma(p)}.

\begin{lemma}\label{lema}
 For $A \in \Asym(T_{\ell_p})$,\par
 (1) if $a_i^{\circ}=a_i \in A$ or $a_i^{\circ}=a_i' \notin A$, then $c_2, c_2' \notin A$,\par
 (2) if $a_i^{\circ}=a_i \notin A$ or $a_i^{\circ}=a_i' \in A$, then $c_1, c_1' \notin A$.\\
Moreover, $A = A|_{A(T_{\ell_p} \setminus T_{\ell_p}^j)} \sqcup A|_{A(T_{\ell_p}^j)}$ and $\res(A)=A|_{A(T_{\ell_p}^j)}$ hold for $j \in \{1,2\}$.
\end{lemma}

\begin{proof}
 If $a_i^{\circ}=a_i \in A$, then $c_2' \notin A$. If $a_i^{\circ}=a_i' \notin A$, then $b_2 \in A$, and $c_2 \notin A$. The assertion (1) follows from Lemma \ref{cc'}. Consequently, we have a decomposition $A = A|_{A(T_{\ell_p} \setminus T_{\ell_p}^2)} \sqcup A|_{A(T_{\ell_p}^2)}$. Since $\#A|_{A(T_{\ell_p}^2)}=n+1$ and $T_{\ell_p}^2$ has $n+1$ triangles, then $A|_{A(T_{\ell_p}^2)} \in \bA(T_{\ell_p}^2)$. Thus $\res(A)=A|_{A(T_{\ell_p}^2)}$ holds. The proof of (2) is similar
\end{proof}

 Next, we consider the triangulated polygon $T_{\g^{(p)}}$ with one puncture $p$. We prepare the following notations as in Figure \ref{Tgamma(p)}. Let $v$ (resp., $v'$) be the common endpoint of $\tau_n$ and $\zeta_m$ (resp., $\zeta_1$) in $T_{\g^{(p)}}$. Let $d$ (resp., $d'$) be the angle at $v$ (resp., $v'$) that comes first in the counterclockwise (resp., clockwise) order around $v$ (resp., $v'$). We denote by $d^{\circ}$ an angle between $\tau_n$ and the boundary segment of the triangle with sides $\tau_{n-1}$ and $\tau_n$ of $T_{\g^{(p)}}$. If $n>1$, it is uniquely determined, that is $d^{\circ}=d$ or $d^{\circ}=d'$. Let $e_1$ (resp., $e_2$) be the angle between $\z_1$ (resp., $\z_m$) and a boundary segment of $T_{\g^{(p)}}$.

\begin{figure}[h]
   \caption{$T_{\g^{(p)}}$}
   \label{Tgamma(p)}
\[
\begin{tikzpicture}[scale=1,baseline=0]
 \coordinate (l) at (-3.3,0);
 \coordinate (llu) at (-2.5,1.2);
 \coordinate (lld) at (-2.5,-1.2);
 \coordinate (ln-1) at (-1,1.2);
 \coordinate (lr) at (-2.1,1.2);
 \coordinate (lu) at (0,1.2); \node[above] at (lu) {$v$};
 \coordinate (ld) at (0,-1.2); \node[below] at (ld) {$v'$};
 \coordinate (p) at (1.2,0);;
 \coordinate (ru) at (2.4,1.2);
 \coordinate (rd) at (2.4,-1.2);
 \node at (-2,2) {\mbox{$\bullet\ d^{\circ}=d$ case}};
 \node at (-1.8,0) {$\cdots$};
 \node at (-1.2,0.65) {$\cdots$};
 \node at (2,0) {$\vdots$};
 \node at (0.6,-1) {$e_1$};
 \node at (1.8,-1) {$\alpha_2$};
 \node at (1,0.9) {$e_2=\alpha_m$};
 \node at (-0.2,0.9) {$d$};
 \node at (-0.55,-1) {$d'$};
 \draw (l)--(llu)--(ru)--(rd)--(lld)--(l);
 \draw (lr)--(ld);
 \draw (llu) to node[fill=white,inner sep=1]{\scriptsize $\tau_1$} (lld);
 \draw (ln-1) to node[fill=white,inner sep=1,pos=0.4]{\scriptsize $\tau_{n-1}$} (ld);
 \draw (lu) to node[fill=white,inner sep=1]{\scriptsize $\tau_n$} (ld);
 \draw (p) to node[fill=white,inner sep=1,pos=0.4]{\scriptsize $\z_1$} (ld);
 \draw (p) to node[fill=white,inner sep=1,pos=0.4]{\scriptsize $\z_2$} (rd);
 \draw (p) to node[fill=white,inner sep=1,pos=0.35]{\scriptsize $\z_{m-1}$} (ru) node[below=13,fill=white,inner sep=1]{$\alpha_{m-1}$};
 \draw (p) node[fill=white,inner sep=1]{$p$} to node[fill=white,inner sep=1,pos=0.4]{\scriptsize $\z_m$} (lu);
\end{tikzpicture}
     \hspace{5mm} {\rm or} \hspace{5mm}
\begin{tikzpicture}[scale=1,baseline=0]
 \coordinate (l) at (-3.3,0);
 \coordinate (llu) at (-2.5,1.2);
 \coordinate (lld) at (-2.5,-1.2);
 \coordinate (ln-1) at (-1,-1.2);
 \coordinate (lr) at (-2.1,-1.2);
 \coordinate (lu) at (0,1.2); \node[above] at (lu) {$v$};
 \coordinate (ld) at (0,-1.2); \node[below] at (ld) {$v'$};
 \coordinate (p) at (1.2,0);;
 \coordinate (ru) at (2.4,1.2);
 \coordinate (rd) at (2.4,-1.2);
 \node at (-2,2) {\mbox{$\bullet\ d^{\circ}=d'$ case}};
 \node at (-1.8,0) {$\cdots$};
 \node at (-1.2,-0.65) {$\cdots$};
 \node at (2,0) {$\vdots$};
 \node at (0.6,-1) {$e_1$};
 \node at (1.8,-1) {$\alpha_2$};
 \node at (1,0.9) {$e_2=\alpha_m$};
 \node at (-0.6,0.9) {$d$};
 \node at (-0.2,-0.95) {$d'$};
 \draw (l)--(llu)--(ru)--(rd)--(lld)--(l);
 \draw (lr)--(lu);
 \draw (llu) to node[fill=white,inner sep=1]{\scriptsize $\tau_1$} (lld);
 \draw (ln-1) to node[fill=white,inner sep=1,pos=0.4]{\scriptsize $\tau_{n-1}$} (lu);
 \draw (lu) to node[fill=white,inner sep=1]{\scriptsize $\tau_n$} (ld);
 \draw (p) to node[fill=white,inner sep=1,pos=0.4]{\scriptsize $\z_1$} (ld);
 \draw (p) to node[fill=white,inner sep=1,pos=0.4]{\scriptsize $\z_2$} (rd);
 \draw (p) to node[fill=white,inner sep=1,pos=0.35]{\scriptsize $\z_{m-1}$} (ru) node[below=13,fill=white,inner sep=1]{$\alpha_{m-1}$};
 \draw (p) node[fill=white,inner sep=1]{$p$} to node[fill=white,inner sep=1,pos=0.4]{\scriptsize $\z_m$} (lu);
\end{tikzpicture}
\]
\end{figure}

\begin{lemma}\label{note1} For $A \in \bA(T_{\g^{(p)}})$,\par
 (1) if $d^{\circ}=d \in A$ or $d^{\circ}=d' \notin A$, then $e_2 \notin A$,\par
 (2) if $d^{\circ}=d \notin A$ or $d^{\circ}=d' \in A$, then $e_1 \notin A$.
\end{lemma}

\begin{proof}
 We only prove (1) since the proof of (2) is similar. Suppose that $e_2 \in A$. For $k \in [2,m]$, we denote by $\alpha_k$ the angle between $\zeta_k$ and the boundary segment of the triangle with sides $\zeta_{k-1}$ and $\zeta_k$. An easy induction shows that $\alpha_k \in A$ for all $k \in [2,m]$ since $\alpha_m=e_2 \in A$. Thus $A$ has the angle between $\z_1$ and $\z_m$, and $d^{\circ}=d \notin A$ or $d^{\circ}=d' \in A$ follows easily.
\end{proof}

 The graph $T_{\g^{(p)}}$ is obtained from $T_{\ell_p} \setminus T_{\ell_p}^2$ by identifying the two edges $\z_m$ along the direction from $p$ to the other endpoint of $\z_m$. Similarly, it is also obtained from $T_{\ell_p} \setminus T_{\ell_p}^1$ by identifying the two edges $\z_1$ from $p$ to the other endpoint of $\z_1$. These constructions induce bijections
\[
g_1 : A(T_{\ell_p} \setminus T_{\ell_p}^2) \rightarrow A(T_{\g^{(p)}}) \setminus \{e_2\}\ \ \mbox{and} \ \ g_2 : A(T_{\ell_p} \setminus T_{\ell_p}^1) \rightarrow A(T_{\g^{(p)}}) \setminus \{e_1\}
\]
 such that $g_1(a_1^{\circ})=d^{\circ}=g_2(a_2^{\circ})$ holds. In particular, for $\{i,j\}=\{1,2\}$ and $A \in \bA(T_{\ell_p} \setminus T_{\ell_p}^j)$, we also have $g_i(A) \in \bA(T_{\g^{(p)}})$ and $x(A)=x(g_i(A))$. Moreover, there are bijections
\begin{equation}\label{Aexg2}
 A_{\rm ex}(T_{\ell_p} \setminus T_{\ell_p}^2) \setminus \{b_1,c_1,\mbox{the angle between $\z_{m-1}$ and $\z_m$}\} \sqcup \{c_2'\} \xrightarrow[]{\sim} A_{\rm ex}(T_{\g^{(p)}})
\end{equation}
given by $a \mapsto g_1(a)$ if $a \neq c_2'$ and $c_2' \mapsto e_2$, and
\begin{equation}\label{Aexg1}
 A_{\rm ex}(T_{\ell_p} \setminus T_{\ell_p}^1) \setminus \{b_2,c_2,\mbox{the angle between $\z_1$ and $\z_2$}\} \sqcup \{c_1'\} \xrightarrow[]{\sim} A_{\rm ex}(T_{\g^{(p)}})
\end{equation}
given by $a \mapsto g_2(a)$ if $a \neq c_1'$ and $c_1' \mapsto e_1$. Finally, we give one lemma for a general $\de$. For $k \in [1,n]$, let $T_{\de}^{-;k}$ and $T_{\de}^{+;k}$ be the two subpolygons of $T_{\de}$ obtained by cutting $T_{\de}$ along $\tau_k$, where $T_{\de}^{-;k}$ contains $q$. We denote by $A'(T_{\de}^{\pm;k})$ the restriction $A(T_{\de})|_{T_{\de}^{\pm;k}}$. We also define that $T_{\de}^{-;n+1}$ (resp., $T_{\de}^{+;0}$) is the subgraph obtained from $T_{\de}^{-;n}$ (resp., $T_{\de}^{+;1}$) by adding the triangle with sides $\tau_n$, $\zeta_1$ and $\zeta_m$ (resp., $\tau_1$, $\xi_1$ and $\xi_{\ell}$).

\begin{lemma}
 For $A \in \bA(T_{\de})$, there is a unique completion $C_{\tau_k}(A|_{A'(T_{\de}^{\pm;k})}) \in \bA(T_{\de}^{\pm;k\mp1})$ containing $A|_{A'(T_{\de}^{\pm;k})}$.
\end{lemma}

\begin{proof}
 Since the equality
\[
 \#A|_{A'(T_{\de}^{\pm;k})}=\#\{\mbox{triangles of }T_{\de}^{\pm;k}\}=\#\{\mbox{vertices of }T_{\de}^{\pm;k}\mbox{ incident to at least one diagonal}\}
\]
holds, there is exactly one endpoint $v$ of $\tau_k$ such that $A|_{A'(T_{\de}^{\pm;k})}$ has no angle incident to $v$. Therefore, there is exactly one angle $a_v$ of $A(T_{\de}^{\pm;k\mp1}) \setminus A'(T_{\de}^{\pm;k})$ incident to $v$, and we have a unique completion $C_{\tau_k}(A|_{A'(T_{\de}^{\pm;k})})=A|_{A'(T_{\de}^{\pm;k})}\sqcup\{a_v\} \in \bA(T_{\de}^{\pm;k\mp1})$.
\end{proof}

 For $\{i,j\}=\{1,2\}$ and $A \in \bA(T_{\ell_p} \setminus T_{\ell_p}^i)$, there exists a unique symmetric completion $\overline{A} \in \Asym(T_{\ell_p})$ of $A$, that is $\overline{A}|_{A(T_{\ell_p} \setminus T_{\ell_p}^i) \sqcup \{c_i'\}} = A$ and $\overline{A}|_{A(T_{\ell_p}^i) \sqcup \{c_i\}} \simeq C_{\tau_n}(A|_{A(U_{\ell_p}^j) \sqcup \{a_j^{\circ}\}})$. We are ready to prove Proposition \ref{bijsymangle}.

\begin{proof}[Proof of Proposition \ref{bijsymangle}]
 By Lemma \ref{lema}, we can define the map $\psi^p : \Asym(T_{\ell_p}) \rightarrow \bA(T_{\g^{(p)}})$ by
\begin{eqnarray}\label{psi}
   \Asym(T_{\ell_p}) \hspace{1mm}\mbox{\reflectbox{$\in$}}\hspace{1mm} A \mapsto \left\{
     \begin{array}{ll}
 g_1(A|_{A(T_{\ell_p} \setminus T_{\ell_p}^2)}) & \mbox{if $a_i^{\circ}=a_i \in A$ or $a_i^{\circ}=a_i' \notin A$,}\\
 g_2(A|_{A(T_{\ell_p} \setminus T_{\ell_p}^1)}) & \mbox{if $a_i^{\circ}=a_i \notin A$ or $a_i^{\circ}=a_i' \in A$.}
     \end{array} \right.
\end{eqnarray}

 We show that $\psi^p$ is injective. Let $A, A' \in \Asym(T_{\ell_p})$ satisfying $A \neq A'$. In particular, the $\g$-symmetry implies that $A|_{A(T_{\ell_p} \setminus T_{\ell_p}^i)} \neq A'|_{A(T_{\ell_p} \setminus T_{\ell_p}^i)}$. If $a_i^{\circ} \in A \cap A'$ or $a_i^{\circ} \notin A \cup A'$, then $\psi^p(A) \neq \psi^p(A')$ follows from (\ref{psi}). Suppose that $a_i^{\circ} \in A$ and $a_i^{\circ} \notin A'$. Then $d^{\circ}=g_i(a_i^{\circ}) \in g_i(A) = \psi^p(A)$ and $d^{\circ}=g_j(a_j^{\circ})\notin g_j(A') = \psi^p(A')$ for $j \in \{1,2\}\setminus\{i\}$. Thus $\psi^p(A) \neq \psi^p(A')$ holds, that is $\psi^p$ is injective.

 We show that $\psi^p$ is surjective. Let $B \in \bA(T_{\g^{(p)}})$. If $d^{\circ}=d \in A$ or $d^{\circ}=d' \notin A$, then $B \subseteq A(T_{\g^{(p)}}) \setminus \{e_2\}$ by Lemma \ref{note1}(1). Thus $g^{-1}_1(B) \subseteq \bA(T_{\ell_p} \setminus T_{\ell_p}^2)$. There is the symmetric completion $\overline{g^{-1}_1(B)} \in \Asym(T_{\ell_p})$ such that $\psi^p\Bigl(\overline{g^{-1}_1(B)}\Bigr)=B$. If $d^{\circ}=d \notin A$ or $d^{\circ}=d' \in A$, then $e_1 \notin B$ by Lemma \ref{note1}(2). In the same way as above, there is $\overline{g^{-1}_2(B)} \in \Asym(T_{\ell_p})$ such that $\psi^p\Bigl(\overline{g^{-1}_2(B)}\Bigr)=B$. Therefore, $\psi^p$ is surjective.

 Let $A \in \Asym(T_{\ell_p})$. Since there is at least one $i \in \{1,2\}$ such that $c_i, c_i' \notin A$ by Lemma \ref{lema}, we have
\[
 \ox(A)=\frac{x(A)}{x(\res(A))}=\frac{x(A)}{x(A|_{A(T_{\ell_p}^i)})}=x(A|_{A(T_{\ell_p} \setminus T_{\ell_p}^i)})=x(\psi^p(A)).
\]

 We only need to prove $Y(A) \setminus Y(\res(A))=Y(\psi^p(A))$ to give $\oy(A)=y(\psi^p(A))$. Suppose that $a_i^{\circ}=a_i \in A$ or $a_i^{\circ}=a_i' \notin A$. From $A_-(T_{\ell_p})=A_-(T_{\ell_p} \setminus T_{\ell_p}^2) \sqcup A_-(T_{\ell_p}^2)$ and $c_2, c_2' \notin A$, we get a decomposition
{\setlength\arraycolsep{0.5mm}
\begin{eqnarray*}
 A_-(T_{\ell_p}) \triangle A &=& (A_-(T_{\ell_p}) \triangle A)|_{A(T_{\ell_p} \setminus T_{\ell_p}^2)} \sqcup (A_-(T_{\ell_p}) \triangle A)|_{A(T_{\ell_p}^2)}\\
 &=& (A_-(T_{\ell_p} \setminus T_{\ell_p}^2) \triangle A|_{A(T_{\ell_p} \setminus T_{\ell_p}^2)}) \sqcup (A_-(T_{\ell_p}^2) \triangle A|_{A(T_{\ell_p}^2)}).
\end{eqnarray*}}Thus we have
\begin{equation}\label{YAYres}
 Y(A)=Y(A|_{A(T_{\ell_p} \setminus T_{\ell_p}^2)}) \sqcup Y(A|_{A(T_{\ell_p}^2)})=Y(A|_{A(T_{\ell_p} \setminus T_{\ell_p}^2)}) \sqcup Y(\res(A)),
\end{equation}
where the second equality holds by Lemma \ref{lema}(1). On the other hand, the equalities
{\setlength\arraycolsep{0.5mm}
\begin{eqnarray}\label{ominuseq}
 g_1(A_-(T_{\ell_p} \setminus T_{\ell_p}^2) \triangle A|_{A(T_{\ell_p} \setminus T_{\ell_p}^2)})&=&g_1((A_-(T_{\ell_p}) \triangle A)|_{A(T_{\ell_p} \setminus T_{\ell_p}^2)})\nonumber\\
     &=&\psi^p(A_{-}(T_{\ell_p}) \triangle A)=A_{-}(T_{\g^{(p)}}) \triangle \psi^p(A)
\end{eqnarray}}hold by (\ref{psi}) and $\psi^p(A_{-}(T_{\ell_p}))=g_1(A_{-}(T_{\ell_p} \setminus T_{\ell_p}^2))=A_{-}(T_{\g^{(p)}})$. Therefore, it follows from Lemma \ref{boundseg} that $A_-(T_{\ell_p} \setminus T_{\ell_p}^2) \triangle A|_{A(T_{\ell_p} \setminus T_{\ell_p}^2)}$ contains $b_1$ (resp., $c_1$, the angle between $\z_{m-1}$ and $\z_m$) if and only if it contains $a_1$ (resp., $c_1'$, the angles between $\z_{m-1}$ and boundary segments). Thus we have $Y(A|_{A(T_{\ell_p} \setminus T_{\ell_p}^2)}) = Y(\psi^p(A))$ by (\ref{Aexg2}) and (\ref{ominuseq}). Consequently, we have
\[
 Y(A) \setminus Y(\res(A)) = Y(A|_{A(T_{\ell_p} \setminus T_{\ell_p}^2)}) = Y(\psi^p(A))
\]
 by (\ref{YAYres}).

 Suppose that $a_i^{\circ}=a_i \notin A$ or $a_i^{\circ}=a_i' \in A$. Since $c_1, c_1' \in A_-(T_{\ell_p}) \triangle A$ by Lemma \ref{lema}(2), then $\z_1 \in Y(A)$. We also have $\z_1 \in Y(\psi^p(A))$ since $e_1 \in A_-(T_{\g^{(p)}})$ and $e_1 \notin \psi^p(A)$ by Lemma \ref{note1}(2). Since we have the equalities
\[
 A_-(T_{\ell_p}^1)=A_-(T_{\ell_p})|_{A(T_{\ell_p}^1)} \sqcup \{\mbox{the angle between $\tau_n$ and $\z_1$}\},
\]
\[
 A_-(T_{\ell_p} \setminus T_{\ell_p}^1)=A_-(T_{\ell_p})|_{A(T_{\ell_p} \setminus T_{\ell_p}^1)} \sqcup \{\mbox{the angle between $\z_1$ and $\z_2$}\},
\]
 then the equalities
{\setlength\arraycolsep{0.5mm}
\begin{eqnarray*}
 Y(A) &=& Y(A)|_{T_{\ell_p}^1} \sqcup \{\z_1\} \sqcup Y(A)|_{T_{\ell_p} \setminus T_{\ell_p}^1}\\
 &=& Y(A|_{A(T_{\ell_p}^1)}) \sqcup \{\z_1\} \sqcup Y(A|_{A(T_{\ell_p} \setminus T_{\ell_p}^1)})\\
 &=& Y(\res(A)) \sqcup \{\z_1\} \sqcup Y(A|_{A(T_{\ell_p} \setminus T_{\ell_p}^1)}).
\end{eqnarray*}}hold by Lemma \ref{boundseg} and Lemma \ref{lema}. In the same way as above proof, we have $Y(A|_{A(T_{\ell_p} \setminus T_{\ell_p}^1)}) = Y(\psi^p(A)) \setminus \{\z_1\}$ by (\ref{Aexg1}). Consequently, we have 
\[
 Y(A) \setminus Y(\res(A)) = Y(A|_{A(T_{\ell_p} \setminus T_{\ell_p}^1)}) \sqcup \{\z_1\} = Y(\psi^p(A)).
\]
This finishes the proof.
\end{proof}

\begin{proof}[Proof of Theorem \ref{1-notchedbij}]
 By Propositions \ref{bijsym} and \ref{bijsymangle}, there is a bijection $\varphi_p=\varphi^p(\psi^p)^{-1} : \bA(T_{\g^{(p)}}) \rightarrow \bP(G_{\g^{(p)}})$ satisfying
\[
 x(A)=\ox((\psi^p)^{-1}(A))=\ox(\varphi^p(\psi^p)^{-1}(A))\ \ \mbox{and}\ \ y(A)=\oy((\psi^p)^{-1}(A))=\oy(\varphi^p(\psi^p)^{-1}(A))
\]
 for $A \in \bA(T_{\g^{(p)}})$.
\end{proof}

\begin{proof}[Proof of Theorem \ref{main} for $1$-notched arcs]
 The assertion follows immediately from Theorems \ref{MSW1sym} and \ref{1-notchedbij}.
\end{proof}

\subsection{The case of $2$-notched arcs}

 In this subsection, we show the following theorem.

\begin{theorem}\label{2-notchedbij}
 There is a bijection $\varphi_{pq} : \bA(T_{\g^{(pq)}}) \rightarrow \bP(G_{\g^{(pq)}})$ satisfying $x(A)=\oox(\varphi_{pq}(A))$ and $y(A)=\ooy(\varphi_{pq}(A))$ for $A \in \bA(T_{\g^{(pq)}})$.
\end{theorem}

 Theorem \ref{2-notchedbij} clearly gives the bijection between (1) and (2) in Theorem \ref{bijthm} for $2$-notched arcs. To prove Theorem \ref{2-notchedbij}, we prepare the following notations as in Figure \ref{decomp}. For $\de=\g$, $\g^{(p)}$, $\g^{(q)}$, or $\g^{(pq)}$, there are three subpolygons $T_{\de}^q := T_{\de}^{-;1}$, $T_{\de}^c := T_{\de}^{+;1} \cap T_{\de}^{-;n}$ and $T_{\de}^p := T_{\de}^{+;n}$ of $T_{\de}$. We denote by $T_{\de}^{\ast\star}$ the subpolygon $T_{\de}^{\ast} \cup T_{\de}^{\star}$ of $T_{\de}$ for $\ast, \star \in \{q,c,p\}$. We have a decomposition $A(T_{\de})=A(T_{\de})^q \sqcup A(T_{\de})^c \sqcup A(T_{\de})^p$, where $A(T_{\de})^{\ast}$ consists of angles contained in $T_{\de}^{\ast}$ for $\ast \in \{q,c,p\}$. For $A \in \bA(T_{\de})$, we define a decomposition $A=A^q \sqcup A^c \sqcup A^p$, where $A^{\ast} \in A(T_{\de})^{\ast}$ for $\ast \in \{q,c,p\}$. For an arbitrary decomposition $S=S^q \sqcup S^c \sqcup S^p$ as above, we use the notations $S^{\ast\star}:=S^{\ast} \sqcup S^{\star}$ for $\ast, \star \in \{q,c,p\}$.

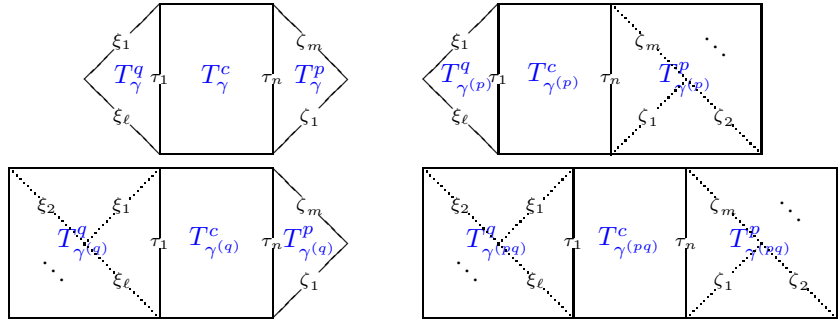
\begin{figure}[h]
   \caption{The decompositions of $T_{\g}$, $T_{\g^{(p)}}$, $T_{\g^{(q)}}$, and $T_{\g^{(pq)}}$}
   \label{decomp}
\[
 \begin{xy}
   (0,0)="q", +(10,10)="0", +(15,0)="1", +(0,-20)="2", +(-15,0)="3", "1"+(10,-10)="p", "0"+(7.5,-10)*{\color{blue}T_{\g}^c}, "p"+(-5,0)*{\color{blue}T_{\g}^p}, "q"+(6,0)*{\color{blue}T_{\g}^q}
   \ar@{-}"0";"1" \ar@{-}"2";"3"
   \ar@{-}"0";"3"|{\tau_1} \ar@{-}"1";"2"|{\tau_n}
   \ar@{-}"q";"0"|{\xi_1} \ar@{-}"q";"3"|{\xi_{\ell}}
   \ar@{-}"p";"2"|{\z_1} \ar@{-}"p";"1"|{\z_m}
 \end{xy}
\hspace{10mm}
 \begin{xy}
   (0,0)="q", +(10,10)="0", +(15,0)="1", +(0,-20)="2", +(-15,0)="3", "1"+(10,-10)="p", "1"+(20,0)="ru", +(0,-20)="rd", "p"+(4,4)*{\rotatebox{-45}{$\cdots$}}, "0"+(7.5,-10)*{\color{blue}T_{\g^{(p)}}^c}, "p"+(0,0)*{\color{blue}T_{\g^{(p)}}^p}, "q"+(6,0)*{\color{blue}T_{\g^{(p)}}^q}
   \ar@{-}"0";"ru" \ar@{-}"ru";"rd" \ar@{-}"rd";"3"
   \ar@{-}"0";"3"|{\tau_1} \ar@{-}"1";"2"|{\tau_n}
   \ar@{-}"q";"0"|{\xi_1} \ar@{-}"q";"3"|{\xi_{\ell}}
   \ar@{.}"p";"2"|{\z_1} \ar@{.}"p";"rd"|{\z_2} \ar@{.}"p";"1"|{\z_m}
 \end{xy}
\]
\[
 \begin{xy}
   (0,0)="q", +(10,10)="0", +(15,0)="1", +(0,-20)="2", +(-15,0)="3", "1"+(10,-10)="p", "0"+(-20,0)="lu", +(0,-20)="ld", "q"+(-4,-4)*{\rotatebox{-45}{$\cdots$}}, "0"+(7.5,-10)*{\color{blue}T_{\g^{(q)}}^c}, "p"+(-5,0)*{\color{blue}T_{\g^{(q)}}^p}, "q"+(0,0)*{\color{blue}T_{\g^{(q)}}^q}
   \ar@{-}"lu";"1" \ar@{-}"2";"ld" \ar@{-}"lu";"ld"
   \ar@{-}"0";"3"|{\tau_1} \ar@{-}"1";"2"|{\tau_n}
   \ar@{.}"q";"0"|{\xi_1} \ar@{.}"q";"lu"|{\xi_2} \ar@{.}"q";"3"|{\xi_{\ell}}
   \ar@{-}"p";"2"|{\z_1} \ar@{-}"p";"1"|{\z_m}
 \end{xy}
\hspace{10mm}
 \begin{xy}
   (0,0)="q", +(10,10)="0", +(15,0)="1", +(0,-20)="2", +(-15,0)="3", "1"+(10,-10)="p", "1"+(20,0)="ru", +(0,-20)="rd", "0"+(-20,0)="lu", +(0,-20)="ld", "p"+(4,4)*{\rotatebox{-45}{$\cdots$}}, "q"+(-4,-4)*{\rotatebox{-45}{$\cdots$}}, "0"+(7.5,-10)*{\color{blue}T_{\g^{(pq)}}^c}, "p"+(0,0)*{\color{blue}T_{\g^{(pq)}}^p}, "q"+(0,0)*{\color{blue}T_{\g^{(pq)}}^q}
   \ar@{-}"lu";"ru" \ar@{-}"rd";"ld" \ar@{-}"lu";"ld" \ar@{-}"ru";"rd"
   \ar@{-}"0";"3"|{\tau_1} \ar@{-}"1";"2"|{\tau_n}
   \ar@{.}"q";"0"|{\xi_1} \ar@{.}"q";"lu"|{\xi_2} \ar@{.}"q";"3"|{\xi_{\ell}}
   \ar@{.}"p";"2"|{\z_1} \ar@{.}"p";"rd"|{\z_2} \ar@{.}"p";"1"|{\z_m}
 \end{xy}
\]
\end{figure}

 Since there is the natural inclusion from $T_{\g}$ (resp., $T_{\g^{(p)}}$, $T_{\g^{(q)}}$) to $T_{\g^{(pq)}}$, we can view $T_{\g}$ (resp., $T_{\g^{(p)}}$, $T_{\g^{(q)}}$) as a subpolygon of $T_{\g^{(pq)}}$, and $A(T_{\g})$ (resp., $A(T_{\g^{(p)}})$, $A(T_{\g^{(q)}})$) as a subset of $A(T_{\g^{(pq)}})$.

\begin{definition}
 The pair $(A_p,A_q) \in \bA(T_{\g^{(p)}}) \times \bA(T_{\g^{(q)}})$ is called $\g${\it -compatible} if $A_p^c=A_q^c$ and $A_p^q \sqcup A_q^{cp} \in \bA(T_{\g})$, where we view $A_p^q \sqcup A_q^{cp}$ as a subset of $A(T_{\g})$. We denote by $\Acom(T_{\g^{(p)}}, T_{\g^{(q)}})$ the set of $\g$-compatible pairs of $\bA(T_{\g^{(p)}}) \times \bA(T_{\g^{(q)}})$.
\end{definition}

\begin{lemma}\label{defgsym}
 If $n=1$, $(A_p,A_q) \in \Acom(T_{\g^{(p)}}, T_{\g^{(q)}})$ if and only if $A_p^q \sqcup A_q^{p} \in \bA(T_{\g})$. If $n>1$, $(A_p,A_q) \in \Acom(T_{\g^{(p)}}, T_{\g^{(q)}})$ if and only if $A_p^c=A_q^c$.
\end{lemma}

\begin{proof}
 If $n=1$, the assertion follows from $A_p^c=\emptyset=A_q^c$. Suppose $n>1$ and $A_p^c=A_q^c$. Since $A_p$ and $A_q$ satisfy the condition (2) in Definition \ref{defpm}, so does $A_p^q \sqcup A_q^{cp}$. Therefore, we only show that $A_p^q \sqcup A_q^{cp}$ satisfies the condition (1) in Definition \ref{defpm} on $T_{\g}$, which is equivalent that any two distinct angles $a$ and $b$ in $A_p^q \sqcup A_q^{cp}$ are not incident to a common vertex. If $a, b \in A_p^{qc}$ or $a, b \in A_q^{cp}$, the assertion holds since $A_p^{qc} \subset A_p$, $A_q^{cp} \subset A_q$, and $A_p$ and $A_q$ satisfy the condition (1) in Definition \ref{defpm}. Suppose that $a \in A_p^q$ and $b \in A_q^p$ are incident to a common vertex. Then $\tau_1,\ldots,\tau_n$ must be incident to the vertex. Since $A_p \in \bA(T_{\g^{(p)}})$, $A_p^c$ contains the angle between $\tau_i$ and a boundary segment of the triangle with sides $\tau_i$ and $\tau_{i+1}$ for $i \in [1,n-1]$. Similarly, since $A_q \in \bA(T_{\g^{(q)}})$, $A_q^c$ contains the angle between $\tau_i$ and a boundary segment of the triangle with sides $\tau_{i-1}$ and $\tau_i$ for $i \in [2,n]$. It contradicts $A_p^c=A_q^c$. Thus the assertion holds.
\end{proof}

 We define the maps $\sr : \Acom(T_{\g^{(p)}},T_{\g^{(q)}}) \rightarrow \{\mbox{subsets of } A(T_{\g})\}$ and $\si : \Acom(T_{\g^{(p)}},T_{\g^{(q)}}) \rightarrow \{\mbox{subsets of } A(T_{\g^{(pq)}})\}$ by
\[
 \sr(A_p,A_q)=A_p^q \sqcup A_q^{cp},\ \ \si(A_p,A_q)=A_q^q \sqcup A_p^{cp}
\]
for $(A_p,A_q)\in\Acom(T_{\g^{(p)}},T_{\g^{(q)}})$.

\begin{lemma}\label{srsi}
 For $(A_p,A_q)\in\Acom(T_{\g^{(p)}},T_{\g^{(q)}})$, $\sr(A_p,A_q) \in \bA(T_{\g})$ and $\si(A_p,A_q) \in \bA(T_{\g^{(pq)}})$ hold.
\end{lemma}

\begin{proof}
 By the $\g$-compatibility, $\sr=A_p^q \sqcup A_q^{cp} \in \bA(T_{\g})$. If $n>1$, in the same as the proof of Lemma \ref{defgsym}, $\si(A_p,A_q) \in \bA(T_{\g^{(pq)}})$ holds. Suppose that $n=1$. If $\si(A_p,A_q) \notin \bA(T_{\g^{(pq)}})$, each of $A_q^q$ and $A_p^p$ has an angle incident to one endpoint of $\tau_1$. Thus each of $A_p^q$ and $A_q^p$ must have an angle incident to the other endpoint of $\tau_1$, so it contradicts $A_p^q \sqcup A_q^p \in \bA(T_{\de})$.
\end{proof}

\begin{lemma}\label{n=1case}
 Let $n=1$ and $A=(A_p,A_q) \in \Acom(T_{\g^{(p)}},T_{\g^{(q)}})$. Then the following conditions are equivalent:\par
 (1) $\tau_1 \in Y(A_p)$, \hspace{1mm} (2) $\tau_1 \in Y(A_q)$, \hspace{1mm} (3) $\tau_1 \in Y(\sr(A))$, \hspace{1mm} (4) $\tau_1 \in Y(\si(A))$.
\end{lemma}

\begin{proof}
 In this case, $\sr(A) = A_p^q \sqcup A_q^p$ has exactly two angles. Each of the conditions (1)-(3) is equivalent that the angle between $\tau_1$ and $\xi_{\ell}$ is contained in $A_p$. Moreover, it is equivalent that $A_p$ contains either the angle between $\tau_1$ and $\z_m$ or the angle between $\z_m$ and a boundary segment of $T_{\g^{(pq)}}$, that is, the condition (4) holds. Therefore, the conditions (1)-(4) are equivalent.
\end{proof}

\begin{proposition}\label{bijcomangle}
 The map $\si$ is a bijection between $\Acom(T_{\g^{(p)}},T_{\g^{(q)}})$ and $\bA(T_{\g^{(pq)}})$ satisfying $\oox(A)=x(\si(A))$ and $\ooy(A)=y(\si(A))$ for $A=(A_p,A_q) \in \Acom(T_{\g^{(p)}},T_{\g^{(q)}})$, where
\[
 \oox(A):=\frac{x(A_p)x(A_q)}{x(\sr(A))},\ \ \ooy(A):=\frac{y(A_p)y(A_q)}{y(\sr(A))}.
\]
\end{proposition}

\begin{proof}
 First of all, we construct the inverse map of $\si$. Let $B \in \bA(T_{\g^{(pq)}})$. If $n>1$, $C_{\tau_1}(B^{cp})^c=B^c=C_{\tau_n}(B^{qc})^c$ holds. If $n=1$, then $C_{\tau_1}(B^{p})^q \sqcup C_{\tau_n}(B^{q})^{p} \in \bA(T_{\g})$ holds by the proof of Lemma \ref{srsi}. Thus $(C_{\tau_1}(B^{cp}),C_{\tau_n}(B^{qc})) \in \Acom(T_{\g^{(p)}},T_{\g^{(q)}})$ by Lemma \ref{defgsym}. We define the map $\w : \bA(T_{\g^{(pq)}}) \rightarrow \Acom(T_{\g^{(p)}},T_{\g^{(q)}})$ by $\w(B)=(C_{\tau_1}(B^{cp}),C_{\tau_n}(B^{qc}))$. Then it is easy to show that $\w\si$ and $\si\w$ are identities. Thus $\si : \Acom(T_{\g^{(p)}},T_{\g^{(q)}}) \rightarrow \bA(T_{\g^{(pq)}})$ is a bijection.

 We have
\[
 \oox(A)=\frac{x(A_p)x(A_q)}{x(A_p^q)x(A_q^{cp})}=x(A_p^{cp})x(A_q^q)=x(\si(A)).
\]

 We only need to prove $Y(A_p) \sqcup Y(A_q) = Y(\si(A)) \sqcup Y(\sr(A))$, possibly with multiple elements, to give $\ooy(A)=y(\si(A))$. Suppose that $n > 1$. By Lemma \ref{boundseg}, $\tau_i \in Y(\si(A))$ (resp., $Y(A_p)$, $Y(A_q)$, $Y(\sr(A))$) if and only if there is at least one exterior angle incident to $\tau_i$ in $(A_-(T_{\g^{(pq)}}) \triangle \si(A))^c$ (resp., $(A_-(T_{\g^{(p)}}) \triangle A_p)^c$, $(A_-(T_{\g^{(q)}}) \triangle A_q)^c$, $(A_-(T_{\g}) \triangle \sr(A))^c$). On the other hand, we have the equalities
\[
 A_-(T_{\g^{(pq)}})^c = A_-(T_{\g^{(p)}})^c = A_-(T_{\g^{(q)}})^c = A_-(T_{\g})^c \ \ \mbox{and} \ \ \si(A)^c = A_p^c = A_q^c = \sr(A)^c.
\]
 Then $\tau_i \in Y(\si(A))$ (resp., $\tau_i \in Y(\sr(A))$) if and only if $\tau_i \in Y(A_p)$ (resp., $\tau_i \in Y(A_q)$). Similarly, $\z_j \in Y(\si(A))$ (resp., $\xi_j \in Y(\si(A))$) if and only if $\z_j \in Y(A_p)$ (resp., $\xi_j \in Y(A_q)$). Thus we have $Y(A_p) \sqcup Y(A_q) = Y(\si(A)) \sqcup Y(\sr(A))$.

 Suppose that $n = 1$. As above, $\z_j \in Y(\si(A))$ (resp., $\xi_j \in Y(\si(A))$) if and only if $\z_j \in Y(A_p)$ (resp., $\xi_j \in Y(A_q)$). Therefore, Lemma \ref{n=1case} implies that $Y(A_p) \sqcup Y(A_q) = Y(\si(A)) \sqcup Y(\sr(A))$. This finishes the proof.
\end{proof}

 All that is left is to give the following proposition for the proof of Theorem \ref{2-notchedbij}.

\begin{proposition}\label{bijcom}
 There is a bijection $\varphi^{pq} : \Acom(T_{\g^{(p)}},T_{\g^{(q)}}) \rightarrow \bP(G_{\g^{(pq)}})$ satisfying $\oox(A)=\oox(\varphi^{pq}(A))$ and $\ooy(A)=\ooy(\varphi^{pq}(A))$ for $A=(A_p,A_q) \in \Acom(T_{\g^{(p)}},T_{\g^{(q)}})$.
\end{proposition}

\begin{proof}
 By Propositions \ref{bijsym} and \ref{bijsymangle}, there are bijections
\[\SelectTips{eu}{}
\xymatrix@C=7mm@R=1mm{
 \bA(T_{\g^{(p)}}) \times \bA(T_{\g^{(q)}})
 && \ar[ll]_{\sim}^{\psi^p\times\psi^q} \Asym(T_{\ell_p}) \times \Asym(T_{\ell_q}) \ar[rr]^{\sim}_{\varphi^p\times\varphi^{q}}
 && \bP(G_{\g^{(p)}}) \times \bP(G_{\g^{(q)}})\\
 \rotatebox{90}{$\in$} && \rotatebox{90}{$\in$} &&  \rotatebox{90}{$\in$}\\
 A=(A_p,A_q) && \ar@{|->}[ll] (S_p,S_q) \ar@{|->}[rr] && (P_p,P_q)
}
\]
 satisfying $x(A_{\ast})=\ox(P_{\ast})$ and $y(A_{\ast})=\oy(P_{\ast})$, where $A_{\ast}=\psi^{\ast}(S_{\ast})$ $P_{\ast}=\varphi^{\ast}(S_{\ast})$ for $\ast\in\{p,q\}$.

 If $n>1$, by construction of $\psi^p$ and $\psi^q$, $A_p^c=A_q^c$ if and only if
\[
 C_{\tau_n}(S_p|_{A(U_{\ell_p}^1)\sqcup \{a_1^{\circ}\}})=C_{\tau_n}(S_p|_{A(U_{\ell_p}^2)\sqcup \{a_2^{\circ}\}})=C_{\tau_1}(S_q|_{A(U_{\ell_q}^1)\sqcup \{a_1^{\circ}\}})=C_{\tau_1}(S_q|_{A(U_{\ell_q}^2)\sqcup \{a_2^{\circ}\}}).
\]
 Thus it is the same as $\res(S_p)=\res(S_q)$, that is $\res(P_p)=\res(P_q)$. By Lemma \ref{defgsym}, $A \in \Acom(T_{\g^{(p)}},T_{\g^{(q)}})$ if and only if $(P_p,P_q) \in \bP(G_{\g^{(pq)}})$.

 If $n=1$ and $\res(P_p)=\res(P_q)$, then $A_p^q \sqcup A_q^{p}$ corresponds to $\res(S_p)=\res(S_q)$. Thus $A_p^q \sqcup A_q^{p} \in \bA(T_{\g})$. Conversely, suppose that $A_p^q \sqcup A_q^p \in \bA(T_{\g})$. The each angle of $S_p$ which is contained in the triangles $U_{\ell_p}^1$ and $U_{\ell_p}^2$ corresponds to the angle of $A_p^q$. Thus $A_p^q \sqcup A_q^{p}$ corresponds to $\res(S_p)$ since $A_p^q \sqcup A_q^p \in \bA(T_{\g})$. Similarly, $A_p^q \sqcup A_q^{p}$ corresponds to $\res(S_q)$. Therefore, we have $\res(S_p)=\res(S_q)$. So, by Lemma \ref{defgsym}, $A \in \Acom(T_{\g^{(p)}},T_{\g^{(q)}})$ if and only if $(P_p,P_q) \in \bP(G_{\g^{(pq)}})$, also in this case.

 Consequently, we have a bijection
\[
\varphi^{pq}:=(\varphi^p\times\varphi^{q})(\psi^p\times\psi^q)^{-1} : \Acom(T_{\g^{(p)}},T_{\g^{(q)}}) \rightarrow \bP(G_{\g^{(pq)}}).
\]
 On the other hand, we have $\sr(A) \simeq \res(S_p)$. As in the proof of Proposition \ref{bijsym}, we also have $x(\res(S_p))=x(\res(P_p))$ and $y(\res(S_p))=y(\res(P_p))$. Therefore, we have
\[
 \oox(\varphi^{pq}(A))=\frac{\ox(P_p)\ox(P_q)}{x(\res(P_p))}=\frac{x(A_p)x(A_q)}{x(\sr(A))}=\oox(A)
\]
and, similarly, $\ooy(\varphi^{pq}(A))=\ooy(A)$.
\end{proof}

\begin{proof}[Proof of Theorem \ref{2-notchedbij}]
 By Propositions \ref{bijcomangle} and \ref{bijcom}, there is a bijection $\varphi_{pq}=\varphi^{pq}\si^{-1} : \bA(T_{\g^{(pq)}}) \rightarrow \bP(G_{\g^{(pq)}})$ satisfying
\[
x(A)=\oox(\si^{-1}(A))=\oox(\varphi^{pq}\si^{-1}(A))\ \ \mbox{and}\ \ y(A)=\ooy(\si^{-1}(A))=\ooy(\varphi^{pq}\si^{-1}(A))
\]
 for $A \in \bA(T_{\g^{(pq)}})$.
\end{proof}

\begin{proof}[Proof of Theorem \ref{main} for $2$-notched arcs]
 The assertion follows immediately from Theorems \ref{MSW1com} and \ref{2-notchedbij}.
\end{proof}

\section{Proofs of our results for bipartite graphs}\label{bipartite}

 We refer the necessary notations in this section to the introduction. First, we prove the bijection between (1) and (3) in Theorem \ref{bijthm} and Proposition \ref{biparenclosed}.

\begin{proof}[Proof of the bijection between (1) and (3) in Theorem \ref{bijthm}]
 Angles incident to each vertex in $A(T_{\de})$ correspond bijectively with edges incident to the corresponding black vertex in $B_{\de}$. Angles in each triangle in $A(T_{\de})$ correspond bijectively with edges incident to the corresponding white vertex in $B_{\de}$. The assertion immediately follows from the definitions of perfect matchings of angles and perfect matchings of graphs.
\end{proof}

\begin{proof}[Proof of Proposition \ref{biparenclosed}]
 Let $E \in \bP(B_{\de})$. For any vertex $v$ of $B_{\de}$, $v$ is incident to exactly zero or two edges in $E_-(B_{\de}) \triangle E$. As a consequence, $E_-(B_{\de}) \triangle E$ is a disjoint union of non-crossing cycles. Thus the assertion holds.
\end{proof}

 Second, we have to be careful of the following special case to prove Proposition \ref{biparformula}.

\begin{lemma}\label{inI}
 Suppose that $\de=\g^{(pq)}$ and $n=1$. For $A \in \bA(T_{\g^{(pq)}})$, $\tau_1 \in Y(A)$ if and only if $\tau_1 \in I(\varpi(A))$.
\end{lemma}

\begin{proof}
 Since $A_{-}(T_{\g^{(pq)}})$ contains the angle between $\xi_1$ and a boundary segment of $T_{\g^{(pq)}}$, the assertion immediately follows from Proposition \ref{biparenclosed}.
\end{proof}

 Finally, we prove Proposition \ref{biparformula} and give an example for the results of this section.

\begin{proof}[Proof of Proposition \ref{biparformula}]
 It is trivial that $\varpi$ induce a bijection between $A_{\rm ex}(T_{\de})$ and the set of boundary edges of $B_{\de}$. Therefore, for $A \in \bA(T_{\de})$ and $\tau \in T_{\de}$, $\tau \in Y'(A)$ if and only if $E_-(B_{\de}) \triangle \varpi(A)$ contains at least one boundary edge of a square labeled by $\tau$, thus $\tau \in I(\varpi(A))$. By Lemma \ref{inI}, $\tau \in Y(A)$ if and only if $\tau \in I(\varpi(A))$.
\end{proof}

\begin{example}\label{ex2b}
 For the tagged arc $\de_2$ given in Subsection \ref{excoefffree}(2), we have
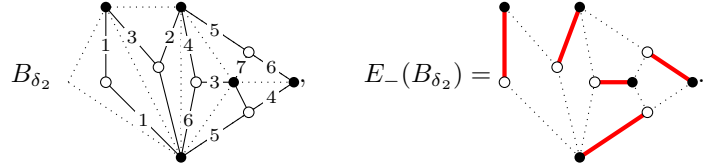
\[
B_{\de_2}\hspace{2mm}
\begin{tikzpicture}[baseline=0]
 \coordinate (l) at (0,0);
 \coordinate (lu) at (0.5,1); \fill (lu) circle (0.7mm);
 \coordinate (ru) at (1.5,1); \fill (ru) circle (0.7mm);
 \coordinate (r) at (3,0); \fill (r) circle (0.7mm);
 \coordinate (d) at (1.5,-1); \fill (d) circle (0.7mm);
 \coordinate (c) at (2.2,0); \fill (c) circle (0.7mm);
 \coordinate (1) at (0.5,0);
 \coordinate (2) at (1.2,0.2);
 \coordinate (3) at (1.7,0);
 \coordinate (4) at (2.4,0.4);
 \coordinate (5) at (2.4,-0.4);
 \draw[dotted] (l)--(lu)--(ru)--(r)--(d)--(l);
 \draw[dotted] (lu)--(d)--(ru)--(c)--(r) (d)--(c);
 \draw (d) to node[fill=white,inner sep=1]{\scriptsize $1$} (1);
 \draw (d)--(2);
 \draw (d) to node[fill=white,inner sep=1]{\scriptsize $6$} (3);
 \draw (d) to node[fill=white,inner sep=1]{\scriptsize $5$} (5);
 \draw (lu) to node[fill=white,inner sep=1]{\scriptsize $1$} (1);
 \draw (lu) to node[fill=white,inner sep=1]{\scriptsize $3$} (2);
 \draw (ru) to node[fill=white,inner sep=1]{\scriptsize $2$} (2);
 \draw (ru) to node[fill=white,inner sep=1]{\scriptsize $4$} (3);
 \draw (ru) to node[fill=white,inner sep=1]{\scriptsize $5$} (4);
 \draw (r) to node[fill=white,inner sep=1]{\scriptsize $6$} (4);
 \draw (r) to node[fill=white,inner sep=1]{\scriptsize $4$} (5);
 \draw (c) to node[fill=white,inner sep=1]{\scriptsize $3$} (3);
 \draw (c) to node[fill=white,inner sep=1]{\scriptsize $7$} (4);
 \draw (c)--(5);
 \filldraw[fill=white] (1) circle (0.7mm); \filldraw[fill=white] (2) circle (0.7mm); \filldraw[fill=white] (3) circle (0.7mm); \filldraw[fill=white] (4) circle (0.7mm); \filldraw[fill=white] (5) circle (0.7mm);
\end{tikzpicture},
 \hspace{7mm}
 E_-(B_{\de_2})=
\begin{tikzpicture}[baseline=0]
 \coordinate (l) at (0,0);
 \coordinate (lu) at (0.5,1);
 \coordinate (ru) at (1.5,1);
 \coordinate (r) at (3,0);
 \coordinate (d) at (1.5,-1);
 \coordinate (c) at (2.2,0);
 \coordinate (1) at (0.5,0);
 \coordinate (2) at (1.2,0.2);
 \coordinate (3) at (1.7,0);
 \coordinate (4) at (2.4,0.4);
 \coordinate (5) at (2.4,-0.4);
 \draw[dotted] (d) to (1);
 \draw[dotted] (d)--(2);
 \draw[dotted] (d) to (3);
 \draw[red,ultra thick] (d) to (5);
 \draw[red,ultra thick] (lu) to (1);
 \draw[dotted] (lu) to (2);
 \draw[red,ultra thick] (ru) to (2);
 \draw[dotted] (ru) to (3);
 \draw[dotted] (ru) to (4);
 \draw[red,ultra thick] (r) to (4);
 \draw[dotted] (r) to (5);
 \draw[red,ultra thick] (c) to (3);
 \draw[dotted] (c) to (4);
 \draw[dotted] (c)--(5);
 \fill (lu) circle (0.7mm); \fill (ru) circle (0.7mm); \fill (r) circle (0.7mm); \fill (d) circle (0.7mm); \fill (c) circle (0.7mm);
 \filldraw[fill=white] (1) circle (0.7mm); \filldraw[fill=white] (2) circle (0.7mm); \filldraw[fill=white] (3) circle (0.7mm); \filldraw[fill=white] (4) circle (0.7mm); \filldraw[fill=white] (5) circle (0.7mm);
\end{tikzpicture}.
\]
Then there are nine perfect matchings of $B_{\de_2}$ as follows:
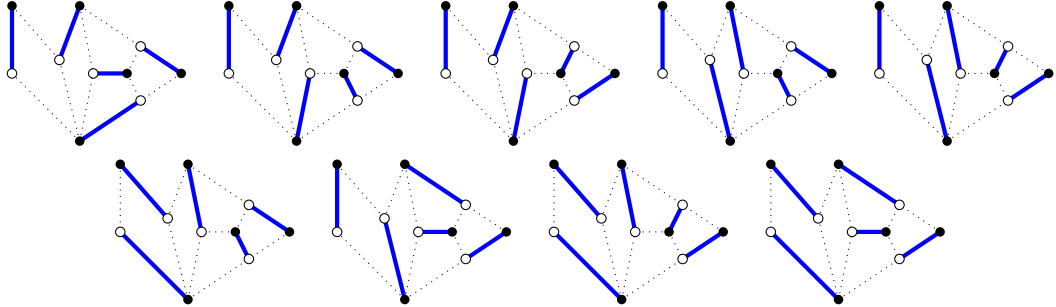
\[
\begin{tikzpicture}[scale=0.9]
 \coordinate (l) at (0,0);
 \coordinate (lu) at (0.5,1);
 \coordinate (ru) at (1.5,1);
 \coordinate (r) at (3,0);
 \coordinate (d) at (1.5,-1);
 \coordinate (c) at (2.2,0);
 \coordinate (1) at (0.5,0);
 \coordinate (2) at (1.2,0.2);
 \coordinate (3) at (1.7,0);
 \coordinate (4) at (2.4,0.4);
 \coordinate (5) at (2.4,-0.4);
 \draw[dotted] (d) to (1);
 \draw[dotted] (d)--(2);
 \draw[dotted] (d) to (3);
 \draw[blue,ultra thick] (d) to (5);
 \draw[blue,ultra thick] (lu) to (1);
 \draw[dotted] (lu) to (2);
 \draw[blue,ultra thick] (ru) to (2);
 \draw[dotted] (ru) to (3);
 \draw[dotted] (ru) to (4);
 \draw[blue,ultra thick] (r) to (4);
 \draw[dotted] (r) to (5);
 \draw[blue,ultra thick] (c) to (3);
 \draw[dotted] (c) to (4);
 \draw[dotted] (c)--(5);
 \fill (lu) circle (0.7mm); \fill (ru) circle (0.7mm); \fill (r) circle (0.7mm); \fill (d) circle (0.7mm); \fill (c) circle (0.7mm);
 \filldraw[fill=white] (1) circle (0.7mm); \filldraw[fill=white] (2) circle (0.7mm); \filldraw[fill=white] (3) circle (0.7mm); \filldraw[fill=white] (4) circle (0.7mm); \filldraw[fill=white] (5) circle (0.7mm);
\end{tikzpicture}
   \hspace{5mm}
\begin{tikzpicture}[scale=0.9]
 \coordinate (l) at (0,0);
 \coordinate (lu) at (0.5,1);
 \coordinate (ru) at (1.5,1);
 \coordinate (r) at (3,0);
 \coordinate (d) at (1.5,-1);
 \coordinate (c) at (2.2,0);
 \coordinate (1) at (0.5,0);
 \coordinate (2) at (1.2,0.2);
 \coordinate (3) at (1.7,0);
 \coordinate (4) at (2.4,0.4);
 \coordinate (5) at (2.4,-0.4);
 \draw[dotted] (d) to (1);
 \draw[dotted] (d)--(2);
 \draw[blue,ultra thick] (d) to (3);
 \draw[dotted] (d) to (5);
 \draw[blue,ultra thick] (lu) to (1);
 \draw[dotted] (lu) to (2);
 \draw[blue,ultra thick] (ru) to (2);
 \draw[dotted] (ru) to (3);
 \draw[dotted] (ru) to (4);
 \draw[blue,ultra thick] (r) to (4);
 \draw[dotted] (r) to (5);
 \draw[dotted] (c) to (3);
 \draw[dotted] (c) to (4);
 \draw[blue,ultra thick] (c)--(5);
 \fill (lu) circle (0.7mm); \fill (ru) circle (0.7mm); \fill (r) circle (0.7mm); \fill (d) circle (0.7mm); \fill (c) circle (0.7mm);
 \filldraw[fill=white] (1) circle (0.7mm); \filldraw[fill=white] (2) circle (0.7mm); \filldraw[fill=white] (3) circle (0.7mm); \filldraw[fill=white] (4) circle (0.7mm); \filldraw[fill=white] (5) circle (0.7mm);
\end{tikzpicture}
   \hspace{5mm}
\begin{tikzpicture}[scale=0.9]
 \coordinate (l) at (0,0);
 \coordinate (lu) at (0.5,1);
 \coordinate (ru) at (1.5,1);
 \coordinate (r) at (3,0);
 \coordinate (d) at (1.5,-1);
 \coordinate (c) at (2.2,0);
 \coordinate (1) at (0.5,0);
 \coordinate (2) at (1.2,0.2);
 \coordinate (3) at (1.7,0);
 \coordinate (4) at (2.4,0.4);
 \coordinate (5) at (2.4,-0.4);
 \draw[dotted] (d) to (1);
 \draw[dotted] (d)--(2);
 \draw[blue,ultra thick] (d) to (3);
 \draw[dotted] (d) to (5);
 \draw[blue,ultra thick] (lu) to (1);
 \draw[dotted] (lu) to (2);
 \draw[blue,ultra thick] (ru) to (2);
 \draw[dotted] (ru) to (3);
 \draw[dotted] (ru) to (4);
 \draw[dotted] (r) to (4);
 \draw[blue,ultra thick] (r) to (5);
 \draw[dotted] (c) to (3);
 \draw[blue,ultra thick] (c) to (4);
 \draw[dotted] (c)--(5);
 \fill (lu) circle (0.7mm); \fill (ru) circle (0.7mm); \fill (r) circle (0.7mm); \fill (d) circle (0.7mm); \fill (c) circle (0.7mm);
 \filldraw[fill=white] (1) circle (0.7mm); \filldraw[fill=white] (2) circle (0.7mm); \filldraw[fill=white] (3) circle (0.7mm); \filldraw[fill=white] (4) circle (0.7mm); \filldraw[fill=white] (5) circle (0.7mm);
\end{tikzpicture}
   \hspace{5mm}
\begin{tikzpicture}[scale=0.9]
 \coordinate (l) at (0,0);
 \coordinate (lu) at (0.5,1);
 \coordinate (ru) at (1.5,1);
 \coordinate (r) at (3,0);
 \coordinate (d) at (1.5,-1);
 \coordinate (c) at (2.2,0);
 \coordinate (1) at (0.5,0);
 \coordinate (2) at (1.2,0.2);
 \coordinate (3) at (1.7,0);
 \coordinate (4) at (2.4,0.4);
 \coordinate (5) at (2.4,-0.4);
 \draw[dotted] (d) to (1);
 \draw[blue,ultra thick] (d)--(2);
 \draw[dotted] (d) to (3);
 \draw[dotted] (d) to (5);
 \draw[blue,ultra thick] (lu) to (1);
 \draw[dotted] (lu) to (2);
 \draw[dotted] (ru) to (2);
 \draw[blue,ultra thick] (ru) to (3);
 \draw[dotted] (ru) to (4);
 \draw[blue,ultra thick] (r) to (4);
 \draw[dotted] (r) to (5);
 \draw[dotted] (c) to (3);
 \draw[dotted] (c) to (4);
 \draw[blue,ultra thick] (c)--(5);
 \fill (lu) circle (0.7mm); \fill (ru) circle (0.7mm); \fill (r) circle (0.7mm); \fill (d) circle (0.7mm); \fill (c) circle (0.7mm);
 \filldraw[fill=white] (1) circle (0.7mm); \filldraw[fill=white] (2) circle (0.7mm); \filldraw[fill=white] (3) circle (0.7mm); \filldraw[fill=white] (4) circle (0.7mm); \filldraw[fill=white] (5) circle (0.7mm);
\end{tikzpicture}
   \hspace{5mm}
\begin{tikzpicture}[scale=0.9]
 \coordinate (l) at (0,0);
 \coordinate (lu) at (0.5,1);
 \coordinate (ru) at (1.5,1);
 \coordinate (r) at (3,0);
 \coordinate (d) at (1.5,-1);
 \coordinate (c) at (2.2,0);
 \coordinate (1) at (0.5,0);
 \coordinate (2) at (1.2,0.2);
 \coordinate (3) at (1.7,0);
 \coordinate (4) at (2.4,0.4);
 \coordinate (5) at (2.4,-0.4);
 \draw[dotted] (d) to (1);
 \draw[blue,ultra thick] (d)--(2);
 \draw[dotted] (d) to (3);
 \draw[dotted] (d) to (5);
 \draw[blue,ultra thick] (lu) to (1);
 \draw[dotted] (lu) to (2);
 \draw[dotted] (ru) to (2);
 \draw[blue,ultra thick] (ru) to (3);
 \draw[dotted] (ru) to (4);
 \draw[dotted] (r) to (4);
 \draw[blue,ultra thick] (r) to (5);
 \draw[dotted] (c) to (3);
 \draw[blue,ultra thick] (c) to (4);
 \draw[dotted] (c)--(5);
 \fill (lu) circle (0.7mm); \fill (ru) circle (0.7mm); \fill (r) circle (0.7mm); \fill (d) circle (0.7mm); \fill (c) circle (0.7mm);
 \filldraw[fill=white] (1) circle (0.7mm); \filldraw[fill=white] (2) circle (0.7mm); \filldraw[fill=white] (3) circle (0.7mm); \filldraw[fill=white] (4) circle (0.7mm); \filldraw[fill=white] (5) circle (0.7mm);
\end{tikzpicture}
\]
\[
\begin{tikzpicture}[scale=0.9]
 \coordinate (l) at (0,0);
 \coordinate (lu) at (0.5,1);
 \coordinate (ru) at (1.5,1);
 \coordinate (r) at (3,0);
 \coordinate (d) at (1.5,-1);
 \coordinate (c) at (2.2,0);
 \coordinate (1) at (0.5,0);
 \coordinate (2) at (1.2,0.2);
 \coordinate (3) at (1.7,0);
 \coordinate (4) at (2.4,0.4);
 \coordinate (5) at (2.4,-0.4);
 \draw[blue,ultra thick] (d) to (1);
 \draw[dotted] (d)--(2);
 \draw[dotted] (d) to (3);
 \draw[dotted] (d) to (5);
 \draw[dotted] (lu) to (1);
 \draw[blue,ultra thick] (lu) to (2);
 \draw[dotted] (ru) to (2);
 \draw[blue,ultra thick] (ru) to (3);
 \draw[dotted] (ru) to (4);
 \draw[blue,ultra thick] (r) to (4);
 \draw[dotted] (r) to (5);
 \draw[dotted] (c) to (3);
 \draw[dotted] (c) to (4);
 \draw[blue,ultra thick] (c)--(5);
 \fill (lu) circle (0.7mm); \fill (ru) circle (0.7mm); \fill (r) circle (0.7mm); \fill (d) circle (0.7mm); \fill (c) circle (0.7mm);
 \filldraw[fill=white] (1) circle (0.7mm); \filldraw[fill=white] (2) circle (0.7mm); \filldraw[fill=white] (3) circle (0.7mm); \filldraw[fill=white] (4) circle (0.7mm); \filldraw[fill=white] (5) circle (0.7mm);
\end{tikzpicture}
   \hspace{5mm}
\begin{tikzpicture}[scale=0.9]
 \coordinate (l) at (0,0);
 \coordinate (lu) at (0.5,1);
 \coordinate (ru) at (1.5,1);
 \coordinate (r) at (3,0);
 \coordinate (d) at (1.5,-1);
 \coordinate (c) at (2.2,0);
 \coordinate (1) at (0.5,0);
 \coordinate (2) at (1.2,0.2);
 \coordinate (3) at (1.7,0);
 \coordinate (4) at (2.4,0.4);
 \coordinate (5) at (2.4,-0.4);
 \draw[dotted] (d) to (1);
 \draw[blue,ultra thick] (d)--(2);
 \draw[dotted] (d) to (3);
 \draw[dotted] (d) to (5);
 \draw[blue,ultra thick] (lu) to (1);
 \draw[dotted] (lu) to (2);
 \draw[dotted] (ru) to (2);
 \draw[dotted] (ru) to (3);
 \draw[blue,ultra thick] (ru) to (4);
 \draw[dotted] (r) to (4);
 \draw[blue,ultra thick] (r) to (5);
 \draw[blue,ultra thick] (c) to (3);
 \draw[dotted] (c) to (4);
 \draw[dotted] (c)--(5);
 \fill (lu) circle (0.7mm); \fill (ru) circle (0.7mm); \fill (r) circle (0.7mm); \fill (d) circle (0.7mm); \fill (c) circle (0.7mm);
 \filldraw[fill=white] (1) circle (0.7mm); \filldraw[fill=white] (2) circle (0.7mm); \filldraw[fill=white] (3) circle (0.7mm); \filldraw[fill=white] (4) circle (0.7mm); \filldraw[fill=white] (5) circle (0.7mm);
\end{tikzpicture}
   \hspace{5mm}
\begin{tikzpicture}[scale=0.9]
 \coordinate (l) at (0,0);
 \coordinate (lu) at (0.5,1);
 \coordinate (ru) at (1.5,1);
 \coordinate (r) at (3,0);
 \coordinate (d) at (1.5,-1);
 \coordinate (c) at (2.2,0);
 \coordinate (1) at (0.5,0);
 \coordinate (2) at (1.2,0.2);
 \coordinate (3) at (1.7,0);
 \coordinate (4) at (2.4,0.4);
 \coordinate (5) at (2.4,-0.4);
 \draw[blue,ultra thick] (d) to (1);
 \draw[dotted] (d)--(2);
 \draw[dotted] (d) to (3);
 \draw[dotted] (d) to (5);
 \draw[dotted] (lu) to (1);
 \draw[blue,ultra thick] (lu) to (2);
 \draw[dotted] (ru) to (2);
 \draw[blue,ultra thick] (ru) to (3);
 \draw[dotted] (ru) to (4);
 \draw[dotted] (r) to (4);
 \draw[blue,ultra thick] (r) to (5);
 \draw[dotted] (c) to (3);
 \draw[blue,ultra thick] (c) to (4);
 \draw[dotted] (c)--(5);
 \fill (lu) circle (0.7mm); \fill (ru) circle (0.7mm); \fill (r) circle (0.7mm); \fill (d) circle (0.7mm); \fill (c) circle (0.7mm);
 \filldraw[fill=white] (1) circle (0.7mm); \filldraw[fill=white] (2) circle (0.7mm); \filldraw[fill=white] (3) circle (0.7mm); \filldraw[fill=white] (4) circle (0.7mm); \filldraw[fill=white] (5) circle (0.7mm);
\end{tikzpicture}
   \hspace{5mm}
\begin{tikzpicture}[scale=0.9]
 \coordinate (l) at (0,0);
 \coordinate (lu) at (0.5,1);
 \coordinate (ru) at (1.5,1);
 \coordinate (r) at (3,0);
 \coordinate (d) at (1.5,-1);
 \coordinate (c) at (2.2,0);
 \coordinate (1) at (0.5,0);
 \coordinate (2) at (1.2,0.2);
 \coordinate (3) at (1.7,0);
 \coordinate (4) at (2.4,0.4);
 \coordinate (5) at (2.4,-0.4);
 \draw[blue,ultra thick] (d) to (1);
 \draw[dotted] (d)--(2);
 \draw[dotted] (d) to (3);
 \draw[dotted] (d) to (5);
 \draw[dotted] (lu) to (1);
 \draw[blue,ultra thick] (lu) to (2);
 \draw[dotted] (ru) to (2);
 \draw[dotted] (ru) to (3);
 \draw[blue,ultra thick] (ru) to (4);
 \draw[dotted] (r) to (4);
 \draw[blue,ultra thick] (r) to (5);
 \draw[blue,ultra thick] (c) to (3);
 \draw[dotted] (c) to (4);
 \draw[dotted] (c)--(5);
 \fill (lu) circle (0.7mm); \fill (ru) circle (0.7mm); \fill (r) circle (0.7mm); \fill (d) circle (0.7mm); \fill (c) circle (0.7mm);
 \filldraw[fill=white] (1) circle (0.7mm); \filldraw[fill=white] (2) circle (0.7mm); \filldraw[fill=white] (3) circle (0.7mm); \filldraw[fill=white] (4) circle (0.7mm); \filldraw[fill=white] (5) circle (0.7mm);
\end{tikzpicture}
\]
 It is easy to check that these correspond bijectively with perfect matchings of angles in $T_{\de_2}$ given in Subsection \ref{excoefffree}(2). Moreover, for each $E \in \bP(B_{\de_2})$, the subgraph $B_E$ in Proposition \ref{biparenclosed} is given as follows:
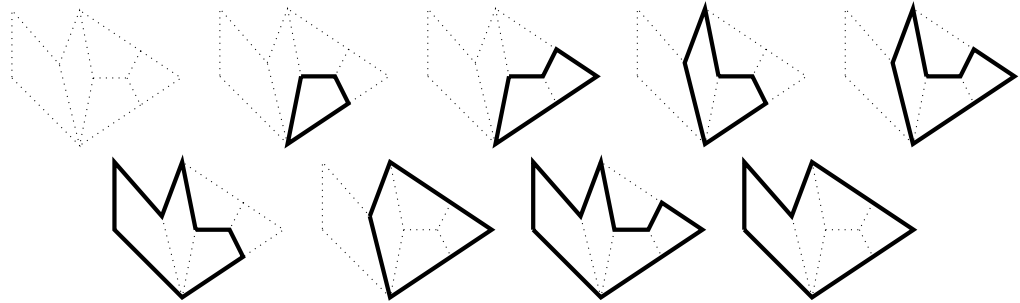
\[
\begin{tikzpicture}[scale=0.9]
 \coordinate (l) at (0,0);
 \coordinate (lu) at (0.5,1);
 \coordinate (ru) at (1.5,1);
 \coordinate (r) at (3,0);
 \coordinate (d) at (1.5,-1);
 \coordinate (c) at (2.2,0);
 \coordinate (1) at (0.5,0);
 \coordinate (2) at (1.2,0.2);
 \coordinate (3) at (1.7,0);
 \coordinate (4) at (2.4,0.4);
 \coordinate (5) at (2.4,-0.4);
 \draw[dotted] (1)--(d)--(2)--(lu)--(1) (2)--(ru)--(3)--(d) (3)--(c)--(4)--(ru) (4)--(r)--(5)--(c) (5)--(d);
\end{tikzpicture}
   \hspace{5mm}
\begin{tikzpicture}[scale=0.9]
 \coordinate (l) at (0,0);
 \coordinate (lu) at (0.5,1);
 \coordinate (ru) at (1.5,1);
 \coordinate (r) at (3,0);
 \coordinate (d) at (1.5,-1);
 \coordinate (c) at (2.2,0);
 \coordinate (1) at (0.5,0);
 \coordinate (2) at (1.2,0.2);
 \coordinate (3) at (1.7,0);
 \coordinate (4) at (2.4,0.4);
 \coordinate (5) at (2.4,-0.4);
 \draw[dotted] (1)--(d)--(2)--(lu)--(1) (2)--(ru)--(3)--(d) (3)--(c)--(4)--(ru) (4)--(r)--(5)--(c) (5)--(d);
 \draw[ultra thick] (3)--(c)--(5)--(d)--(3);
\end{tikzpicture}
   \hspace{5mm}
\begin{tikzpicture}[scale=0.9]
 \coordinate (l) at (0,0);
 \coordinate (lu) at (0.5,1);
 \coordinate (ru) at (1.5,1);
 \coordinate (r) at (3,0);
 \coordinate (d) at (1.5,-1);
 \coordinate (c) at (2.2,0);
 \coordinate (1) at (0.5,0);
 \coordinate (2) at (1.2,0.2);
 \coordinate (3) at (1.7,0);
 \coordinate (4) at (2.4,0.4);
 \coordinate (5) at (2.4,-0.4);
 \draw[dotted] (1)--(d)--(2)--(lu)--(1) (2)--(ru)--(3)--(d) (3)--(c)--(4)--(ru) (4)--(r)--(5)--(c) (5)--(d);
 \draw[ultra thick] (3)--(c)--(4)--(r)--(5)--(d)--(3) ;
\end{tikzpicture}
   \hspace{5mm}
\begin{tikzpicture}[scale=0.9]
 \coordinate (l) at (0,0);
 \coordinate (lu) at (0.5,1);
 \coordinate (ru) at (1.5,1);
 \coordinate (r) at (3,0);
 \coordinate (d) at (1.5,-1);
 \coordinate (c) at (2.2,0);
 \coordinate (1) at (0.5,0);
 \coordinate (2) at (1.2,0.2);
 \coordinate (3) at (1.7,0);
 \coordinate (4) at (2.4,0.4);
 \coordinate (5) at (2.4,-0.4);
 \draw[dotted] (1)--(d)--(2)--(lu)--(1) (2)--(ru)--(3)--(d) (3)--(c)--(4)--(ru) (4)--(r)--(5)--(c) (5)--(d);
 \draw[ultra thick] (3)--(c)--(5)--(d)--(2)--(ru)--(3);
\end{tikzpicture}
   \hspace{5mm}
\begin{tikzpicture}[scale=0.9]
 \coordinate (l) at (0,0);
 \coordinate (lu) at (0.5,1);
 \coordinate (ru) at (1.5,1);
 \coordinate (r) at (3,0);
 \coordinate (d) at (1.5,-1);
 \coordinate (c) at (2.2,0);
 \coordinate (1) at (0.5,0);
 \coordinate (2) at (1.2,0.2);
 \coordinate (3) at (1.7,0);
 \coordinate (4) at (2.4,0.4);
 \coordinate (5) at (2.4,-0.4);
 \draw[dotted] (1)--(d)--(2)--(lu)--(1) (2)--(ru)--(3)--(d) (3)--(c)--(4)--(ru) (4)--(r)--(5)--(c) (5)--(d);
 \draw[ultra thick] (3)--(c)--(4)--(r)--(5)--(d)--(2)--(ru)--(3);
\end{tikzpicture}
\]
\[
\begin{tikzpicture}[scale=0.9]
 \coordinate (l) at (0,0);
 \coordinate (lu) at (0.5,1);
 \coordinate (ru) at (1.5,1);
 \coordinate (r) at (3,0);
 \coordinate (d) at (1.5,-1);
 \coordinate (c) at (2.2,0);
 \coordinate (1) at (0.5,0);
 \coordinate (2) at (1.2,0.2);
 \coordinate (3) at (1.7,0);
 \coordinate (4) at (2.4,0.4);
 \coordinate (5) at (2.4,-0.4);
 \draw[dotted] (1)--(d)--(2)--(lu)--(1) (2)--(ru)--(3)--(d) (3)--(c)--(4)--(ru) (4)--(r)--(5)--(c) (5)--(d);
 \draw[ultra thick] (3)--(c)--(5)--(d)--(1)--(lu)--(2)--(ru)--(3);
\end{tikzpicture}
   \hspace{5mm}
\begin{tikzpicture}[scale=0.9]
 \coordinate (l) at (0,0);
 \coordinate (lu) at (0.5,1);
 \coordinate (ru) at (1.5,1);
 \coordinate (r) at (3,0);
 \coordinate (d) at (1.5,-1);
 \coordinate (c) at (2.2,0);
 \coordinate (1) at (0.5,0);
 \coordinate (2) at (1.2,0.2);
 \coordinate (3) at (1.7,0);
 \coordinate (4) at (2.4,0.4);
 \coordinate (5) at (2.4,-0.4);
 \draw[dotted] (1)--(d)--(2)--(lu)--(1) (2)--(ru)--(3)--(d) (3)--(c)--(4)--(ru) (4)--(r)--(5)--(c) (5)--(d);
 \draw[ultra thick] (2)--(ru)--(4)--(r)--(5)--(d)--(2);
\end{tikzpicture}
   \hspace{5mm}
\begin{tikzpicture}[scale=0.9]
 \coordinate (l) at (0,0);
 \coordinate (lu) at (0.5,1);
 \coordinate (ru) at (1.5,1);
 \coordinate (r) at (3,0);
 \coordinate (d) at (1.5,-1);
 \coordinate (c) at (2.2,0);
 \coordinate (1) at (0.5,0);
 \coordinate (2) at (1.2,0.2);
 \coordinate (3) at (1.7,0);
 \coordinate (4) at (2.4,0.4);
 \coordinate (5) at (2.4,-0.4);
 \draw[dotted] (1)--(d)--(2)--(lu)--(1) (2)--(ru)--(3)--(d) (3)--(c)--(4)--(ru) (4)--(r)--(5)--(c) (5)--(d);
 \draw[ultra thick] (1)--(lu)--(2)--(ru)--(3)--(c)--(4)--(r)--(5)--(d)--(1);
\end{tikzpicture}
   \hspace{5mm}
\begin{tikzpicture}[scale=0.9]
 \coordinate (l) at (0,0);
 \coordinate (lu) at (0.5,1);
 \coordinate (ru) at (1.5,1);
 \coordinate (r) at (3,0);
 \coordinate (d) at (1.5,-1);
 \coordinate (c) at (2.2,0);
 \coordinate (1) at (0.5,0);
 \coordinate (2) at (1.2,0.2);
 \coordinate (3) at (1.7,0);
 \coordinate (4) at (2.4,0.4);
 \coordinate (5) at (2.4,-0.4);
 \draw[dotted] (1)--(d)--(2)--(lu)--(1) (2)--(ru)--(3)--(d) (3)--(c)--(4)--(ru) (4)--(r)--(5)--(c) (5)--(d);
 \draw[ultra thick] (1)--(lu)--(2)--(ru)--(4)--(r)--(5)--(d)--(1);
\end{tikzpicture}
\]
By comparing with Subsection \ref{exprincipal}(2), we can check that Proposition \ref{biparformula} holds in this case.
\end{example}

\section{Minimal cuts of quivers with potential}\label{pmmc}

 In this section, we show that perfect matchings of angles in $T_{\de}$ coincide with minimal cuts of quiver with potential obtained from $T_{\de}$, that is the bijection between (1) and (4) in Theorem \ref{bijthm}.


\subsection{Quivers with potential and cuts}\label{QPandcuts}

 We recall the definitions of quivers with potential \cite{DWZ} and of their cuts \cite{BFPPT,HI}. We denote by $\bZ Q$ the path algebra of a quiver $Q$ over the ring $\bZ$ of integers.

\begin{definition}
 (1) A {\it quiver with potential} (QP for short) is a pair $(Q,W)$ of a quiver $Q$ and an element $W \in \bZ Q$ which is a linear combination of cyclic paths.\par
 (2) A {\it cut} of a QP $(Q,W)$ is a subset $C$ of $Q_1$ such that any cyclic path appearing in $W$ contains precisely one arrow in $C$.
\end{definition}

 We define a quiver $Q_{\de}$ as follows: the set of vertices consists of diagonals and boundary segments of $T_{\de}$; the set of arrows consists of arrows from $i$ to $j$, where $i$ and $j$ are in the common triangle of $T_{\de}$ and $j$ follows $i$ in the counterclockwise order. We denote by $\oQ_{\de}$ the quiver obtained from $Q_{\de}$ by adding arrows from $i$ to $j$, where $i$ and $j$ are boundary segments which are not in the common triangle of $T_{\de}$ and $i$ is a predecessor of $j$ with respect to clockwise order.

 To define a potential $\oW_{\de}$ of $\oQ_{\de}$, we consider the following cycles of $\oQ_{\de}$. A {\it triangle cycle} is a cycle of length $3$ inside a triangle of $T_{\de}$. An {\it exterior cycle} is a cycle winding around a vertex (possibly a puncture) of $T_{\de}$. We define
\begin{equation*}
 \oW_{\de}=\sum(\mbox{triangle cycles in }\oQ_{\de})-\sum(\mbox{exterior cycles in }\oQ_{\de}).
\end{equation*}
 Note that this extends QPs for triangulated polygons without punctures defined in \cite{DL} to QPs for triangulated polygons with punctures.

\begin{lemma}\label{number}
 The number of triangle cycles in $\oQ_{\de}$ and the number of exterior cycles in $\oQ_{\de}$ coincide.
\end{lemma}

\begin{proof}
 By construction, the number of triangle cycles in $\oQ_{\de}$ and the number of triangles in $T_{\de}$ coincide. Similarly, the number of exterior cycles in $\oQ_{\de}$ and the number of vertices incident to at least one diagonal in $T_{\de}$. So all these numbers coincide.
\end{proof}

 We denote by $n(\de)$ the number in $Lemma$ \ref{number}.

\subsection{Minimal cuts of QPs and Perfect matchings of angles}

  We have a natural injection $\rho : A(T_{\de}) \rightarrow (Q_{\de})_1$ given by the following picture:
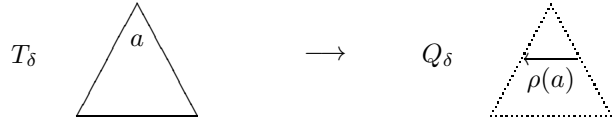
\[\SelectTips{eu}{}
  \begin{xy}
   (0,0)="1", +(15,0)="2", +(15,0)="3", +(0,-15)="4", +(-7,0)="5", +(-16,0)="6", +(-7,0)="7", "1"+(0,-7)*{T_{\de}}, "3"+(10,-7)*{\longrightarrow}, +(15,0)*{Q_{\de}}, "3"+(25,0)="01", +(15,0)="02", +(15,0)="03", +(0,-15)="04", +(-7,0)="05", +(-16,0)="06", +(-7,0)="07", "2"+(0,-5)*{a}, "02"+(-3.5,-7.5)="A", +(7,0)="B"
   \ar@{-}"5";"6" \ar@{-}"2";"6" \ar@{-}"2";"5" \ar@{.}"05";"06" \ar@{.}"02";"06" \ar@{.}"02";"05" \ar@{->}"B";"A"^*{\rho(a)}
  \end{xy}
\]

 Cuts of $(\oQ_{\de},\oW_{\de})$ have the following property using the map $\rho$.

\begin{lemma}\label{minnumber}
 (a) Any cut $C$ has the cardinality $|C| \ge n(\de)$.\\
 (b) The equality in (a) holds if and only if $C$ is contained in $\rho(A(T_{\de}))$.
\end{lemma}

\begin{proof}
 Since there are $n(\de)$ triangle cycles (resp., $n(\de)$ exterior cycles) not sharing arrows with each other, (a) holds. There is an exterior cycle sharing arrows with each triangle cycle. Since the shared arrows are contained in $\rho(A(T_{\de}))$, the sufficiency of (b) holds. Since $\rho(A(T_{\de}))$ is contained in the set of arrows appearing in a triangle cycle of $\oQ_{\de}$, then $|C| \le n(\de)$ for $C \subset \rho(A(T_{\de}))$. Thus the necessity of (b) holds.
\end{proof}

\begin{definition}\label{defmin}
 A cut $C$ of $(\oQ_{\de},\oW_{\de})$ is called {\it minimal} if $|C| = n(\de)$.
\end{definition}

By Theorem \ref{bijthm}, $(\oQ_{\de},\oW_{\de})$ always has minimal cuts.

\begin{proof}[Proof of the bijection between (1) and (4) in Theorem \ref{bijthm}]
 Let $A \subseteq A(T_{\de})$ and $C:=\rho(A) \subseteq (Q_{\de})_1$. Then there is exactly one element $a$ of $A$ in any triangle of $T_{\de}$ (resp., incident to any vertex of $T_{\de}$) if and only if the corresponding triangle cycle (resp., exterior cycle) contains precisely one arrow $\rho(a)$ in $C$. Thus $A \in \bA(T_{\de})$ if and only if $C$ is a cut. Since minimal cuts are precisely cuts contained in $\rho(A(T_{\de}))$ by Lemma \ref{minnumber}(b), the assertion follows.
\end{proof}

 Consequently, we can give another cluster expansion formula in terms of minimal cuts.

\begin{corollary}\label{cutformula}
 We have
\[
 x_{\de}=\Phi\Biggl(\frac{1}{\cross(T,\de)}\sum_{C}x(\rho^{-1}(C))y(\rho^{-1}(C))\Biggr),
\]
where $C$ runs over all minimal cuts of $(\oQ_{\de},\oW_{\de})$ and $\cross(T,\de)$, $x(\rho^{-1}(C))$ and $y(\rho^{-1}(C))$ are defined in Theorems \ref{main'} and \ref{main}.
\end{corollary}

\begin{proof}
 The assertion follows immediately from Theorems \ref{bijthm} and \ref{main}.
\end{proof}

\begin{example}\label{ex2q}
 For the tagged arc $\de_2$ given in Subsection \ref{excoefffree}(2), we have
\begin{eqnarray*}
(\oQ_{\de_2},\oW_{\de_2})=
\left(\hspace{2mm}
\begin{tikzpicture}[baseline=-1mm]
 \coordinate (0) at (0,0); \node (a) at (0) {$2$};
 \coordinate (1) at (1,0); \node (b) at (1) {$3$};
 \coordinate (2) at (1.765,-0.5); \node (c) at (2) {$4$};
 \coordinate (3) at (2.53,0); \node (d) at (3) {$5$};
 \coordinate (4) at (1.765,0.5); \node (e) at (4) {$6$};
 \coordinate (5) at (2.53,1); \node (ru) at (5) {$7$};
 \coordinate (6) at (2.53,-1); \node (rd) at (6) {$\circ$};
 \coordinate (7) at (0.5,0.765); \node (ab) at (7) {$\circ$};
 \coordinate (8) at (-0.765,0.5); \node (lu) at (8) {$1$};
 \coordinate (9) at (-0.765,-0.5); \node (ld) at (9) {$1$};
 \draw[->] (a) -- (b);
 \draw[->] (b) -- (c);
 \draw[->] (d) -- (c);
 \draw[->] (e) -- (d);
 \draw[->] (e) -- (b);
 \draw[->] (c) -- (e);
 \draw[->] (d) -- (ru);
 \draw[->] (ru) -- (e);
 \draw[->] (c) -- (rd);
 \draw[->] (rd) -- (d);
 \draw[->] (b) -- (ab);
 \draw[->] (ab) -- (a);
 \draw[->] (a) -- (lu);
 \draw[->] (lu) -- (ld);
 \draw[->] (ld) -- (a);
 \draw[->] (lu) .. controls (-0.6,1.2) and (0,1.4) .. (ab);
 \draw[->] (ab) .. controls (1,1.6) and (2,1.6) .. (ru);
 \draw[->] (ru) .. controls (3.5,0.7) and (3.5,-0.7) .. (rd);
 \draw[->] (rd) .. controls (1.7,-1.7) and (0,-1.5) .. (ld);
\end{tikzpicture}
, 
     \left.\begin{array}{ll}
\sum\Bigl(\mbox{five triangle cycles}\hspace{2mm}
\begin{tikzpicture}[scale=0.6,baseline=-1mm]
 \coordinate (a) at (0,0);
 \coordinate (b) at (1,0);
 \coordinate (c) at (1.765,-0.5);
 \coordinate (d) at (2.53,0);
 \coordinate (e) at (1.765,0.5);
 \coordinate (ru) at (2.53,1);
 \coordinate (rd) at (2.53,-1);
 \coordinate (ab) at (0.5,0.765);
 \coordinate (lu) at (-0.765,0.5);
 \coordinate (ld) at (-0.765,-0.5);
 \filldraw [pattern=north east lines] (a)--(lu)--(ld)--(a);
 \filldraw [pattern=north east lines] (a)--(ab)--(b)--(a);
 \filldraw [pattern=north east lines] (b)--(c)--(e)--(b);
 \filldraw [pattern=north east lines] (e)--(ru)--(d)--(e);
 \filldraw [pattern=north east lines] (d)--(rd)--(c)--(d);
 \draw (a) -- (b);
 \draw (b) -- (c);
 \draw (d) -- (c);
 \draw (e) -- (d);
 \draw (e) -- (b);
 \draw (c) -- (e);
 \draw (d) -- (ru);
 \draw (ru) -- (e);
 \draw (c) -- (rd);
 \draw (rd) -- (d);
 \draw (b) -- (ab);
 \draw (ab) -- (a);
 \draw (a) -- (lu);
 \draw (lu) -- (ld);
 \draw (ld) -- (a);
 \draw (lu) .. controls (-0.6,1.2) and (0,1.4) .. (ab);
 \draw (ab) .. controls (1,1.6) and (2,1.6) .. (ru);
 \draw (ru) .. controls (3.5,0.7) and (3.5,-0.7) .. (rd);
 \draw (rd) .. controls (1.7,-1.7) and (0,-1.5) .. (ld);
\end{tikzpicture}
\Bigr)\\
-\sum\Bigl(\mbox{five exterior cycles}\hspace{2mm}
\begin{tikzpicture}[scale=0.6,baseline=-1mm]
 \coordinate (a) at (0,0);
 \coordinate (b) at (1,0);
 \coordinate (c) at (1.765,-0.5);
 \coordinate (d) at (2.53,0);
 \coordinate (e) at (1.765,0.5);
 \coordinate (ru) at (2.53,1);
 \coordinate (rd) at (2.53,-1);
 \coordinate (ab) at (0.5,0.765);
 \coordinate (lu) at (-0.765,0.5);
 \coordinate (ld) at (-0.765,-0.5);
 \filldraw [pattern=north east lines] (a)--(lu) .. controls (-0.6,1.2) and (0,1.4) .. (ab)--(a);
 \filldraw [pattern=north east lines] (b)--(ab) .. controls (1,1.6) and (2,1.6) .. (ru)--(b);
 \filldraw [pattern=north east lines] (rd)--(ru) .. controls (3.5,0.7) and (3.5,-0.7) .. (rd);
 \filldraw [pattern=north east lines] (b)--(rd) .. controls (1.7,-1.7) and (0,-1.5) .. (ld)--(a)--(b);
 \filldraw [pattern=north east lines] (d)--(e)--(c)--(d);
 \draw (a) -- (b);
 \draw (b) -- (c);
 \draw (d) -- (c);
 \draw (e) -- (d);
 \draw (e) -- (b);
 \draw (c) -- (e);
 \draw (d) -- (ru);
 \draw (ru) -- (e);
 \draw (c) -- (rd);
 \draw (rd) -- (d);
 \draw (b) -- (ab);
 \draw (ab) -- (a);
 \draw (a) -- (lu);
 \draw (lu) -- (ld);
 \draw (ld) -- (a);
 \draw (lu) .. controls (-0.6,1.2) and (0,1.4) .. (ab);
 \draw (ab) .. controls (1,1.6) and (2,1.6) .. (ru);
 \draw (ru) .. controls (3.5,0.7) and (3.5,-0.7) .. (rd);
 \draw (rd) .. controls (1.7,-1.7) and (0,-1.5) .. (ld);
\end{tikzpicture}
\Bigr)
     \end{array} \right.
\right).
\end{eqnarray*}
Then there are nine minimal cuts of $(\oQ_{\de_2},\oW_{\de_2})$ as follows:
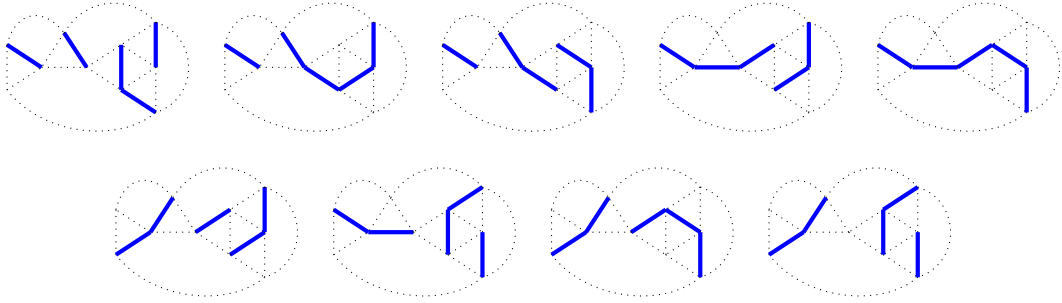
\[
\begin{tikzpicture}[scale=0.6]
 \coordinate (a) at (0,0);
 \coordinate (b) at (1,0);
 \coordinate (c) at (1.765,-0.5);
 \coordinate (d) at (2.53,0);
 \coordinate (e) at (1.765,0.5);
 \coordinate (ru) at (2.53,1);
 \coordinate (rd) at (2.53,-1);
 \coordinate (ab) at (0.5,0.765);
 \coordinate (lu) at (-0.765,0.5);
 \coordinate (ld) at (-0.765,-0.5);
 \draw[dotted] (a) -- (b) -- (c) -- (d) -- (e) -- (b);
 \draw[blue,ultra thick] (c) -- (e);
 \draw[blue,ultra thick] (d) -- (ru);
 \draw[dotted] (ru) -- (e);
 \draw[blue,ultra thick] (c) -- (rd);
 \draw[dotted] (rd) -- (d);
 \draw[blue,ultra thick] (b) -- (ab);
 \draw[dotted] (ab) -- (a);
 \draw[blue,ultra thick] (a) -- (lu);
 \draw[dotted] (lu) -- (ld);
 \draw[dotted] (ld) -- (a);
 \draw[dotted] (lu) .. controls (-0.6,1.2) and (0,1.4) .. (ab);
 \draw[dotted] (ab) .. controls (1,1.6) and (2,1.6) .. (ru);
 \draw[dotted] (ru) .. controls (3.5,0.7) and (3.5,-0.7) .. (rd);
 \draw[dotted] (rd) .. controls (1.7,-1.7) and (0,-1.5) .. (ld);
\end{tikzpicture}
   \hspace{3mm}
\begin{tikzpicture}[scale=0.6]
 \coordinate (a) at (0,0);
 \coordinate (b) at (1,0);
 \coordinate (c) at (1.765,-0.5);
 \coordinate (d) at (2.53,0);
 \coordinate (e) at (1.765,0.5);
 \coordinate (ru) at (2.53,1);
 \coordinate (rd) at (2.53,-1);
 \coordinate (ab) at (0.5,0.765);
 \coordinate (lu) at (-0.765,0.5);
 \coordinate (ld) at (-0.765,-0.5);
 \draw[dotted] (a) -- (b);
 \draw[blue,ultra thick] (b) -- (c);
 \draw[blue,ultra thick] (d) -- (c);
 \draw[dotted] (e) -- (d);
 \draw[dotted] (e) -- (b);
 \draw[dotted] (c) -- (e);
 \draw[blue,ultra thick] (d) -- (ru);
 \draw[dotted] (ru) -- (e);
 \draw[dotted] (c) -- (rd);
 \draw[dotted] (rd) -- (d);
 \draw[blue,ultra thick] (b) -- (ab);
 \draw[dotted] (ab) -- (a);
 \draw[blue,ultra thick] (a) -- (lu);
 \draw[dotted] (lu) -- (ld);
 \draw[dotted] (ld) -- (a);
 \draw[dotted] (lu) .. controls (-0.6,1.2) and (0,1.4) .. (ab);
 \draw[dotted] (ab) .. controls (1,1.6) and (2,1.6) .. (ru);
 \draw[dotted] (ru) .. controls (3.5,0.7) and (3.5,-0.7) .. (rd);
 \draw[dotted] (rd) .. controls (1.7,-1.7) and (0,-1.5) .. (ld);
\end{tikzpicture}
   \hspace{3mm}
\begin{tikzpicture}[scale=0.6]
 \coordinate (a) at (0,0);
 \coordinate (b) at (1,0);
 \coordinate (c) at (1.765,-0.5);
 \coordinate (d) at (2.53,0);
 \coordinate (e) at (1.765,0.5);
 \coordinate (ru) at (2.53,1);
 \coordinate (rd) at (2.53,-1);
 \coordinate (ab) at (0.5,0.765);
 \coordinate (lu) at (-0.765,0.5);
 \coordinate (ld) at (-0.765,-0.5);
 \draw[dotted] (a) -- (b);
 \draw[blue,ultra thick] (b) -- (c);
 \draw[dotted] (d) -- (c);
 \draw[blue,ultra thick] (e) -- (d);
 \draw[dotted] (e) -- (b);
 \draw[dotted] (c) -- (e);
 \draw[dotted] (d) -- (ru);
 \draw[dotted] (ru) -- (e);
 \draw[dotted] (c) -- (rd);
 \draw[blue,ultra thick] (rd) -- (d);
 \draw[blue,ultra thick] (b) -- (ab);
 \draw[dotted] (ab) -- (a);
 \draw[blue,ultra thick] (a) -- (lu);
 \draw[dotted] (lu) -- (ld);
 \draw[dotted] (ld) -- (a);
 \draw[dotted] (lu) .. controls (-0.6,1.2) and (0,1.4) .. (ab);
 \draw[dotted] (ab) .. controls (1,1.6) and (2,1.6) .. (ru);
 \draw[dotted] (ru) .. controls (3.5,0.7) and (3.5,-0.7) .. (rd);
 \draw[dotted] (rd) .. controls (1.7,-1.7) and (0,-1.5) .. (ld);
\end{tikzpicture}
   \hspace{3mm}
\begin{tikzpicture}[scale=0.6]
 \coordinate (a) at (0,0);
 \coordinate (b) at (1,0);
 \coordinate (c) at (1.765,-0.5);
 \coordinate (d) at (2.53,0);
 \coordinate (e) at (1.765,0.5);
 \coordinate (ru) at (2.53,1);
 \coordinate (rd) at (2.53,-1);
 \coordinate (ab) at (0.5,0.765);
 \coordinate (lu) at (-0.765,0.5);
 \coordinate (ld) at (-0.765,-0.5);
 \draw[blue,ultra thick] (a) -- (b);
 \draw[dotted] (b) -- (c);
 \draw[blue,ultra thick] (d) -- (c);
 \draw[dotted] (e) -- (d);
 \draw[blue,ultra thick] (e) -- (b);
 \draw[dotted] (c) -- (e);
 \draw[blue,ultra thick] (d) -- (ru);
 \draw[dotted] (ru) -- (e);
 \draw[dotted] (c) -- (rd);
 \draw[dotted] (rd) -- (d);
 \draw[dotted] (b) -- (ab);
 \draw[dotted] (ab) -- (a);
 \draw[blue,ultra thick] (a) -- (lu);
 \draw[dotted] (lu) -- (ld);
 \draw[dotted] (ld) -- (a);
 \draw[dotted] (lu) .. controls (-0.6,1.2) and (0,1.4) .. (ab);
 \draw[dotted] (ab) .. controls (1,1.6) and (2,1.6) .. (ru);
 \draw[dotted] (ru) .. controls (3.5,0.7) and (3.5,-0.7) .. (rd);
 \draw[dotted] (rd) .. controls (1.7,-1.7) and (0,-1.5) .. (ld);
\end{tikzpicture}
   \hspace{3mm}
\begin{tikzpicture}[scale=0.6]
 \coordinate (a) at (0,0);
 \coordinate (b) at (1,0);
 \coordinate (c) at (1.765,-0.5);
 \coordinate (d) at (2.53,0);
 \coordinate (e) at (1.765,0.5);
 \coordinate (ru) at (2.53,1);
 \coordinate (rd) at (2.53,-1);
 \coordinate (ab) at (0.5,0.765);
 \coordinate (lu) at (-0.765,0.5);
 \coordinate (ld) at (-0.765,-0.5);
 \draw[blue,ultra thick] (a) -- (b);
 \draw[dotted] (b) -- (c);
 \draw[dotted] (d) -- (c);
 \draw[blue,ultra thick] (e) -- (d);
 \draw[blue,ultra thick] (e) -- (b);
 \draw[dotted] (c) -- (e);
 \draw[dotted] (d) -- (ru);
 \draw[dotted] (ru) -- (e);
 \draw[dotted] (c) -- (rd);
 \draw[blue,ultra thick] (rd) -- (d);
 \draw[dotted] (b) -- (ab);
 \draw[dotted] (ab) -- (a);
 \draw[blue,ultra thick] (a) -- (lu);
 \draw[dotted] (lu) -- (ld);
 \draw[dotted] (ld) -- (a);
 \draw[dotted] (lu) .. controls (-0.6,1.2) and (0,1.4) .. (ab);
 \draw[dotted] (ab) .. controls (1,1.6) and (2,1.6) .. (ru);
 \draw[dotted] (ru) .. controls (3.5,0.7) and (3.5,-0.7) .. (rd);
 \draw[dotted] (rd) .. controls (1.7,-1.7) and (0,-1.5) .. (ld);
\end{tikzpicture}
\]
\[
\begin{tikzpicture}[scale=0.6]
 \coordinate (a) at (0,0);
 \coordinate (b) at (1,0);
 \coordinate (c) at (1.765,-0.5);
 \coordinate (d) at (2.53,0);
 \coordinate (e) at (1.765,0.5);
 \coordinate (ru) at (2.53,1);
 \coordinate (rd) at (2.53,-1);
 \coordinate (ab) at (0.5,0.765);
 \coordinate (lu) at (-0.765,0.5);
 \coordinate (ld) at (-0.765,-0.5);
 \draw[dotted] (a) -- (b);
 \draw[dotted] (b) -- (c);
 \draw[blue,ultra thick] (d) -- (c);
 \draw[dotted] (e) -- (d);
 \draw[blue,ultra thick] (e) -- (b);
 \draw[dotted] (c) -- (e);
 \draw[blue,ultra thick] (d) -- (ru);
 \draw[dotted] (ru) -- (e);
 \draw[dotted] (c) -- (rd);
 \draw[dotted] (rd) -- (d);
 \draw[dotted] (b) -- (ab);
 \draw[blue,ultra thick] (ab) -- (a);
 \draw[dotted] (a) -- (lu);
 \draw[dotted] (lu) -- (ld);
 \draw[blue,ultra thick] (ld) -- (a);
 \draw[dotted] (lu) .. controls (-0.6,1.2) and (0,1.4) .. (ab);
 \draw[dotted] (ab) .. controls (1,1.6) and (2,1.6) .. (ru);
 \draw[dotted] (ru) .. controls (3.5,0.7) and (3.5,-0.7) .. (rd);
 \draw[dotted] (rd) .. controls (1.7,-1.7) and (0,-1.5) .. (ld);
\end{tikzpicture}
   \hspace{3mm}
\begin{tikzpicture}[scale=0.6]
 \coordinate (a) at (0,0);
 \coordinate (b) at (1,0);
 \coordinate (c) at (1.765,-0.5);
 \coordinate (d) at (2.53,0);
 \coordinate (e) at (1.765,0.5);
 \coordinate (ru) at (2.53,1);
 \coordinate (rd) at (2.53,-1);
 \coordinate (ab) at (0.5,0.765);
 \coordinate (lu) at (-0.765,0.5);
 \coordinate (ld) at (-0.765,-0.5);
 \draw[blue,ultra thick] (a) -- (b);
 \draw[dotted] (b) -- (c);
 \draw[dotted] (d) -- (c);
 \draw[dotted] (e) -- (d);
 \draw[dotted] (e) -- (b);
 \draw[blue,ultra thick] (c) -- (e);
 \draw[dotted] (d) -- (ru);
 \draw[blue,ultra thick] (ru) -- (e);
 \draw[dotted] (c) -- (rd);
 \draw[blue,ultra thick] (rd) -- (d);
 \draw[dotted] (b) -- (ab);
 \draw[dotted] (ab) -- (a);
 \draw[blue,ultra thick] (a) -- (lu);
 \draw[dotted] (lu) -- (ld);
 \draw[dotted] (ld) -- (a);
 \draw[dotted] (lu) .. controls (-0.6,1.2) and (0,1.4) .. (ab);
 \draw[dotted] (ab) .. controls (1,1.6) and (2,1.6) .. (ru);
 \draw[dotted] (ru) .. controls (3.5,0.7) and (3.5,-0.7) .. (rd);
 \draw[dotted] (rd) .. controls (1.7,-1.7) and (0,-1.5) .. (ld);
\end{tikzpicture}
   \hspace{3mm}
\begin{tikzpicture}[scale=0.6]
 \coordinate (a) at (0,0);
 \coordinate (b) at (1,0);
 \coordinate (c) at (1.765,-0.5);
 \coordinate (d) at (2.53,0);
 \coordinate (e) at (1.765,0.5);
 \coordinate (ru) at (2.53,1);
 \coordinate (rd) at (2.53,-1);
 \coordinate (ab) at (0.5,0.765);
 \coordinate (lu) at (-0.765,0.5);
 \coordinate (ld) at (-0.765,-0.5);
 \draw[dotted] (a) -- (b);
 \draw[dotted] (b) -- (c);
 \draw[dotted] (d) -- (c);
 \draw[blue,ultra thick] (e) -- (d);
 \draw[blue,ultra thick] (e) -- (b);
 \draw[dotted] (c) -- (e);
 \draw[dotted] (d) -- (ru);
 \draw[dotted] (ru) -- (e);
 \draw[dotted] (c) -- (rd);
 \draw[blue,ultra thick] (rd) -- (d);
 \draw[dotted] (b) -- (ab);
 \draw[blue,ultra thick] (ab) -- (a);
 \draw[dotted] (a) -- (lu);
 \draw[dotted] (lu) -- (ld);
 \draw[blue,ultra thick] (ld) -- (a);
 \draw[dotted] (lu) .. controls (-0.6,1.2) and (0,1.4) .. (ab);
 \draw[dotted] (ab) .. controls (1,1.6) and (2,1.6) .. (ru);
 \draw[dotted] (ru) .. controls (3.5,0.7) and (3.5,-0.7) .. (rd);
 \draw[dotted] (rd) .. controls (1.7,-1.7) and (0,-1.5) .. (ld);
\end{tikzpicture}
   \hspace{3mm}
\begin{tikzpicture}[scale=0.6]
 \coordinate (a) at (0,0);
 \coordinate (b) at (1,0);
 \coordinate (c) at (1.765,-0.5);
 \coordinate (d) at (2.53,0);
 \coordinate (e) at (1.765,0.5);
 \coordinate (ru) at (2.53,1);
 \coordinate (rd) at (2.53,-1);
 \coordinate (ab) at (0.5,0.765);
 \coordinate (lu) at (-0.765,0.5);
 \coordinate (ld) at (-0.765,-0.5);
 \draw[dotted] (a) -- (b);
 \draw[dotted] (b) -- (c);
 \draw[dotted] (d) -- (c);
 \draw[dotted] (e) -- (d);
 \draw[dotted] (e) -- (b);
 \draw[blue,ultra thick] (c) -- (e);
 \draw[dotted] (d) -- (ru);
 \draw[blue,ultra thick] (ru) -- (e);
 \draw[dotted] (c) -- (rd);
 \draw[blue,ultra thick] (rd) -- (d);
 \draw[dotted] (b) -- (ab);
 \draw[blue,ultra thick] (ab) -- (a);
 \draw[dotted] (a) -- (lu);
 \draw[dotted] (lu) -- (ld);
 \draw[blue,ultra thick] (ld) -- (a);
 \draw[dotted] (lu) .. controls (-0.6,1.2) and (0,1.4) .. (ab);
 \draw[dotted] (ab) .. controls (1,1.6) and (2,1.6) .. (ru);
 \draw[dotted] (ru) .. controls (3.5,0.7) and (3.5,-0.7) .. (rd);
 \draw[dotted] (rd) .. controls (1.7,-1.7) and (0,-1.5) .. (ld);
\end{tikzpicture}
\]
 It is easy to check that these correspond bijectively with perfect matchings of angles in $T_{\de_2}$ given in Subsection \ref{excoefffree}(2).
\end{example}

\section{Essential loops}\label{essentialloops}
 Recall the definition of essential loops \cite{MSW2}. Throughout this section, we suppose that a marked surface $(S,M)$ has no punctures. An {\it essential loop} $\z$ in $(S,M)$ is a closed curve in $S$, considered up to isotopy, such that: $\z$ is disjoint from $M$ and the boundary of $S$; $\z$ does not intersect itself; $\z$ is not a contractible loop.

 Choose a triangle $\triangle$ of $T$ that $\z$ crosses. Let $p$ be a point in the interior of $\triangle$ that lies on $\z$. Let $\alpha$ and $\beta$ be the two sides of $\triangle$ crossed $\z$ immediately before and following its travel through $p$, and let $\tau$ be the third side of $\triangle$. Let $\wz$ be the curve whose starting and ending points are $p$ that exactly follows $\z$. We can construct the triangulated polygon $T_{\wz}$ associated with $\wz$ in the same way as for plain arcs. Also, we obtain the snake graph $G_{\wz}$ from $T_{\wz}$. Let $v$ (resp., $w$) be the endpoint of $\tau$ and $\alpha$ (resp., $\beta$) in the first triangle of $T_{\wz}$ or $G_{\wz}$, and let $v'$ (resp., $w'$) be the endpoint of $\tau$ and $\beta$ (resp., $\alpha$) in the last triangle of $T_{\wz}$ or $G_{\wz}$ (see Figure \ref{TwzGwz}).

\begin{figure}[h]
   \caption{$T_{\wz}$ and $G_{\wz}$ associated with an essential loop $\z$}
   \label{TwzGwz}
\[
\begin{tikzpicture}
 \coordinate (0) at (0,0);   \fill[blue] (0) circle (0.7mm) node[above]{\scriptsize $p$};
 \coordinate (l) at (-0.7,-0.5);
 \coordinate (u) at (0,0.8);
 \coordinate (r) at (0.7,-0.5);
 \coordinate (lu) at (-0.8,1);
 \coordinate (ru) at (0.8,1);
 \draw[blue] (0) .. controls (0.8,0) and (1,0.7) .. node[right]{\scriptsize $\z$} (ru);
 \draw[blue,dotted] (ru) .. controls (0.5,1.7) and (-0.5,1.7) .. (lu);
 \draw[blue] (0) .. controls (-0.8,0) and (-1,0.7) .. (lu);
 \draw (l) to node[fill=white,inner sep=1,pos=0.2]{\scriptsize $\alpha$} (u);
 \draw (u) to node[fill=white,inner sep=1,pos=0.8]{\scriptsize $\beta$} (r);
 \draw (l) to node[fill=white,inner sep=1]{\scriptsize $\tau$} (r);
\end{tikzpicture}
 \hspace{7mm} T_{\wz}
\begin{tikzpicture}
 \coordinate (l) at (0,0);   \fill (l) circle (0.7mm) node[left]{\scriptsize $v$};
 \coordinate (lu) at (0.8,1);
 \coordinate (ru) at (1.8,1);
 \coordinate (ld) at (0.8,-1);   \fill (ld) circle (0.7mm) node[left]{\scriptsize $w$};
 \coordinate (rd) at (1.8,-1);   \fill (rd) circle (0.7mm) node[right]{\scriptsize $w'$};
 \coordinate (r) at (2.6,0);   \fill (r) circle (0.7mm) node[right]{\scriptsize $v'$};
 \coordinate (pl) at (0.5,0);   \fill[blue] (pl) circle (0.7mm) node[above]{\scriptsize $p$};
 \coordinate (pr) at (2.1,0);   \fill[blue] (pr) circle (0.7mm) node[above]{\scriptsize $p$};
 \draw (l) to node[fill=white,inner sep=1]{\scriptsize $\alpha$} (lu);
 \draw (l) to node[fill=white,inner sep=1]{\scriptsize $\tau$} (ld);
 \draw (lu) to node[fill=white,inner sep=1,pos=0.7]{\scriptsize $\beta$} (ld);
 \draw (lu) to (ru);
 \draw (ld) to (rd);
 \draw (ru) to node[fill=white,inner sep=1]{\scriptsize $\beta$} (r);
 \draw (rd) to node[fill=white,inner sep=1]{\scriptsize $\tau$} (r);
 \draw (ru) to node[fill=white,inner sep=1,pos=0.7]{\scriptsize $\alpha$} (rd);
 \draw[blue] (pl) to node[fill=white,inner sep=1]{\scriptsize $\z$} (pr);
\end{tikzpicture}
 \hspace{7mm} G_{\wz} \hspace{1mm}
\begin{tikzpicture}
 \coordinate (lld) at (0,0);   \fill (lld) circle (0.7mm) node[below]{\scriptsize $v$};
 \coordinate (llu) at (0,0.8);
 \coordinate (lru) at (0.8,0.8);
 \coordinate (lrd) at (0.8,0);  
 \coordinate (rld) at (1.7,0.8);
 \coordinate (rlu) at (1.7,1.6);
 \coordinate (rru) at (2.5,1.6);   \fill (rru) circle (0.7mm) node[above]{\scriptsize $v'$};
 \coordinate (rrd) at (2.5,0.8);  
 \node at (3.3,0.8) {\mbox{or}};
 \node at (1.3,0.8) {\mbox{\rotatebox{45}{$\cdots$}}};
 \draw (lld) to node[fill=white,inner sep=1]{\scriptsize $\alpha$} (llu);
 \draw (lld) to node[fill=white,inner sep=1]{\scriptsize $\tau$} (lrd);
 \draw (llu) to (lru);
 \draw (lru) to (lrd);
 \draw[white,ultra thin] (llu) to node[black]{\scriptsize $\beta$} (lrd);
 \draw (rld) to (rlu);
 \draw (rld) to (rrd);
 \draw (rlu) to node[fill=white,inner sep=1]{\scriptsize $\beta$} (rru);
 \draw (rru) to node[fill=white,inner sep=1]{\scriptsize $\tau$} (rrd);
 \draw[white,ultra thin] (rlu) to node[black]{\scriptsize $\alpha$} (rrd);
 \fill (lrd) circle (0.7mm) node[below]{\scriptsize $w$}; \fill (rrd) circle (0.7mm) node[below]{\scriptsize $w'$};
\end{tikzpicture}
 \hspace{3mm}
\begin{tikzpicture}
 \coordinate (lld) at (0,0);   \fill (lld) circle (0.7mm) node[below]{\scriptsize $v$};
 \coordinate (llu) at (0,0.8);
 \coordinate (lru) at (0.8,0.8);
 \coordinate (lrd) at (0.8,0);
 \coordinate (rld) at (1.7,0.8);
 \coordinate (rlu) at (1.7,1.6);
 \coordinate (rru) at (2.5,1.6);   \fill (rru) circle (0.7mm) node[above]{\scriptsize $v'$};
 \coordinate (rrd) at (2.5,0.8);
 \node at (1.3,0.8) {\mbox{\rotatebox{45}{$\cdots$}}};
 \draw (lld) to node[fill=white,inner sep=1]{\scriptsize $\alpha$} (llu);
 \draw (lld) to node[fill=white,inner sep=1]{\scriptsize $\tau$} (lrd);
 \draw (llu) to (lru);
 \draw (lru) to (lrd);
 \draw[white,ultra thin] (llu) to node[black]{\scriptsize $\beta$} (lrd);
 \draw (rld) to (rlu);
 \draw (rld) to (rrd);
 \draw (rlu) to node[fill=white,inner sep=1]{\scriptsize $\tau$} (rru);
 \draw (rru) to node[fill=white,inner sep=1]{\scriptsize $\beta$} (rrd);
 \draw[white,ultra thin] (rlu) to node[black]{\scriptsize $\alpha$} (rrd);
 \fill (rlu) circle (0.7mm) node[above]{\scriptsize $w'$}; \fill (lrd) circle (0.7mm) node[below]{\scriptsize $w$};
\end{tikzpicture}
\]
\[
\begin{tikzpicture}
 \coordinate (0) at (0,0);   \fill[blue] (0) circle (0.7mm) node[above]{\scriptsize $p$};
 \coordinate (l) at (-0.7,0.8);
 \coordinate (u) at (0,-0.5);
 \coordinate (r) at (0.7,0.8);
 \coordinate (lu) at (-0.8,1);
 \coordinate (ru) at (0.8,1);
 \draw[blue] (0) .. controls (0.8,0) and (1,0.7) .. node[right]{\scriptsize $\z$} (ru);
 \draw[blue,dotted] (ru) .. controls (0.5,1.7) and (-0.5,1.7) .. (lu);
 \draw[blue] (0) .. controls (-0.8,0) and (-1,0.7) .. (lu);
 \draw (l) to node[fill=white,inner sep=1,pos=0.3]{\scriptsize $\alpha$} (u);
 \draw (u) to node[fill=white,inner sep=1,pos=0.7]{\scriptsize $\beta$} (r);
 \draw (l) to node[fill=white,inner sep=1]{\scriptsize $\tau$} (r);
\end{tikzpicture}
 \hspace{7mm} T_{\wz}
\begin{tikzpicture}
 \coordinate (l) at (0,0);   \fill (l) circle (0.7mm) node[left]{\scriptsize $v$};
 \coordinate (lu) at (0.8,1);   \fill (lu) circle (0.7mm) node[left]{\scriptsize $w$};
 \coordinate (ru) at (1.8,1);   \fill (ru) circle (0.7mm) node[right]{\scriptsize $w'$};
 \coordinate (ld) at (0.8,-1);
 \coordinate (rd) at (1.8,-1);
 \coordinate (r) at (2.6,0);   \fill (r) circle (0.7mm) node[right]{\scriptsize $v'$};
 \coordinate (pl) at (0.5,0);   \fill[blue] (pl) circle (0.7mm) node[above]{\scriptsize $p$};
 \coordinate (pr) at (2.1,0);   \fill[blue] (pr) circle (0.7mm) node[above]{\scriptsize $p$};
 \draw (l) to node[fill=white,inner sep=1]{\scriptsize $\tau$} (lu);
 \draw (l) to node[fill=white,inner sep=1]{\scriptsize $\alpha$} (ld);
 \draw (lu) to node[fill=white,inner sep=1,pos=0.7]{\scriptsize $\beta$} (ld);
 \draw (lu) to (ru);
 \draw (ld) to (rd);
 \draw (ru) to node[fill=white,inner sep=1]{\scriptsize $\tau$} (r);
 \draw (rd) to node[fill=white,inner sep=1]{\scriptsize $\beta$} (r);
 \draw (ru) to node[fill=white,inner sep=1,pos=0.7]{\scriptsize $\alpha$} (rd);
 \draw[blue] (pl) to node[fill=white,inner sep=1]{\scriptsize $\z$} (pr);
\end{tikzpicture}
 \hspace{7mm} G_{\wz} \hspace{1mm}
\begin{tikzpicture}
 \coordinate (lld) at (0,0);   \fill (lld) circle (0.7mm) node[below]{\scriptsize $v$};
 \coordinate (llu) at (0,0.8);
 \coordinate (lru) at (0.8,0.8);
 \coordinate (lrd) at (0.8,0);
 \coordinate (rld) at (1.7,0.8);
 \coordinate (rlu) at (1.7,1.6);
 \coordinate (rru) at (2.5,1.6);   \fill (rru) circle (0.7mm) node[above]{\scriptsize $v'$};
 \coordinate (rrd) at (2.5,0.8);
 \node at (3.3,0.8) {\mbox{or}};
 \node at (1.3,0.8) {\mbox{\rotatebox{45}{$\cdots$}}};
 \draw (lld) to node[fill=white,inner sep=1]{\scriptsize $\tau$} (llu);
 \draw (lld) to node[fill=white,inner sep=1]{\scriptsize $\alpha$} (lrd);
 \draw (llu) to (lru);
 \draw (lru) to (lrd);
 \draw[white,ultra thin] (llu) to node[black]{\scriptsize $\beta$} (lrd);
 \draw (rld) to (rlu);
 \draw (rld) to (rrd);
 \draw (rlu) to node[fill=white,inner sep=1]{\scriptsize $\tau$} (rru);
 \draw (rru) to node[fill=white,inner sep=1]{\scriptsize $\beta$} (rrd);
 \draw[white,ultra thin] (rlu) to node[black]{\scriptsize $\alpha$} (rrd);
 \fill (rlu) circle (0.7mm) node[above]{\scriptsize $w'$}; \fill (llu) circle (0.7mm) node[above]{\scriptsize $w$};
\end{tikzpicture}
 \hspace{3mm}
\begin{tikzpicture}
 \coordinate (lld) at (0,0);   \fill (lld) circle (0.7mm) node[below]{\scriptsize $v$};
 \coordinate (llu) at (0,0.8);
 \coordinate (lru) at (0.8,0.8);
 \coordinate (lrd) at (0.8,0);
 \coordinate (rld) at (1.7,0.8);
 \coordinate (rlu) at (1.7,1.6);
 \coordinate (rru) at (2.5,1.6);   \fill (rru) circle (0.7mm) node[above]{\scriptsize $v'$};
 \coordinate (rrd) at (2.5,0.8);
 \node at (1.3,0.8) {\mbox{\rotatebox{45}{$\cdots$}}};
 \draw (lld) to node[fill=white,inner sep=1]{\scriptsize $\tau$} (llu);
 \draw (lld) to node[fill=white,inner sep=1]{\scriptsize $\alpha$} (lrd);
 \draw (llu) to (lru);
 \draw (lru) to (lrd);
 \draw[white,ultra thin] (llu) to node[black]{\scriptsize $\beta$} (lrd); \fill (llu) circle (0.7mm) node[above]{\scriptsize $w$};
 \draw (rld) to (rlu);
 \draw (rld) to (rrd);
 \draw (rlu) to node[fill=white,inner sep=1]{\scriptsize $\beta$} (rru);
 \draw (rru) to node[fill=white,inner sep=1]{\scriptsize $\tau$} (rrd);
 \draw[white,ultra thin] (rlu) to node[black]{\scriptsize $\alpha$} (rrd); \fill (rrd) circle (0.7mm) node[below]{\scriptsize $w'$};
\end{tikzpicture}
\]
\end{figure}

\begin{definition}\cite[Definition 3.4, 3.8]{MSW2}
 The {\it band graph} $\wG_{\z}$ associated with the essential loop $\z$ is the graph obtained from $G_{\wz}$ by identifying the edges $\tau$ in the first and last squares such that $v$ corresponds to $v'$. That is, the band graph lies on an annulus or a M\"{o}bius strip. A perfect matching $P$ of $\wG_{\z}$ is called {\it good} either if $\tau \in P$ or if both edges incident to $v$ and incident to $w$ in $P$ lie on the same square. We denote by $\bP_g(\wG_{\z})$ the set of good perfect matchings of $\wG_{\z}$.
\end{definition}

 Viewing $P \in \bP_g(\wG_{\z})$ as a subset of $(G_{\wz})_1$, we can obtain $\oP \in \bP(G_{\wz})$ from $P$ by adding either the edge $\tau$ in the first square or in the last square in $\wG_{\z}$. Then it is easy to show that there is a bijection $\bP_g(\wG_{\z})$ and the set
\[
 \bP_g(G_{\wz}) := \{P \in \bP(G_{\wz}) \mid \mbox{$P$ contains $\tau$ in the first or the last triangle of $G_{\wz}$}\}
\]
given by sending $P$ to $\oP$. In particular, there is a unique good perfect matching $P_{-}(\wG_{\z})$ such that $\overline{P_{-}(\wG_{\z})} = P_{-}(G_{\wz})$, called the {\it minimal matching} (see \cite[Remark 3.9]{MSW2}).

\begin{definition}\cite[Definition 3.14]{MSW2}
 For an essential loop $\z$ in $(S,M)$, we define a Laurent polynomial
\[
 x_{\z}:=\frac{1}{\cross(T,\z)}\sum_{P \in \bP_g(\wG_{\z})}x(P)y(P).
\]
\end{definition}

 One reason to consider $x_{\z}$ is that they give rise to a base for the cluster algebra with principal coefficients obtained from a triangulated surface without punctures. Let $T$ be a triangulation of $(S,M)$. A collection of arcs and essential loops in $(S,M)$ is $\mathcal{C}^{\circ}${\it-compatible} if they do not intersect each other.

\begin{theorem}\cite[Theorem 1.1, 4.1]{MSW2}\label{base}
 Let $(S,M)$ be a marked surface without punctures and $T$ be a triangulation of $(S,M)$. Then the set
\[
 \Bigl\{ \prod_{c \in C}x_c \mid C \ \mbox{is a 
$\mathcal{C}^{\circ}$-compatible collection of} \ (S,M)\Bigr\}
\]
is a base of $\cA(T)$.
\end{theorem}

 In this case, we study perfect matchings of angles. For an essential loop $\z$ in $(S,M)$, we can construct a triangulated polygon $T_{\z}$ in the same way as for plain arcs, that is, it is a triangulated annulus (see Figure \ref{T closed loop}). In particular, it is not twisted unlike band graphs. Since $T_{\z}$ has the same numbers of triangles and of vertices, then $\bA(T_{\z})\neq\emptyset$.

\begin{figure}[h]
   \caption{Example of $T_{\z}$ for an essential loop $\z$}
   \label{T closed loop}
\[
 T \hspace{1mm}
\begin{tikzpicture}
 \coordinate (0) at (0,0.2);   \fill (0) circle (0.7mm);
 \coordinate (1) at (0,-0.2);   \fill (1) circle (0.7mm);
 \coordinate (u) at (0,1.3);   \fill (u) circle (0.7mm);
 \coordinate (r) at (1.3,0);   \fill (r) circle (0.7mm);
 \coordinate (d) at (0,-1.3);   \fill (d) circle (0.7mm);
 \coordinate (l) at (-1.3,0);   \fill (l) circle (0.7mm);
 \draw (0,0) circle (0.2);
 \draw (0,0) circle (1.3);
 \draw[blue] (0,0) circle (0.7) node[above=16,left=8,fill=white,inner sep=1]{\scriptsize $\z$};
 \draw (d) .. controls (-0.8,0) and (-0.5,0.5) .. node[fill=white,inner sep=1]{\scriptsize $1$} (0);
 \draw (d) .. controls (0.8,0) and (0.5,0.5) .. node[fill=white,inner sep=1]{\scriptsize $3$} (0);
 \draw (u) to node[fill=white,inner sep=1,pos=0.3]{\scriptsize $2$} (0);
 \draw (d) to node[fill=white,inner sep=1,pos=0.3]{\scriptsize $4$} (1);
 \draw (d) .. controls (-1.3,-0.7) and (-1.3,0.7) .. node[fill=white,inner sep=1]{\scriptsize $5$} (u);
 \draw (d) .. controls (1.3,-0.7) and (1.3,0.7) .. node[fill=white,inner sep=1]{\scriptsize $6$} (u);
\end{tikzpicture}
 \hspace{15mm} T_{\z} \hspace{1mm}
\begin{tikzpicture}
 \coordinate (0) at (0,0.2);
 \coordinate (1) at (0,-0.2);
 \coordinate (u) at (0,1.3);
 \coordinate (d) at (0,-1.3);
 \draw (0,0) circle (0.2);
 \draw (d) arc (-90:90:1.3) node[below=37,right=34,fill=white,inner sep=1]{\scriptsize $6$};
 \draw (d) arc [start angle = -90, delta angle = -180, radius = 1.3] node[below=37,left=34,fill=white,inner sep=1]{\scriptsize $5$};
 \draw (d) .. controls (-0.8,0) and (-0.5,0.5) .. node[fill=white,inner sep=1]{\scriptsize $1$} (0);
 \draw (d) .. controls (0.8,0) and (0.5,0.5) .. node[fill=white,inner sep=1]{\scriptsize $3$} (0);
 \draw (u) to node[fill=white,inner sep=1]{\scriptsize $2$} (0);
 \draw (d) to node[fill=white,inner sep=1]{\scriptsize $4$} (1);
\end{tikzpicture}
\]
\end{figure}
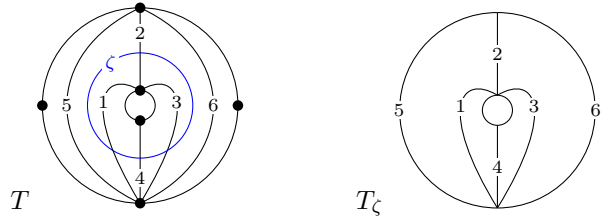

 We define {\it max-condition} as the dual min-condition.

\begin{definition}
 Let $\z$ be an essential loop in $(S,M)$. We say that a perfect matching of angles in $T_{\z}$ is {\it bad} if all angles incident to one boundary component satisfy min-condition and all angles incident to the other boundary component satisfy max-condition (see Figure \ref{badpm}). A non-bad perfect matching of angles in $T_{\z}$ is called {\it good}. We denote by $\bA_g(T_{\z})$ the set of good perfect matchings of angles in $T_{\z}$.
\end{definition}

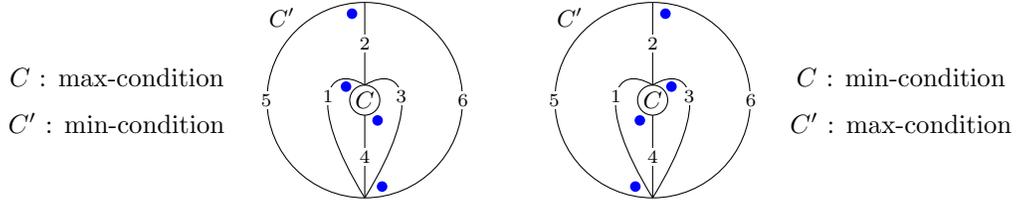
\begin{figure}[h]
   \caption{Bad perfect matchings of angles in the above $T_{\z}$ with boundary components $C$ and $C'$}
   \label{badpm}
\[
\begin{tikzpicture}
 \coordinate (0) at (0,0.2);
 \coordinate (1) at (0,-0.2);
 \coordinate (u) at (0,1.3);
 \coordinate (d) at (0,-1.3);
 \node at (0,0) {\small$C$};
 \node at (-1.1,1.1) {\small$C'$};
 \node at (-3.3,0.3) {$C$ : max-condition};
 \node at (-3.3,-0.3) {$C'$ : min-condition};
 \fill[blue] (-0.17,1.15) circle (0.7mm);
 \fill[blue] (-0.25,0.18) circle (0.7mm);
 \fill[blue] (0.23,-1.15) circle (0.7mm);
 \fill[blue] (0.17,-0.27) circle (0.7mm);
 \draw (0,0) circle (0.2);
 \draw (d) arc (-90:90:1.3) node[below=37,right=34,fill=white,inner sep=1]{\scriptsize $6$};
 \draw (d) arc [start angle = -90, delta angle = -180, radius = 1.3] node[below=37,left=34,fill=white,inner sep=1]{\scriptsize $5$};
 \draw (d) .. controls (-0.8,0) and (-0.5,0.5) .. node[fill=white,inner sep=1]{\scriptsize $1$} (0);
 \draw (d) .. controls (0.8,0) and (0.5,0.5) .. node[fill=white,inner sep=1]{\scriptsize $3$} (0);
 \draw (u) to node[fill=white,inner sep=1]{\scriptsize $2$} (0);
 \draw (d) to node[fill=white,inner sep=1]{\scriptsize $4$} (1);
\end{tikzpicture}
 \hspace{10mm}
\begin{tikzpicture}
 \coordinate (0) at (0,0.2);
 \coordinate (1) at (0,-0.2);
 \coordinate (u) at (0,1.3);
 \coordinate (d) at (0,-1.3);
 \node at (0,0) {\small$C$};
 \node at (-1.1,1.1) {\small$C'$};
 \node at (3.3,0.3) {$C$ : min-condition};
 \node at (3.3,-0.3) {$C'$ : max-condition};
 \fill[blue] (0.17,1.15) circle (0.7mm);
 \fill[blue] (0.25,0.18) circle (0.7mm);
 \fill[blue] (-0.23,-1.15) circle (0.7mm);
 \fill[blue] (-0.17,-0.27) circle (0.7mm);
 \draw (0,0) circle (0.2);
 \draw (d) arc (-90:90:1.3) node[below=37,right=34,fill=white,inner sep=1]{\scriptsize $6$};
 \draw (d) arc [start angle = -90, delta angle = -180, radius = 1.3] node[below=37,left=34,fill=white,inner sep=1]{\scriptsize $5$};
 \draw (d) .. controls (-0.8,0) and (-0.5,0.5) .. node[fill=white,inner sep=1]{\scriptsize $1$} (0);
 \draw (d) .. controls (0.8,0) and (0.5,0.5) .. node[fill=white,inner sep=1]{\scriptsize $3$} (0);
 \draw (u) to node[fill=white,inner sep=1]{\scriptsize $2$} (0);
 \draw (d) to node[fill=white,inner sep=1]{\scriptsize $4$} (1);
\end{tikzpicture}
\]
\end{figure}

 Then we have the following result.

\begin{theorem}\label{loopformula}
 Let $\z$ be an essential loop in $(S,M)$. There is a bijection $\psi_{\z} : \bP_g(\wG_{\z}) \rightarrow \bA_g(T_{\z})$ satisfying $x(P)=x(\psi_{\z}(P))$ and $y(P)=y(\psi_{\z}(P))$ for $P \in \bP_g(\wG_{\z})$. In particular, we have the equation
\[
 x_{\z}=\frac{1}{\cross(T,\z)}\sum_{A \in \bA_g(T_{\z})}x(A)y(A).
\]
\end{theorem}

 To prove Theorem \ref{loopformula}, we need some preparations. By rotational symmetry of order two, we can assume that $T_{\wz}$ is the above case in Figure \ref{TwzGwz}. Since there is a bijection between $\bP_g(\wG_{\z})$ and $\bP_g(G_{\wz})$, Theorem \ref{p.m.bij} induces a bijection between $\bP_g(\wG_{\z})$ and the set
\[
 \bA_g(T_{\wz}) := \{A \in \bA(T_{\wz}) \mid \mbox{$A$ contains $c$ or $c'$}\},
\]
where $c$ (resp., $c'$) is the angle between $\alpha$ and $\beta$ in the first (resp., last) triangle of $T_{\wz}$ (see Figure \ref{abcc'}). In particular, this bijection preserves the values of $x(-)$ and $y(-)$ by Theorem \ref{p.m.bij} and Proposition \ref{ybijgamma}. We denote by $c_A$ an angle $c$ or $c'$ contained in $A \in \bA_g(T_{\wz})$. If both $c$ and $c'$ are contained in $A$, we define $c_A=c$. We only need to construct a bijection $\psi_{\z}' : \bA_g(T_{\wz}) \rightarrow \bA_g(T_{\z})$ satisfying $x(A \setminus \{c_A\})=x(\psi_{\z}'(A))$ and $y(A \setminus \{c_A\})=y(\psi_{\z}'(A))$ for $A \in \bA_g(T_{\wz})$. Let $a$ (resp., $b$) be the angle between $\alpha$ (resp., $\beta$) and $\tau$ in the first triangle of $T_{\wz}$ (see Figure \ref{abcc'}). We denote by $\bA_g(T_{\wz})_{\reflectbox{$\in$}b}$ (resp., $\bA_g(T_{\wz})_{\reflectbox{$\notin$}b}$) the subset of elements in $\bA_g(T_{\wz})$ containing (resp., not containing) $b$, in particular, $\bA_g(T_{\wz})=\bA_g(T_{\wz})_{\reflectbox{$\in$}b}\sqcup\bA_g(T_{\wz})_{\reflectbox{$\notin$}b}$.

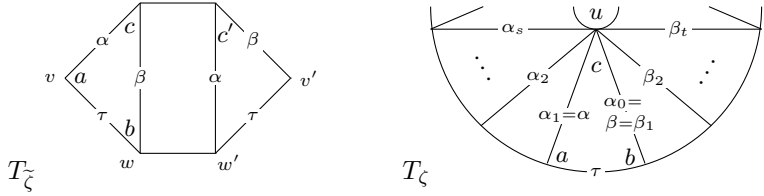
\begin{figure}[h]
   \caption{$T_{\wz}$ and $T_{\z}$ for an essential loop $\z$}
   \label{abcc'}
\[
 T_{\wz}
\begin{tikzpicture}
 \coordinate (l) at (0,0);   \fill (l) circle (0mm) node[left]{\scriptsize $v$};   \fill (l) circle (0mm) node[right]{\small $a$};
 \coordinate (lu) at (1,1);   \fill (lu) circle (0mm) node[left=4,below=4]{\small $c$};
 \coordinate (ru) at (2,1);   \fill (ru) circle (0mm) node[right=4,below=4]{\small $c'$};
 \coordinate (ld) at (1,-1);   \fill (ld) circle (0mm) node[left=5,below=-1]{\scriptsize $w$};   \fill (ld) circle (0mm) node[left=4,above=3]{\small $b$};
 \coordinate (rd) at (2,-1);   \fill (rd) circle (0mm) node[right=5,below=-3]{\scriptsize $w'$};   \fill (rd) circle (0mm) node[right=4.5,above=3]{\small $$};
 \coordinate (r) at (3,0);   \fill (r) circle (0mm) node[right]{\scriptsize $v'$};   \fill (r) circle (0mm) node[left=2]{\small $$};
 \draw (l) to node[fill=white,inner sep=1]{\scriptsize $\alpha$} (lu);
 \draw (l) to node[fill=white,inner sep=1]{\scriptsize $\tau$} (ld);
 \draw (lu) to node[fill=white,inner sep=1]{\scriptsize $\beta$} (ld);
 \draw (lu) to (ru);
 \draw (ld) to (rd);
 \draw (ru) to node[fill=white,inner sep=1]{\scriptsize $\beta$} (r);
 \draw (rd) to node[fill=white,inner sep=1]{\scriptsize $\tau$} (r);
 \draw (ru) to node[fill=white,inner sep=1]{\scriptsize $\alpha$} (rd);
\end{tikzpicture}
 \hspace{10mm} T_{\z}
\begin{tikzpicture}
 \coordinate (u) at (0,0);
 \coordinate (s) at (-0.3,0.3);
 \coordinate (l) at (-2.2,0.3);
 \coordinate (ld) at (0,-1.9);
 \node at (-1.5,-0.55) {\rotatebox{-60}{$\cdots$}};
 \node at (1.5,-0.55) {\rotatebox{60}{$\cdots$}};
 \node at (-0.45,-1.7) {\small $a$};
 \node at (0.45,-1.7) {\small $b$};
 \node at (0,-0.5) {\small $c$};
 \node at (0,0.2) {\small $u$};
 \draw (s) arc (-180:0:0.3);
 \draw (ld) node[fill=white,inner sep=1]{} arc (-90:0:2.2);
 \draw (l) arc (-180:-90:2.2) node[fill=white,inner sep=1]{\scriptsize $\tau$};
 \draw (u) -- node[fill=white,inner sep=1,pos=0.65]{\scriptsize $\alpha_1${\tiny $=$}$\alpha$} (-110:1.9);
 \draw (u) -- node[fill=white,inner sep=1,pos=0.55]{\scriptsize $\alpha_0${\tiny $=$}} node[fill=white,inner sep=1,pos=0.7]{\scriptsize $\beta${\tiny $=$}$\beta_1$} (-70:1.9);
 \draw (u) -- node[fill=white,inner sep=1]{\scriptsize $\alpha_2$} (-140:2);
 \draw (u) -- node[fill=white,inner sep=1]{\scriptsize $\beta_2$} (-40:2);
 \draw (u) -- node[fill=white,inner sep=1]{\scriptsize $\alpha_s$} (-180:2.2);
 \draw (u) -- node[fill=white,inner sep=1]{\scriptsize $\beta_t$} (0:2.2);
 \draw (-2.2,0) -- (-1.5,0.3);
 \draw (2.2,0) -- (1.5,0.3);
\end{tikzpicture}
\]
\end{figure}

 Let $A \in \bA_g(T_{\wz})_{\reflectbox{$\in$}b}$. Then $c' \in A$ follows from the definition of $\bA_g(T_{\wz})$, that is $c_A=c'$.  The triangulated annulus $T_{\z}$ is obtained from $T_{\wz}$ by removing the last triangle in $T_{\wz}$ and by identifying the edges $\alpha$ in the first triangle and in the last triangle. It is easy to show that this construction induces a natural map $\psi_{\z}^{\reflectbox{$\in$}b} : \bA_g(T_{\wz})_{\reflectbox{$\in$}b} \rightarrow \bA(T_{\z})$. Abusing notation, let $a$ (resp., $b$, $c$) be the angle between $\tau$ and $\alpha$ (resp., $\tau$ and $\beta$, $\alpha$ and $\beta$) in $T_{\z}$. We denote by $u$ the common endpoint of $\alpha$ and $\beta$ in $T_{\z}$. Let $\alpha_s,\ldots,\alpha_1=\alpha,\alpha_0=\beta=\beta_1,\ldots,\beta_t$ be all arcs incident to $u$ winding counter-clockwisely around $u$ (see Figure \ref{abcc'}).

\begin{lemma}\label{bin}
 For $A \in \bA_g(T_{\wz})_{\reflectbox{$\in$}b}$, then $\psi_{\z}^{\reflectbox{$\in$}b}(A) \in \bA_g(T_{\z})$. Moreover, the map $\psi_{\z}^{\reflectbox{$\in$}b}$ induces a bijection between $\bA_g(T_{\wz})_{\reflectbox{$\in$}b}$ and the set
\[
 \bA_g(T_{\z})_{\reflectbox{$\in$}b} := \{A' \in \bA_g(T_{\z}) \mid b \in A'\}.
\]
\end{lemma}
\begin{proof}
 Since $c' \in A$, then $\psi_{\z}^{\reflectbox{$\in$}b}(A)$ does not contain the angle between $\alpha_s$ and a boundary segment incident to $u$. Thus the angle incident to $u$ in $\psi_{\z}^{\reflectbox{$\in$}b}(A)$ does not satisfy min-condition. Since $b \in \psi_{\z}^{\reflectbox{$\in$}b}(A)$ satisfies max-condition, $\psi_{\z}^{\reflectbox{$\in$}b}(A)$ is good.

 By construction, there is a bijection between $\bA_g(T_{\wz})_{\reflectbox{$\in$}b}$ and the set
\begin{equation}\label{setalpha}
 \{A' \in \bA_g(T_{\z})_{\reflectbox{$\in$}b} \mid \mbox{$A'$ contains the angle between $\beta_i$ and $\beta_{i+1}$ for some $i \in [1,t-1]$}\}.
\end{equation}
 Let $A' \in \bA_g(T_{\z})_{\reflectbox{$\in$}b}$. If $A'$ contains the angle between $\alpha_j$ and $\alpha_{j+1}$ for some $j \in [1,s-1]$, $A'$ must contain the angle between $\alpha_j$ and the boundary segment of the triangle with sides $\alpha_j$ and $\alpha_{j-1}$. Continuing this process, $A'$ contains $a$, and it contradicts $b \in A'$. If $A'$ contains the angle between $\alpha_s$ and a boundary segment incident to $u$, then $A'$ must contain the angle between $\alpha_i$ and the boundary segment of the triangle with sides $\alpha_i$ and $\alpha_{i+1}$ for $[1,s-1]$. Then by the same argument for the other endpoint $u' \neq u$ of $\alpha_s$, the angle of $A'$ incident to $u'$ satisfies max-condition. Continuing this process, $A'$ consists only of exterior angles whose angles incident to the boundary with $u$ satisfy min-condition and angles incident to the boundary with $\tau$ satisfy max-condition. It contradicts that $A'$ is good. Therefore any $A' \in \bA_g(T_{\z})_{\reflectbox{$\in$}b}$ satisfies the condition of (\ref{setalpha}), thus the set (\ref{setalpha}) and $\bA_g(T_{\z})_{\reflectbox{$\in$}b}$ coincide.
\end{proof}

 Let $A \in \bA_g(T_{\wz})_{\reflectbox{$\notin$}b}$. Then $c \in A$ follows from the definition of perfect matchings of angles, that is $c_A=c$.  The triangulated annulus $T_{\z}$ is obtained from $T_{\wz}$ by removing the first triangle in $T_{\wz}$ and by identifying the edges $\beta$ in the first triangle and in the last triangle. In particular, $c'$ in $T_{\wz}$ corresponds to $c$ in $T_{\z}$. It is easy to show that this construction induces a natural map $\psi_{\z}^{\reflectbox{$\notin$}b} : \bA_g(T_{\wz})_{\reflectbox{$\notin$}b} \rightarrow \bA(T_{\z})$.

\begin{lemma}
 For $A \in \bA_g(T_{\wz})_{\reflectbox{$\notin$}b}$, then $\psi_{\z}^{\reflectbox{$\notin$}b}(A) \in \bA_g(T_{\z})$. In particular, the map $\psi_{\z}^{\reflectbox{$\notin$}b}$ induces a bijection between $\bA_g(T_{\wz})_{\reflectbox{$\notin$}b}$ and the set
\[
\bA_g(T_{\z})_{\reflectbox{$\notin$}b} := \{A' \in \bA_g(T_{\z}) \mid b \notin A'\}.
\]
\end{lemma}

\begin{proof}
 Since $c \in A$, then $\psi_{\z}^{\reflectbox{$\notin$}b}(A)$ does not contain the angle between $\beta_t$ and a boundary segment incident to $u$. Thus $\psi_{\z}^{\reflectbox{$\notin$}b}(A)$ is good since $b \notin \psi_{\z}^{\reflectbox{$\notin$}b}(A)$.

 By construction, there is a bijection between $\bA_g(T_{\wz})_{\reflectbox{$\notin$}b}$ and the set
\begin{equation}\label{setbeta}
 \{A' \in \bA_g(T_{\z})_{\reflectbox{$\notin$}b} \mid \mbox{$A'$ does not contain the angles between $\beta_i$ and $\beta_{i+1}$ for all $i \in [1,t-1]$}\}.
\end{equation}
  Let $A' \in \bA_g(T_{\z})_{\reflectbox{$\notin$}b}$. If $c \in A'$, it satisfies the condition of (\ref{setbeta}). Suppose that $a \in A'$. Then, in the same way as the proof of Lemma \ref{bin}, $A'$ satisfies the condition of (\ref{setbeta}). Therefore, the set (\ref{setbeta}) and $\bA_g(T_{\z})_{\reflectbox{$\notin$}b}$ coincide. Thus the assertion holds.
\end{proof}

\begin{proof}[Proof of Theorem \ref{loopformula}]
  We have decompositions $\bA_g(T_{\wz})=\bA_g(T_{\wz})_{\reflectbox{$\in$}b}\sqcup\bA_g(T_{\wz})_{\reflectbox{$\notin$}b}$ and $\bA_g(T_{\z})=\bA_g(T_{\z})_{\reflectbox{$\in$}b}\sqcup\bA_g(T_{\z})_{\reflectbox{$\notin$}b}$. We define the map $\psi_{\z}' : \bA_g(T_{\wz}) \rightarrow \bA_g(T_{\z})$ by
\[
   \psi_{\z}'(A) = \left\{
     \begin{array}{ll}
 \psi_{\z}^{\reflectbox{$\in$}b}(A) & \mbox{if $A \in \bA_g(T_{\wz})_{\reflectbox{$\in$}b}$},\\
 \psi_{\z}^{\reflectbox{$\notin$}b}(A) & \mbox{if $A \in \bA_g(T_{\wz})_{\reflectbox{$\notin$}b}$.}
     \end{array} \right.
\]
 By Lemmas \ref{setalpha} and \ref{setbeta}, $\psi_{\z}'$ is bijective. It satisfies $x(A \setminus \{c_A\})=x(\psi_{\z}'(A))$ and $y(A \setminus \{c_A\})=y(\psi_{\z}'(A))$ for $A \in \bA_g(T_{\wz})$ since $\psi_{\z}^{\reflectbox{$\in$}b}$ and $\psi_{\z}^{\reflectbox{$\notin$}b}$ are natural maps. Therefore, we have a bijection
\[
 \psi_{\z} :
\SelectTips{eu}{}
\xymatrix@R=0.1mm{
 \bP_g(\wG_{\z}) \ar[r] & \bP_g(G_{\wz}) \ar[r] & \bA_g(T_{\wz}) \ar[r] & \bA_g(T_{\z})\\
 \rotatebox{90}{$\in$} & \rotatebox{90}{$\in$} &  \rotatebox{90}{$\in$} & \rotatebox{90}{$\in$}\\
 P \ar@{|->}[r] & \oP \ar@{|->}[r] & \varphi_{\wz}(\oP) \ar@{|->}[r] & \psi_{\z}'\phi_{\wz}(\oP),
}
\]
 where $\varphi_{\wz}$ is the bijection between $\bP_g(G_{\wz})$ and $\bA_g(T_{\wz})$ induced by Theorem \ref{p.m.bij}, satisfying $x(P)=x(\psi_{\z}(P))$ and $y(P)=y(\psi_{\z}(P))$.
\end{proof}

\medskip\noindent{\bf Acknowledgements}.
The author is a Research Fellow of Society for the Promotion of Science (JSPS). This work was supported by JSPS KAKENHI Grant Number JP17J04270.\par
The author would like to thank Laurent Demonet and his supervisor Osamu Iyama for their guidance and helpful advice.



\begin{thebibliography}{Y}

 \bibitem[BFPPT]{BFPPT} M. Barot, E. Fern\'{a}ndez, M. I. Platzeck, N. I. Pratti and S. Trepode, {\it From iterated tilted algebras to cluster-tilted algebras}, Adv. Math. Vol. 223 (2010) 1468--1494.

 \bibitem[CPr]{CPr} G. Carroll and G. Price, {\it Two new combinatorial models for the Ptolemy recurrence}, unpublished memo (2003).

 \bibitem[CPi]{CPi} C. Ceballos and V. Pilaud, {\it Cluster algebras of type $D$: pseudotriangulations approach}, Electron. J. Combin., 22(4) (2015) Paper 4.44.

 \bibitem[DL]{DL} L. Demonet and X. Luo, {\it Ice quivers with potential associated with triangulations and Cohen-Macaulay modules over orders}, Trans. Amer. Math. Soc. 368 (2016) 4257--4293.

 \bibitem[DWZ]{DWZ} H. Derksen, J. Weyman and A. Zelevinsky, {\it Quivers with potentials and their representations I: Mutations}, Selecta Math. Vol. 14 (2008) 59--119. 

 \bibitem[FoG1]{FoG1} V. Fock and A. Goncharov, {\it Moduli spaces of local systems and higher Teichm\"{u}ller theory}, Publ. Math. Inst. Hautes \'{E}tudes Sci. No. 103 (2006) 1--211.

 \bibitem[FoG2]{FoG2} V. Fock and A. Goncharov, {\it Cluster ensembles, quantization and the dilogarithm}, Ann. Sci. Ec. Norm. Super. (4) 42, no. 6 (2009) 865--930.

 \bibitem[FST]{FST} S. Fomin, M. Shapiro and D. Thurston, {\it Cluster algebras and triangulated surfaces Part I: Cluster complexes}, Acta Math. Vol. 201 (2008) 83--146.

 \bibitem[FT]{FT} S. Fomin and D. Thurston, {\it Cluster algebras and triangulated surfaces. Part II: Lambda lengths}, Memoirs AMS, 255(1223) 2018.

 \bibitem[FZ1]{FZ1} S. Fomin and A. Zelevinsky, {\it Cluster algebras I: Foundations}, J. Amer. Math. Soc. 15 (2002) 497--529.

 \bibitem[FZ2]{FZ2} S. Fomin and A. Zelevinsky, {\it Cluster algebras IV: Coefficients}, Compos. Math. 143 (2007) 112--164.

 \bibitem[FuG]{FuG} S. Fujiwara and Y. Gyoda, {\it Duality between front and rear mutations in cluster algebras}, arXiv:1808.02156.

 \bibitem[GSV]{GSV} M. Gekhtman, M. Shapiro and A. Vainshtein, {\it Cluster algebras and Weil-Petersson forms},
Duke Math. J. 127 (2005) 291--311.

 \bibitem[HI]{HI} M. Herschend and O. Iyama, {\it Selfinjective quivers with potential and $2$-representation-finite algebras}, Compos. Math. 147 (2011) 1885--1920.

 \bibitem[MS]{MS} G. Musiker and R. Schiffler, {\it Cluster expansion formulas and perfect matchings}, Alg. Comb. Vol. 32 (2010), 187--209.

 \bibitem[MSW1]{MSW1} G. Musiker, R. Schiffler and L. Williams, {\it Positivity for cluster algebras from surfaces}, Adv. Math. Vol. 227 (2011) 2241--2308.

 \bibitem[MSW2]{MSW2} G. Musiker, R. Schiffler and L. Williams, {\it Bases for cluster algebras from surfaces}, Compos. Math. 149(2) (2013) 217--263.

 \bibitem[P]{P} J. Propp, {\it The combinatorics of frieze patterns and markoff numbers}, arXiv:math/0511633.

 \bibitem[QZ]{QZ} Y. Qiu and Y. Zhou, {\it Cluster categories for marked surfaces: punctured case}, Compos. Math. 153, no. 9 (2017) 1779--1819.

 \bibitem[Y]{Y} T. Yurikusa, {\it Cluster expansion formulas in type $A$}, Algebr Represent Theor (2017) https://doi.org/10.1007/s10468-017-9755-3.
\end{thebibliography}
\end{document}